\input ./style/arxiv-general.cfg
\documentclass[aop,MSNbibl,seceqn,secthm,dvips]{arximspdf}
\makeatletter
   \@ifpackageloaded{graphicx}{}{\usepackage{graphicx}}
\makeatother
\usepackage{mathrsfs}

\doi{10.1214/15-AOP1013}% Updated by VTEXPTS2LaTeX.exe, 25.03.2015
\volume{44}
\issue{3}
\pubyear{2016}
\firstpage{1916}
\lastpage{1955}
\docsubty{FLA}

\makeatletter
\newcommand{\eqref}[1]{(\ref{#1})}
\newtheorem{lemma}[thm]{Lemma}
\newtheorem{prop}[thm]{Proposition}
\newtheorem{cor}[thm]{Corollary}
\newproclaim{defin}[thm]{Definition}
\newproclaim{rem}[thm]{Remark}
\mathchardef\varGamma="0100
\makeatother

\begin{document}
\begin{frontmatter}

\title{Degenerate parabolic stochastic partial differential equations:
Quasilinear case}
\runtitle{Degenerate parabolic SPDE's: Quasilinear case}

\begin{aug}
% Corresponding author: Martina Hofmanov\'{a} - martina.hofmanova@mis.mpg.de% Updated by VTEXPTS2LaTeX.exe, 26.03.2015 07:21
%martina.hofmanova@mis.mpg.de% Updated by VTEXPTS2LaTeX.exe, 25.03.2015
%15:02
\author[A]{\fnms{Arnaud}~\snm{Debussche}\ead[label=e1]{arnaud.debussche@bretagne.ens-cachan.fr}},
\author[B]{\fnms{Martina}~\snm{Hofmanov\'a}\ead[label=e2]{martina.hofmanova@mis.mpg.de}\corref{}}
\and
\author[C]{\fnms{Julien}~\snm{Vovelle}\ead[label=e3]{vovelle@math.univ-lyon1.fr}}
\runauthor{A. Debussche, M. Hofmanov\'a and J. Vovelle}
\affiliation{IRMAR CNRS et ENS Rennes, Max Planck Institute for Mathematics\\  in the Sciences and Universit\'e Lyon 1}
%\dedicated{}
\address[A]{A. Debussche\\
IRMAR CNRS et ENS Rennes\\
Campus de Ker Lann\\
F-35170 Bruz\\
France\\
\printead{e1}}
\address[B]{M. Hofmanov\'a\\
Max Planck Institute for Mathematics\\
\quad in the Sciences\\
Inselstr. 22\\
04103 Leipzig\\
Germany\\
\printead{e2}}
\address[C]{J. Vovelle\\
Universit\'e de Lyon\\
CNRS, Universit\'e Lyon 1\\
Institut de Camille Jordan\\
43 boulevard du 11 novembre 1918\\
F-69622 Villeurbanne Cedex\\
France\\
\printead{e3}}
\end{aug}

% HISTORY:
%
\received{\smonth{8} \syear{2013}}% Updated by VTEXPTS2LaTeX.exe,
%25.03.2015 15:02
%
\revised{\smonth{11} \syear{2014}}% Updated by VTEXPTS2LaTeX.exe,
%25.03.2015 15:02

% ABSTRACT
%
\begin{abstract}
In this paper, we study the Cauchy problem for a quasilinear degenerate
parabolic stochastic partial differential equation driven by a
cylindrical Wiener process. In particular, we adapt the notion of
kinetic formulation and kinetic solution and develop a well-posedness
theory that includes also an $L^1$-contraction property. In comparison
to the previous works of the authors concerning stochastic hyperbolic
conservation laws [\textit{J. Funct. Anal.} \textbf{259}
(2010) 1014--1042] and semilinear
degenerate parabolic SPDEs [\textit{Stochastic Process. Appl.}
\textbf{123} (2013) 4294--4336], the present result
contains two new ingredients that provide simpler and more effective
method of the proof: a generalized It\^{o} formula that permits a
rigorous derivation of the kinetic formulation even in the case of weak
solutions of certain nondegenerate approximations and a direct proof of
strong convergence of these approximations to the desired kinetic
solution of the degenerate problem.
\end{abstract}

% KEYWORDS
\end{frontmatter}

%s1 #&#
\section{Introduction}

We consider the Cauchy problem for a quasilinear degenerate parabolic
stochastic partial differential equation
%
%e1.1 #&#
\begin{eqnarray}
\qquad\quad\mathrm{d}u+\operatorname{div} \bigl(B(u) \bigr)\,\mathrm {d}t&=&
\operatorname{div} \bigl(A(u)\nabla u \bigr)\,\mathrm{d}t+\Phi(u)\,\mathrm{d}W,
\qquad x\in\mathbb{T}^N, t\in(0,T),
\nonumber
\\[-8pt]
\label{eq}
\\[-8pt]
\nonumber
u(0)&=& u_0,
\end{eqnarray}
where $W$ is a cylindrical Wiener process. Equations of this type model
the phenomenon of convection-diffusion of ideal fluids and, therefore,
arise in a wide variety of important applications, including, for
instance, two or three phase flows in porous media or
sedimentation--consolidation processes (for a thorough exposition of
this area given from a practical point of view we refer the reader to
\cite{petrol} and the references therein). The addition of a
stochastic noise to this physical model is fully natural
as it represents external perturbations or a lack of knowledge of
certain physical parameters. Toward the applicability of the results,
it is necessary to treat the problem \eqref{eq} under very general
hypotheses. Particularly, without the assumption of positive
definiteness of the diffusion matrix $A$, the equation can be
degenerate which brings the main difficulty in the problem solving. We
assume the matrix $A$ to be positive semidefinite and, as a
consequence, it can, for instance, vanish completely which leads to a
hyperbolic conservation law. We point out that we do not intend to
employ any form of regularization by the noise to solve \eqref{eq},
and thus the deterministic equation is included in our theory as well.

In order to find a suitable concept of solution for our model problem
\eqref{eq}, we observe that already in the case of deterministic
hyperbolic conservation law it is possible to find simple examples
supporting the two following claims (see, e.g., \cite{malek}):
\begin{longlist}[(ii)]
\item[(i)] classical $C^1$ solutions do not exist,
\item[(ii)] weak (distributional) solutions lack uniqueness.
\end{longlist}
The first claim is a consequence of the fact that any smooth solution
has to be constant along characteristic lines, which can intersect in
finite time (even in the case of smooth data) and shocks can be
produced. The second claim demonstrates the inconvenience that often
appears in the study of PDEs and SPDEs: the usual way of weakening the
equation leads to the occurrence of nonphysical solutions and,
therefore, additional assumptions need to be imposed in order to select
the physically relevant ones and to ensure uniqueness. Hence, one needs
to find some balance that allows to establish existence of a unique
(physically reasonable) solution.

Toward this end, we adapt the notion of kinetic formulation and kinetic
solution. This concept was first introduced by Lions, Perthame, Tadmor
\cite{lions} for deterministic hyperbolic conservation laws. In
comparison to the notion of entropy solution introduced by Kru\v{z}kov
\cite{kruzk}, kinetic solutions seem to be better suited particularly
for degenerate parabolic problems since they allow us to keep the
precise structure of the parabolic dissipative measure, whereas in the
case of entropy solution part of this information is lost and has to be
recovered at some stage. This technique also supplies a good technical
framework to establish a well-posedness theory which is the main goal
of the present paper.

Other references for kinetic or entropy solutions in the case of
deterministic hyperbolic conservation laws include, for instance, \cite{car,vov,lpt1,pert1,perth}.
Deterministic degenerate parabolic PDEs were studied by Carrillo \cite
{car} and Chen and Perthame \cite{chen} by means of both entropy and
kinetic solutions. Also in the stochastic setting there are several
papers concerned with entropy solutions for hyperbolic conservation
laws, the first one being \cite{kim} then \cite{bauzet,feng,wittbold}. The first work dealing with kinetic solutions
in the stochastic setting was given by Debussche and Vovelle \cite
{debus}. Their concept was then further generalized to the case of
semilinear degenerate parabolic SPDEs by Hofmanov\'a \cite{hof}. To
the best of our knowledge, stochastic equations of type \eqref{eq}
have not been studied in this generality yet, neither by means of
kinetic formulation nor by any other approach. Recently, Bauzet, Vallet
and Wittbold \cite{bauzet1} considered entropy solutions for
degenerate parabolic--hyperbolic SPDEs under different assumptions on
the data and the nonlinearities and under stronger assumptions on the
noise. There is also a different kind of stochastic conservation laws:
equations with a stochastic forcing not in the source term but in the
flux term. Such equations, in the first-order case, have been studied
recently by Lions, Perthame and Souganidis, \cite
{LionsPerthameSouganidis12,LionsPerthameSouganidis13}.

In comparison to the previous works of the authors \cite{debus} and
\cite{hof}, the present proof of well-posedness contains two new
ingredients: a generalized It\^o formula that permits a rigorous
derivation of the kinetic formulation even in the case of weak
solutions of certain nondegenerate approximations (see \hyperref[secito]{Appendix}) and a direct proof of strong convergence of these
approximations to the desired kinetic solution of the degenerate
problem (see Section~\ref{subsecstrong}). In order to explain
these recent developments more precisely, let us recall the basic ideas
of the proofs in \cite{debus} and \cite{hof}.

In the case of hyperbolic conservation laws \cite{debus}, the authors
defined a notion of generalized kinetic solution and obtained a
comparison result showing that any generalized kinetic solution is
actually a kinetic solution. Accordingly,
the proof of existence simplified since only weak convergence of
approximate viscous solutions was
necessary. The situation was quite different in the case of semilinear
degenerate parabolic equations \cite{hof}, since this approach was no
longer applicable.
The proof of the comparison principle was much more delicate and,
consequently, generalized kinetic solutions were not allowed and, therefore,
strong convergence of approximate solutions was needed in order to
prove existence. The limit argument was based on a compactness method:
uniform estimates yielded tightness and consequently also strong
convergence of the approximate sequence on another probability space
and the existence of a martingale kinetic solution followed. The
existence of a pathwise kinetic solution was then obtained by the Gy\"{o}ngy--Krylov characterization of convergence in probability.

Due to the second-order term in \eqref{eq}, we are for the moment not
able to apply efficiently the method of generalized kinetic solutions.
Let us explain why by considering the Definition~\ref{kinsol} of
solution. We may adapt this definition to introduce a notion of
generalized kinetic solution (in the spirit of \cite{debus}, e.g.), and we would then easily obtain the equivalent of the kinetic
equation \eqref{eqkinformul} by passing to the limit on suitable
approximate problems. This works well in the first-order case, provided
uniqueness of generalized solutions can be shown. To prove such a
result here, with second-order terms, we need the second important item
in Definition~\ref{kinsol}, the chain-rule \eqref{eqchainrule}. We
do not know how to relax this equality and we do not know how to obtain
it by mere weak convergence of approximations: strong convergence seems
to be necessary. Therefore, it would not bring any simplification here
to consider generalized
solutions. On the other hand, it would be possible to apply the
compactness method as established in \cite{hof} to obtain strong
convergence. However, as this is quite technical, we propose a simpler
proof of the strong convergence based on the techniques developed in
the proof of the comparison principle: comparing two (suitable)
nondegenerate approximations, we obtain the strong convergence in $L^1$
directly. Note that this approach does not apply to the semilinear case
as no sufficient control of the second-order term is known.

Another important issue here was the question of regularity of the
approximate solutions. In both works \cite{debus} and \cite{hof}, the
authors derived the kinetic formulation for sufficiently regular
approximations only. This obstacle was overcome by showing the
existence of these regular approximations in \cite{hof2}, however, it
does not apply to the quasilinear case where a suitable regularity
result is still missing: even in the deterministic setting the proofs,
which can be found in \cite{lady}, are very difficult and technical
while the stochastic case remains open. In the present paper, we
propose a different way to solve this problem, namely, the generalized
It\^o formula (Proposition~\ref{propito}) that leads to a clear-cut
derivation of the kinetic formulation also for weak solutions, and
hence avoids the necessity of regular approximations.

The paper is organized as follows. In Section~\ref{secnotation}, we
introduce the basic setting, define the notion of kinetic solution and
state our main result, Theorem~\ref{thmmain}. Section~\ref{seccomparison} is devoted to the proof of uniqueness together with
the $L^1$-comparison principle, Theorem~\ref{uniqueness}. The
remainder of the paper deals with the existence part of Theorem~\ref
{thmmain} which is divided into four parts. First, we prove existence
under three additional hypotheses: we consider \eqref{eq} with regular
initial data, positive definite diffusion matrix $A$ and Lipschitz
continuous flux function $B$, Section~\ref{secnondeg}. Second, we
relax the hypothesis upon $B$ and prove existence under the remaining
two additional hypotheses in Section~\ref{secnondeg1}. In Section~\ref{secdeg}, we proceed to the proof of existence in the degenerate
case while keeping the assumption upon the initial condition. The proof
of Theorem~\ref{thmmain} is then completed in Section~\ref{secgeneral}. In \hyperref[secito]{Appendix}, we establish the above
mentioned generalized It\^o formula for weak solutions of a general
class of SPDEs.

%s2 #&#
\section{Hypotheses and the main result}
\label{secnotation}

%s2.1 #&#
\subsection{Hypotheses}

We now give the precise assumptions on each of the terms appearing in
the above equation \eqref{eq}. We work on a finite-time interval
$[0,T], T>0$, and consider periodic boundary conditions: $x\in\mathbb{T}^N$
where $\mathbb{T}^N$ is the $N$-dimensional torus.
The flux function
\[
B=(B_1,\dots,B_N)\dvtx \mathbb{R}\longrightarrow
\mathbb{R}^N
\]
is supposed to be of class $C^2$ with a polynomial growth of its
derivative, which is denoted by $b=(b_1,\dots,b_N)$.
The diffusion matrix
\[
A=(A_{ij})_{i,j=1}^N\dvtx \mathbb{R}\longrightarrow
\mathbb{R}^{N\times N}
\]
is symmetric and positive semidefinite. Its square-root matrix, which
is also symmetric and positive semidefinite, is denoted by $\sigma$.
We assume that $\sigma$ is bounded and locally $\gamma$-H\"older
continuous for some $\gamma>1/2$, that is,
%
%e2.1 #&#
\begin{equation}
\label{sigma}
\bigl|\sigma(\xi)-\sigma(\zeta)\bigr|\leq C|\xi-\zeta|^\gamma\qquad
\forall\xi,\zeta\in\mathbb{R}, |\xi-\zeta|< 1.
\end{equation}

Regarding the stochastic term, let $(\Omega,\mathscr{F},(\mathscr
{F}_t)_{t\geq
0},\mathbb{P})$ be a stochastic basis with a complete, right-continuous
filtration. Let $\mathcal{P}$ denote the predictable $\sigma$-algebra
on $\Omega\times[0,T]$ associated to $(\mathscr{F}_t)_{t\geq0}$. The
initial datum may be random in general, that is, $\mathscr{F}_0$-measurable,
and we assume $u_0\in L^p(\Omega;L^p(\mathbb{T}^N))$ for all $p\in
[1,\infty
)$. The process $W$ is a cylindrical Wiener process: $W(t)=\sum_{k\geq
1}\beta_k(t) e_k$ with $(\beta_k)_{k\geq1}$ being mutually
independent real-valued standard Wiener processes relative to
$(\mathscr{F}
_t)_{t\geq0}$ and $(e_k)_{k\geq1}$ a complete orthonormal system in a
separable Hilbert space~$\mathfrak{U}$. In this setting, we can assume
without loss of generality that the $\sigma$-algebra $\mathscr{F}$ is
countably generated and $(\mathscr{F}_t)_{t\geq0}$ is the filtration
generated by the Wiener process and the initial condition.
For each $z\in L^2(\mathbb{T}^N)$, we consider a mapping $ \Phi
(z)\dvtx \mathfrak{U}\rightarrow L^2(\mathbb{T}^N)$ defined by $\Phi
(z)e_k=g_k(\cdot,z(\cdot))$. In particular, we suppose that $g_k\in
C(\mathbb{T}^N\times\mathbb{R})$ and the following conditions:
%
%e2.2 #&#
%e2.3 #&#
\begin{eqnarray}
\label{linrust}
G^2(x,\xi) &=& \sum_{k\geq1}\bigl|g_k(x,
\xi)\bigr|^2\leq C \bigl(1+|\xi|^2 \bigr),
\\
\label{skorolip}
\sum_{k\geq1}\bigl|g_k(x,
\xi)-g_k(y,\zeta)\bigr|^2 &\leq &  C \bigl(|x-y|^2+|\xi
-\zeta|h\bigl(|\xi-\zeta|\bigr) \bigr),
\end{eqnarray}
are fulfilled for every $x,y\in\mathbb{T}^N, \xi,\zeta\in\mathbb
{R}$, where
$h$ is a continuous nondecreasing function on $\mathbb{R}_+$
satisfying, for
some $\alpha>0$,
%
%e2.4 #&#
\begin{equation}
\label{fceh}
h(\delta)\leq C\delta^\alpha,\qquad\delta<1.
\end{equation}
The conditions imposed on $\Phi$, particularly assumption \eqref
{linrust}, imply that
\[
\Phi\dvtx L^2\bigl(\mathbb{T}^N\bigr)\longrightarrow
L_2\bigl(\mathfrak{U};L^2\bigl(\mathbb{T}^N
\bigr)\bigr),
\]
where $L_2(\mathfrak{U};L^2(\mathbb{T}^N))$ denotes the collection of
Hilbert--Schmidt operators from $\mathfrak{U}$ to $L^2(\mathbb
{T}^N)$. Thus,
given a predictable process $u\in L^2(\Omega;L^2(0,T;\break L^2(\mathbb{T}^N)))$,
the stochastic integral $t\mapsto\int_0^t\Phi(u)\,\mathrm{d}W$ is
a well-defined process taking values in $L^2(\mathbb{T}^N)$ (see \cite
{daprato} for
detailed construction).

Finally, we define the auxiliary space $\mathfrak{U}_0\supset
\mathfrak{U}$ via
\[
\mathfrak{U}_0= \biggl\{v=\sum_{k\geq1}
\alpha_k e_k; \sum_{k\geq
1}
\frac{\alpha_k^2}{k^2}<\infty \biggr\},
\]
endowed with the norm
\[
\|v\|^2_{\mathfrak{U}_0}=\sum_{k\geq1}
\frac{\alpha
_k^2}{k^2},\qquad v=\sum_{k\geq1}
\alpha_k e_k.
\]
Note that the embedding $\mathfrak{U}\hookrightarrow\mathfrak{U}_0$
is Hilbert--Schmidt. Moreover, trajectories of $W$ are $\mathbb
{P}$-a.s. in
$C([0,T];\mathfrak{U}_0)$ (see \cite{daprato}).

In this paper, we use the brackets $\langle\cdot,\cdot\rangle$ to
denote the duality between the space of distributions over $\mathbb{T}
^N\times\mathbb{R}$ and $C_c^\infty(\mathbb{T}^N\times\mathbb
{R})$ and the duality
between $L^p(\mathbb{T}^N\times\mathbb{R})$ and $L^q(\mathbb
{T}^N\times\mathbb{R})$. If there
is no danger of confusion, the same brackets will also denote the
duality between $L^p(\mathbb{T}^N)$ and $L^q(\mathbb{T}^N)$.
%$$\langle F,G\rangle=\int_{\mt^N}\int_\mr F(x,\xi) G(x,\xi)\,\dif x\,
%\dif\xi,\qquad F\in L^p(\mt^N\times\mr),\;G\in L^q(\mt^N\times\mr),$$
%where $p,q\in[1,\infty]$ are conjugate exponents.{\color{red} and also
%duality in $L^2(\mt^N)$}
The differential operators of gradient $\nabla$, divergence
$\operatorname{div}$
and Laplacian $\Delta$ are always understood with respect to the space
variable $x$.

%s2.2 #&#
\subsection{Definitions}

As the next step, we introduce the kinetic formulation of~\eqref{eq}
as well as the basic definitions concerning the notion of kinetic
solution. The motivation for this approach is given by the nonexistence
of a strong solution and, on the other hand, the nonuniqueness of weak
solutions, even in simple cases. The idea is to establish an additional
criterion---the kinetic formulation---which is automatically
satisfied by any weak solution to \eqref{eq} in the nondegenerate case
and which permits to ensure the well-posedness.

%de2.1 #&#
\begin{defin}[(Kinetic measure)]\label{mees}
A mapping $m$ from $\Omega$ to $\mathcal{M}_b^+([0,T]\times\mathbb{T}
^M\times\mathbb{R})$, the set of nonnegative bounded measures over
$[0,T]\times\mathbb{T}^N\times\mathbb{R}$, is said to be a kinetic
measure provided:
\begin{longlist}[(iii)]
\item[(i)] $m$ is measurable in the following sense: for each $\psi\in
C_0([0,T]\times\mathbb{T}^N\times\mathbb{R})$ the mapping $m(\psi
)\dvtx \Omega
\rightarrow\mathbb{R}$ is measurable,
\item[(ii)] $m$ vanishes for large $\xi$: if $B_R^c=\{\xi\in\mathbb{R};
|\xi
|\geq R\}$ then
\[
\label{infinity}
\lim_{R\rightarrow\infty}\mathbb{E} m \bigl([0,T]\times\mathbb
{T}^N\times B_R^c \bigr)=0,
\]
\item[(iii)] for any $\psi\in C_0(\mathbb{T}^N\times\mathbb{R})$
\[
\int_{\mathbb{T}^N\times[0,t]\times\mathbb{R}}\psi(x,\xi)\, \mathrm{d}m(s,x,\xi)\in
L^2\bigl(\Omega\times[0,T]\bigr)
\]
admits a predictable representative.\footnote{Throughout the paper, the
term \textit{representative} stands for an element of a class of equivalence.}
\end{longlist}
\end{defin}

%de2.2 #&#
\begin{defin}[(Kinetic solution)]\label{kinsol}
Assume that, for all $p\in[1,\infty)$,
\[
u\in L^p\bigl(\Omega\times[0,T],\mathcal{P},\mathrm{d}\mathbb{P}
\otimes \mathrm{d} t;L^p\bigl(\mathbb{T}^N\bigr)\bigr)\cap
L^p\bigl(\Omega;L^\infty\bigl(0,T;L^p\bigl(
\mathbb{T}^N\bigr)\bigr)\bigr)
\]
is such that:
\begin{longlist}[(ii)]
%\item there exists $C_p>0$ such that
%\begin{equation}\label{integrov}
%\stred\esssup_{0\leq t\leq T}\|u(t)\|^p_{L^p(\mt^N)}\leq C_p,
%\end{equation}
%
\item[(i)] $\operatorname{div}\int_0^u\sigma(\zeta)\,\mathrm{d}\zeta
\in L^2(\Omega
\times[0,T]\times\mathbb{T}^N)$,
\item[(ii)] for any $\phi\in C_b(\mathbb{R})$ the following chain rule formula
holds true:
%
%e2.5 #&#
\begin{equation}
\label{eqchainrule}
\quad\operatorname{div}\int^u_0
\phi(\zeta)\sigma(\zeta)\,\mathrm {d}\zeta=\phi (u)\operatorname{div}\int
^u_0\sigma(\zeta)\,\mathrm{d}\zeta\qquad
\mbox{in } \mathcal{D'}\bigl(\mathbb{T}^N\bigr)
\mbox{ a.e. }(\omega,t).
\end{equation}
\end{longlist}
Let $n_1\dvtx \Omega\rightarrow\mathcal{M}^+_b([0,T]\times\mathbb
{T}^M\times
\mathbb{R})$ be defined as follows: for any $\varphi\in C_0([0,T]
\times\mathbb{T}
^N\times\mathbb{R})$
\[
n_1(\varphi)=\int_0^T\!\!\int
_{\mathbb{T}^N}\!\int_{\mathbb{R}}\varphi (t,x,\xi ) \biggl|
\operatorname{div}\int_0^u\sigma(\zeta)\,
\mathrm{d}\zeta \biggr|^2\,\mathrm{d} \delta_{u(t,x)}(\xi)\,\mathrm{d}x
\,\mathrm{d}t.
\]
Then $u$ is said to be a kinetic solution to \eqref{eq} with initial
datum $u_0$ provided there exists a kinetic measure $m\geq n_1$,
$\mathbb{P}$-a.s., such that the pair $(f=\mathbf{1}_{u>\xi},m)$ satisfies,
%\begin{equation}\label{eqkinformul}
%\begin{split}
%\dif f+b\cdot\nabla f\,\dif t-A:\totdif^2 f\,\dif t&=\delta_{u=\xi}\,
%\Phi\,\dif W+\partial_\xi\bigg(m-\frac{1}{2}G^2\delta_{u=\xi}\bigg)
%\,\dif t,\\
%f(0)&=f_0:=\ind_{u_0>\xi},
%\end{split}
%\end{equation}
%in the sense of $\mathcal{D}'(\mt^N\times\mr)$.
for all $\varphi\in C_c^\infty([0,T)\times\mathbb{T}^N\times
\mathbb{R})$,
$\mathbb{P}\mbox{-a.s.}$,
%nuo cia
%e2.6 #&#
\begin{eqnarray}
 %
%\begin{split}
&& \int_0^T
\bigl\langle f(t),\partial_t\varphi(t) \bigr\rangle\,\mathrm{d} t+
\bigl\langle f_0,\varphi(0) \bigr\rangle+\int_0^T
\bigl\langle f(t),b\cdot\nabla\varphi(t) \bigr\rangle\,\mathrm{d}t
\nonumber\\
&&\quad {}+\int_0^T \bigl\langle f(t),A\dvtx
\mathrm{D}^2 \varphi (t) \bigr\rangle\,\mathrm{d}t
\nonumber
\\[-8pt]
\label{eqkinformul}
\\[-8pt]
\nonumber
&&\qquad =-\sum_{k\geq1}\int_0^T\!\!
\int_{\mathbb{T}^N}g_k \bigl(x,u(t,x) \bigr)\varphi
\bigl(t,x,u(t,x) \bigr)\,\mathrm{d}x\,\mathrm{d}\beta_k(t)
\\
\nonumber
&&\qquad\quad {}-\frac{1}{2}\int_0^T\!\!\int
_{\mathbb{T}^N}G^2 \bigl(x,u(t,x) \bigr)
\,\partial_\xi\varphi \bigl(t,x,u(t,x) \bigr)\,\mathrm{d}x\,\mathrm{d}
t+m(\partial_\xi\varphi). %\end{split}
\end{eqnarray}
We have used the notation $A\dvtx  B=\sum_{i,j}a_{ij}b_{ij}$ for
two matrices $A=(a_{ij})$, $B=(b_{ij})$ of
the same size.
\end{defin}

%re2.3 #&#
\begin{rem}
We emphasize that a kinetic solution is, in fact, a class of equivalence in $L^p(\Omega\times[0,T],\mathcal{P},\mathrm{d}\mathbb
{P}\otimes
\mathrm{d}t;L^p(\mathbb{T}^N))$ so not necessarily a stochastic
process in the
usual sense. Nevertheless, it will be seen later (see Corollary~\ref
{cont}) that, in this class of equivalence, there exists a
representative with good continuity properties, namely, $u\in
C([0,T];L^p(\mathbb{T}^N))$, $\mathbb{P}$-a.s. and, therefore, it can
be regarded
as a stochastic process.
%In general, the term {\em representative} will be used throughout the
%paper to describe an element of an equivalence class of functions.
%Therefore, the comparison principle \eqref{goo} is to be understood as
%an expression for suitable representatives (cf. Theorem
%\ref{uniqueness}).
\end{rem}

By $f=\mathbf{1}_{u>\xi}$ we understand a real function of four variables,
where the additional variable $\xi$ is called velocity.
In the deterministic case, that is, corresponding to the situation
$\Phi=0$, the equation \eqref{eqkinformul} in the above
definition is a weak form of the so-called kinetic formulation of
\eqref{eq}
\[
\partial_t \mathbf{1}_{u>\xi}+b\cdot\nabla
\mathbf{1}_{u>\xi
}-A\dvtx \mathrm{D}^2 \mathbf{1}_{u>\xi}=
\partial_\xi m,
\]
where the unknown is the pair $(\mathbf{1}_{u>\xi},m)$ and it is
solved in
the sense of distributions over $[0,T)\times\mathbb{T}^N\times
\mathbb{R}$.
In the stochastic case, we write formally\footnote{Hereafter, we
employ the notation which is commonly used in papers concerning the
kinetic solutions to conservation laws and write $\delta_{u=\xi}$ for
the Dirac measure centered at $u(t,x)$.}
%
%e2.7 #&#
\begin{equation}
\label{kinetic}
\quad\partial_t \mathbf{1}_{u>\xi}+b\cdot\nabla
\mathbf{1}_{u>\xi
}-A\dvtx \mathrm{D} ^2\mathbf{1}_{u>\xi}=
\delta_{u=\xi} \Phi(u)\dot{W}+\partial _\xi \bigl(m-
\tfrac{1}{2}G^2\delta_{u=\xi} \bigr).\hspace*{-9pt}
\end{equation}
It will be seen later that this choice is reasonable since for any $u$
being a weak solution to \eqref{eq} that belongs to $L^p(\Omega
;C([0,T];L^p(\mathbb{T}^N)))\cap L^2(\Omega;L^2(0,T;\break H^1(\mathbb
{T}^N)))$, $\forall
p\in[2,\infty)$, the pair $(\mathbf{1}_{u>\xi},n_1)$ satisfies
\eqref
{eqkinformul}, and consequently $u$ is a kinetic solution to \eqref
{eq}. The measure $n_1$ relates to the diffusion term in \eqref{eq}
and so is called parabolic dissipative measure.
% the measure $n_2=m-n_1$ which takes account of possible singularities
%of solution vanishes in the nondegenerate case.
% It gives us better regularity of solutions in the nondegeneracy zones
%of the diffusion matrix $A$ which is exactly what one would expect
%according to the theory of (nondegenerate) parabolic SPDEs. Indeed,
%for the case of a nondegenerate diffusion matrix $A$, i.e. when the
%second order term defines a strongly elliptic differential operator,
%the kinetic solution $u$ belongs to $L^2(\Omega;L^2(0,T;H^1(\mt^N)))$
%(cf. Definition~\ref{kinsol}(ii)). Thus, the measure $n_2=m-n_1$ which
%takes account of possible singularities of solution vanishes in the
%nondegenerate case.

We proceed with two related definitions.

%de2.4 #&#
\begin{defin}[(Young measure)]
Let $(X,\lambda)$ be a finite measure space. A~mapping $\nu$ from $X$
to the set of probability measures on $\mathbb{R}$ is said to be a Young
measure if, for all $\psi\in C_b(\mathbb{R})$, the map $z\mapsto\nu
_z(\psi
)$ from $X$ into $\mathbb{R}$ is measurable. We say that a Young
measure $\nu
$ vanishes at infinity if, for all $p\geq1$,
\[
\int_X\!\int_\mathbb{R}|
\xi|^p\,\mathrm{d}\nu_z(\xi)\,\mathrm {d}\lambda(z)<\infty.
\]
\end{defin}

%de2.5 #&#
\begin{defin}[(Kinetic function)]
Let $(X,\lambda)$ be a finite measure space. A~measurable function
$f\dvtx  X\times\mathbb{R}\rightarrow[0,1]$ is said to be a kinetic
function if
there exists a Young measure $\nu$ on $X$ vanishing at infinity such
that, for $\lambda$-a.e. $z\in X$, for all $\xi\in\mathbb{R}$,
\[
f(z,\xi)=\nu_z(\xi,\infty).
\]
\end{defin}

%re2.6 #&#
\begin{rem}
Note, that if $f$ is a kinetic function then $\partial_\xi f=-\nu$
for $\lambda$-a.e. $z\in X$. Similarly, let $u$ be a kinetic solution
of \eqref{eq} and consider $f=\mathbf{1}_{u >\xi}$. We have $
\partial
_\xi f=-\delta_{u=\xi}$, where $\nu=\delta_{u=\xi}$ is a Young
measure on $\Omega\times[0,T]\times\mathbb{T}^N$.
Therefore, \eqref{eqkinformul} can be rewritten as follows: for all
$\varphi\in C_c^\infty([0,T)\times\mathbb{T}^N\times\mathbb{R})$,
$\mathbb{P}$-a.s.,
%
%\[%\label{general}
\begin{eqnarray*}
&& \int_0^T  \bigl\langle f(t),
\partial_t\varphi(t) \bigr\rangle\,\mathrm{d} t+ \bigl\langle
f_0,\varphi(0) \bigr\rangle+\int_0^T
\bigl\langle f(t),b\cdot\nabla\varphi(t) \bigr\rangle\,\mathrm{d}t
\\
&&\quad {}+\int_0^T \bigl\langle f(t),A\dvtx
\mathrm{D}^2\varphi (t) \bigr\rangle\,\mathrm{d}t
\\
&&\qquad=-\sum_{k\geq1}\int_0^T
\!\!\int_{\mathbb{T}^N}\!\int_\mathbb {R}g_k(x,
\xi)\varphi (t,x,\xi)\,\mathrm{d}\nu_{t,x}(\xi)\,\mathrm{d}x\,\mathrm{d}
\beta _k(t)
\\
&&\qquad\quad {}-\frac{1}{2}\int_0^T\!\!\int
_{\mathbb{T}^N}\!\int_\mathbb {R}G^2(x,\xi
)\,\partial_\xi\varphi(t,x,\xi)\,\mathrm{d}\nu_{t,x}(\xi)\,
\mathrm {d}x\,\mathrm{d} t+m(\partial_\xi\varphi).
\end{eqnarray*}
%
%\]
%
For a general kinetic function $f$ with corresponding Young measure
$\nu$, the above formulation leads to the notion of generalized
kinetic solution as introduced in \cite{debus}. Although this concept
is not established here, the notation will be used throughout the
paper, that is, we will often write $\nu_{t,x}(\xi)$ instead of
$\delta_{u(t,x)=\xi}$.
\end{rem}

%s2.3 #&#
\subsection{Derivation of the kinetic formulation}
\label{subsecformulation}

Let us now clarify that the kinetic formulation \eqref{eqkinformul}
represents a reasonable way to weaken the original model problem \eqref
{eq}. In particular, we show that if $u$ is a weak solution to \eqref
{eq} such that $u\in L^p(\Omega;C([0,T];L^p(\mathbb{T}^N)))\cap
L^2(\Omega
;L^2(0,T;H^1(\mathbb{T}^N)))$, $\forall p\in[2,\infty)$, then
$f=\mathbf{1}
_{u>\xi}$ satisfies
\[
\mathrm{d}f+b\cdot\nabla f\,\mathrm{d}t-A\dvtx  \mathrm{D}^2 f\,\mathrm
{d}t=\delta_{u=\xi} \Phi\,\mathrm{d}W+\partial_\xi
\bigl(n_1-\tfrac{1}{2}G^2\delta _{u=\xi
}
\bigr)\,\mathrm{d}t
\]
in the sense of $\mathcal{D}'(\mathbb{T}^N\times\mathbb{R})$, where
\[
\mathrm{d}n_1(t,x,\xi)=\bigl|\sigma(u)\nabla u\bigr|^2\,\mathrm{d}
\delta _{u=\xi}\,\mathrm{d} x\,\mathrm{d}t.
\]
Indeed, it follows from Proposition~\ref{propito}, for $\varphi\in
C^2(\mathbb{R})$, $\psi\in C^1(\mathbb{T}^N)$,
\begin{eqnarray*}
\bigl\langle\varphi\bigl(u(t)\bigr),\psi \bigr\rangle&=&  \bigl\langle\varphi
(u_0),\psi \bigr\rangle-\int_0^t
\bigl\langle\varphi '(u)\operatorname{div} \bigl(B(u) \bigr),\psi
\bigr\rangle\,\mathrm{d}s
\\
&&{}-\int_0^t \bigl\langle\varphi''(u)
\nabla u\cdot \bigl(A(u)\nabla u \bigr),\psi \bigr\rangle\,\mathrm{d}s
\\
&&{}+\int_0^t \bigl\langle
\operatorname{div} \bigl(\varphi'(u) A(u)\nabla u \bigr),\psi \bigr
\rangle\,\mathrm{d}s
\\
&&{}+\sum_{k\geq1}\int_0^t
\bigl\langle\varphi'(u)g_k(u),\psi \bigr\rangle\,
\mathrm{d}\beta_k(s)
\\
&&{}+\frac{1}{2}\int_0^t \bigl\langle
\varphi''\bigl(u(s)\bigr)G^2(u),\psi \bigr
\rangle\,\mathrm{d}s.
\end{eqnarray*}
Afterward, we proceed term by term and employ the chain rule for
functions from Sobolev spaces. We obtain the following equalities that
hold true in $\mathcal{D}'(\mathbb{T}^N)$:
\begin{eqnarray*}
\varphi'(u)\operatorname{div} \bigl(B(u) \bigr) &=&
\varphi'(u)b(u)\cdot \nabla u
\\
&=& \operatorname{div} \biggl(\int_{-\infty}^{u}b(
\xi )\varphi '(\xi)\,\mathrm{d}\xi \biggr)=\operatorname{div}\bigl
\langle b\mathbf {1}_{u>\xi},\varphi'\bigr\rangle
_\xi,
\\
\varphi''(u)\nabla u\cdot \bigl(A(u)\nabla u
\bigr)&=& - \bigl\langle \partial_\xi n_1,
\varphi' \bigr\rangle_\xi,
\\
\operatorname{div} \bigl(\varphi'(u)A(u)\nabla u \bigr)&=& \mathrm
{D}^2\dvtx \biggl(\int_{-\infty}^u A(\xi)
\varphi'(\xi)\,\mathrm{d}\xi \biggr)=\mathrm {D}^2\dvtx \bigl
\langle A\mathbf{1}_{u>\xi},\varphi'\bigr
\rangle_\xi, %
\\
%\begin{split}
%&\varphi'(u)\diver\big(A(u)\nabla u\big)=\sum_{i,j=1}^N\partial_{x_i}
%\big[A_{ij}(x)\varphi'(u)\partial_{x_j}u\big]\\
%&\hspace{4.5cm}\qquad-\sum_{i,j=1}^N\varphi''(u)
%\partial_{x_i}uA_{ij}(x)\partial_{x_j}u\\
%&\hspace{3.2cm}=\sum_{i,j=1}^N\partial_{x_i}\bigg(A_{ij}(x)
%\partial_{x_j}\int_{-\infty}^{u}\varphi'(\xi)\dif\xi\bigg)+\big\langle
%\partial_\xi n_1(x,t),\varphi'\big\rangle_\xi\\
%&\hspace{3.2cm}=\diver\Big(A(x)\nabla\langle\ind_{u>\xi},\varphi'
%\rangle_\xi\Big)+\big\langle\partial_\xi n_1,\theta'\big\rangle_\xi
%\end{split}\\
%
\varphi'(u)g_k(u)&=& \bigl\langle g_k
\delta_{u=\xi},\varphi'\bigr\rangle_\xi,
\\
\varphi''(u)G^2(u) &=& \bigl
\langle G^2\delta_{u=\xi},\varphi''
\bigr\rangle_\xi =- \bigl\langle\partial_\xi
\bigl(G^2\delta_{u=\xi}\bigr),\varphi' \bigr
\rangle_\xi.
\end{eqnarray*}
Moreover,
\[
\bigl\langle\varphi\bigl(u(t)\bigr),\psi \bigr\rangle= \bigl\langle\mathbf{1}
_{u(t)>\xi},\varphi'\psi \bigr\rangle_{x,\xi}
\]
hence setting $\varphi(\xi)=\int_{-\infty}^\xi\phi(\zeta)\,
\mathrm{d}
\zeta$ for some $\phi\in C_c^\infty(\mathbb{R})$ yields the claim.

%s2.4 #&#
\subsection{The main result}

To conclude this section, we state our main result.

%th2.7 #&#
\begin{thm}\label{thmmain}
Let $u_0\in L^p(\Omega;L^p(\mathbb{T}^N))$, for all $p\in[1,\infty)$.
Under the above assumptions, there exists a unique kinetic solution to
\eqref{eq} and it has almost surely continuous trajectories in
$L^p(\mathbb{T}^N)$, for all $p\in[1,\infty)$. Moreover, if $u_1,
u_2$ are
kinetic solutions to \eqref{eq} with initial data $u_{1,0}$ and
$u_{2,0}$, respectively, then for all $t\in[0,T]$
\[
%\label{goo}
\mathbb{E}\bigl\|u_1(t)-u_2(t)\bigr\|_{L^1(\mathbb{T}^N)}
\leq\mathbb{E}\| u_{1,0}-u_{2,0}\| _{L^1(\mathbb{T}^N)}.
\]
\end{thm}

%s3 #&#
\section{Comparison principle}
\label{seccomparison}

Let us start with the question of uniqueness. As the first step, we
follow the approach of \cite{debus} and \cite{hof} and obtain an
auxiliary property of kinetic solutions, which will be useful later on
in the proof of the comparison principle in\vspace*{-2pt} Theorem~\ref{uniqueness}.

%pr3.1 #&#
\begin{prop}[(Left- and right-continuous representatives)]\label{limits}
Let $u$ be a kinetic solution to \eqref{eq}. Then $f=\mathbf
{1}_{u>\xi}$
admits representatives $f^-$ and $f^+$ which are almost surely left-
and right-continuous, respectively, at all points $t^*\in[0,T]$ in the
sense of distributions over $\mathbb{T}^N\times\mathbb{R}$. More
precisely, for
all $t^*\in[0,T]$ there exist kinetic functions $f^{*,\pm}$ on $
\Omega\times\mathbb{T}^N\times\mathbb{R}$ such that setting $f^\pm
(t^*)=f^{*,\pm}$ yields $f^\pm=f$ almost everywhere\vspace*{-2pt} and
\[
\bigl\langle f^\pm\bigl(t^*\pm \varepsilon\bigr),\psi \bigr\rangle
\longrightarrow \bigl\langle f^{\pm}\bigl(t^*\bigr),\psi \bigr\rangle,\qquad
\varepsilon\downarrow 0,  \forall\psi\in C^2_c\bigl(\mathbb
{T}^N\times\mathbb{R} \bigr), \mathbb{P}\mbox{-a.s.}
\]
Moreover, $f^{+}=f^{-}$ for all $t^*\in[0,T]$ except for some at most
countable\vspace*{-2pt} set.
\end{prop}

\begin{pf}
A detailed proof of this result can be found in \cite{hof}, Proposition~3.1.
\end{pf}\vspace*{-3pt}

From now on, we will work with these two fixed representatives of $f$
and we can take any of them in an
integral with respect to time or in a stochastic integral.

As the next step toward the proof of uniqueness, we need a technical
proposition relating two kinetic solutions of \eqref{eq}.
We will also use the following notation: if $f\dvtx  X\times\mathbb
{R}\rightarrow
[0,1]$ is a kinetic function, we denote by $\bar{f}$ the conjugate
function $\bar{f}=1-f$.\vspace*{-2pt}

%pr3.2 #&#
\begin{prop}[(Doubling of variables)]\label{propdoubling}
Let $u_1,u_2$ be kinetic solutions to~\eqref{eq} and\vspace*{1.5pt} denote
$f_1=\mathbf{1}
_{u_1>\xi}, f_2=\mathbf{1}_{u_2>\xi}$. Then for all $t\in[0,T]$
and any
nonnegative functions $\varrho\in C^\infty(\mathbb{T}^N)$, $\psi
\in
C^\infty_c(\mathbb{R})$ we have
\begin{eqnarray*}
%\label{doubling}
&& \mathbb{E} \int_{(\mathbb{T}^N)^2}\!\int
_{\mathbb{R}^2}\varrho (x-y)\psi(\xi-\zeta )f_1^{\pm}(x,t,
\xi)\bar{f}_2^{\pm}(y,t,\zeta)\,\mathrm{d}\xi\, \mathrm{d}
\zeta\,\mathrm{d}x\,\mathrm{d}y
\\[-2pt]
&&\qquad \leq\mathbb{E}\int_{(\mathbb{T}^N)^2}\!\int_{\mathbb{R}^2}
\varrho (x-y)\psi(\xi -\zeta)f_{1,0}(x,\xi)\bar{f}_{2,0}(y,
\zeta)\,\mathrm{d}\xi\, \mathrm{d}\zeta \,\mathrm{d}x\,\mathrm{d}y\\[-2pt]
&&\qquad\quad{}+\mathrm{I}+
\mathrm{J}+\mathrm{K},
\end{eqnarray*}
where\vspace*{-3pt}
\begin{eqnarray*}
\mathrm{I} &=&  \mathbb{E}\int_0^t\!\int
_{(\mathbb{T}^N)^2}\!\int_{\mathbb{R}^2}f_1\bar{f}
_2 \bigl(b(\xi)-b(\zeta) \bigr)\cdotp\nabla_x
\varrho(x-y)\psi(\xi -\zeta)\,\mathrm{d}\xi\, \mathrm{d}\zeta\,\mathrm{d}x\,
\mathrm{d}y\,\mathrm{d}s,
\\[-2pt]
\mathrm{J}&=& \mathbb{E}\int_0^t\!\int
_{(\mathbb{T}^N)^2}\!\int_{\mathbb{R}^2}f_1\bar{f}
_2 \bigl(A(\xi)+A(\zeta) \bigr)\dvtx  \mathrm{D}^2_x
\varrho(x-y)\psi(\xi -\zeta)\, \mathrm{d}\xi\,\mathrm{d}\zeta\,\mathrm{d}x\,
\mathrm{d}y\, \mathrm{d}s
\\[-2pt]
&&{}-\mathbb{E}\int_0^t\!\int
_{(\mathbb{T}^N)^2}\!\int_{\mathbb
{R}^2}\varrho (x-y)\psi(\xi-
\zeta)\, \mathrm{d}\nu^ { 1 } _ { x , s } (\xi)\,\mathrm{d}x\,
\mathrm{d}n_{2,1}(y,s,\zeta)
\\[-2pt]
&&{}-\mathbb{E}\int_0^t\!\int
_{(\mathbb{T}^N)^2}\!\int_{\mathbb
{R}^2}\varrho (x-y)\psi(\xi-
\zeta)\, \mathrm{d}\nu^{2}_{y,s} (\zeta)\,\mathrm{d}y\,
\mathrm{d}n_{1,1}(x,s,\xi),
\\[-2pt]
\mathrm{K} &=& \frac{1}{2}\mathbb{E}\int_0^t
\!\int_{(\mathbb
{T}^N)^2}\!\int_{\mathbb{R}^2}
\varrho(x-y)\psi(\xi-\zeta) \\[-2pt]
&&\qquad\qquad\hspace*{29pt}{}\times\sum_{k\geq1}
\bigl|g_k(x,\xi)-g_k(y, \zeta)\bigr|^2\,\mathrm{d}
\nu^1_{x,s}(\xi)\,\mathrm{d}\nu ^2_{y,s}(
\zeta)\,\mathrm{d}x\, \mathrm{d}y\,\mathrm{d}s.
\end{eqnarray*}
\end{prop}

\begin{pf}
The proof follows the ideas developed in \cite{debus}, Proposition~9,
and \cite{hof}, Proposition~3.2, and is left to the reader.
\end{pf}
%

%th3.3 #&#
\begin{thm}[(Comparison principle)]\label{uniqueness}
Let $u$ be a kinetic solution to \eqref{eq}. Then there exist $u^+$
and $u^-$, representatives of $u$, such that, for all $t\in[0,T]$,
$f^\pm(t,x,\xi)=\mathbf{1}_{u^{\pm}(t,x)>\xi}$ for a.e. $(\omega
,x,\xi)$. Moreover, if $ u_1,u_2$ are kinetic solutions to \eqref{eq} with
initial data $u_{1,0}$ and $u_{2,0}$, respectively, then for all $t\in
[0,T]$ we have
%
%e3.1 #&#
\begin{equation}
\label{comparison}
\mathbb{E}\bigl\|u_1^{\pm}(t)-u_2^{\pm}(t)
\bigr\|_{L^1(\mathbb{T}^N)}\leq \mathbb{E}\| u_{1,0}-u_{2,0}
\|_{L^1(\mathbb{T}^N)}.
\end{equation}
\end{thm}

\begin{pf}
Let $(\varrho_\varepsilon), (\psi_\delta)$ be approximations to
the identity on $\mathbb{T}^N$ and $\mathbb{R}$, respectively, that
is, let
$\varrho\in C^\infty(\mathbb{T}^N), \psi\in C^\infty_c(\mathbb
{R})$ be
symmetric nonnegative functions such as
$\int_{\mathbb{T}^N}\varrho=1, \int_\mathbb{R}\psi=1$
and $\operatorname{supp}\psi\subset(-1,1)$. We define
\begin{eqnarray*}
\varrho_\varepsilon(x) &=& \frac{1}{\varepsilon^N} \varrho \biggl(\frac
{x}{\varepsilon}
\biggr), \\
\psi_\delta(\xi) &=& \frac{1}{\delta} \psi \biggl(\frac{\xi}{\delta}
\biggr).
\end{eqnarray*}
Then
%
%\[
\begin{eqnarray*}
&& \mathbb{E}\int_{\mathbb{T}^N}\!\int_\mathbb{R}f_1^{\pm}(x,t,
\xi )\bar{f}_2^{\pm
}(x,t,\xi)\,\mathrm{d}\xi\,\mathrm{d}x
\\
&&\qquad =\mathbb{E}\int_{(\mathbb{T}^N)^2}\!\int_{\mathbb{R}^2}\varrho
_\varepsilon(x-y)\psi _\delta(\xi-\zeta)f_1^{\pm}(x,t,
\xi)\bar{f}_2^{\pm}(y,t,\zeta)\, \mathrm{d}\xi\,\mathrm{d}
\zeta\,\mathrm{d}x \,\mathrm{d}y\\
&&\qquad\quad{}+\eta _t(\varepsilon,\delta),
\end{eqnarray*}
%
%\]
%
where $\lim_{\varepsilon,\delta\rightarrow0}\eta_t(\varepsilon
,\delta)=0$. With regard to Proposition~\ref{propdoubling} we need
to find suitable bounds for terms $\mathrm{I}, \mathrm{J}, \mathrm{K}$.

Since $b$ has at most polynomial growth, there exist $C>0, p>1$ such that
\begin{eqnarray*}
\bigl|b(\xi)-b(\zeta) \bigr| &\leq & \varGamma(\xi,\zeta)|\xi-\zeta |,\\
\varGamma(\xi,\zeta) &\leq &
C \bigl(1+|\xi|^{p-1}+|\zeta |^{p-1} \bigr).
\end{eqnarray*}
Hence,
\[
|\mathrm{I}|\leq\mathbb{E}\int_0^t\!\int
_{(\mathbb{T}^N)^2}\!\int_{\mathbb{R}
^2}f_1
\bar{f}_2\varGamma(\xi,\zeta)|\xi-\zeta|\psi_\delta(\xi -
\zeta)\,\mathrm{d}\xi\,\mathrm{d}\zeta \bigl|\nabla_x\varrho
_\varepsilon (x-y) \bigr|\,\mathrm{d}x\,\mathrm{d}y\,\mathrm{d}s.
\]
As the next step, we apply integration by parts with respect to $\zeta, \xi$. Focusing only on the relevant integrals, we get
\begin{eqnarray*}
%\label{nn} %
&& \int_\mathbb{R}f_1(\xi)
\int_\mathbb{R}\bar{f}_2(\zeta)\varGamma (\xi,\zeta
)|\xi-\zeta|\psi_\delta(\xi-\zeta)\,\mathrm{d}\zeta\,\mathrm {d}\xi
\\
&&\qquad =\int_\mathbb{R}f_1(\xi)\int_\mathbb{R}
\varGamma\bigl(\xi,\zeta '\bigr)\bigl|\xi-\zeta'\bigr|
\psi_\delta\bigl(\xi-\zeta'\bigr)\,\mathrm{d}
\zeta'\,\mathrm{d}\xi
\\
&&\quad\qquad{}-\int_{\mathbb{R}^2} f_1(\xi)\int
_{-\infty}^\zeta \varGamma\bigl(\xi ,\zeta'
\bigr)\bigl|\xi-\zeta'\bigr|\psi_\delta\bigl(\xi-\zeta'
\bigr)\,\mathrm{d}\zeta'\, \mathrm{d} \xi\,\mathrm{d}
\nu^2_{y,s}(\zeta)
\\
&&\qquad=\int_{\mathbb{R}^2}f_1(\xi)\int^{\infty}_\zeta
\varGamma\bigl(\xi ,\zeta '\bigr)\bigl|\xi-\zeta'\bigr|
\psi_\delta\bigl(\xi-\zeta'\bigr)\,\mathrm{d}
\zeta'\, \mathrm{d}\xi\, \mathrm{d}\nu^2_{y,s}(
\zeta)
\\
&&\qquad=\int_{\mathbb{R}^2}\Upsilon(\xi,\zeta)\,\mathrm{d}\nu
^1_{x,s}(\xi)\, \mathrm{d}\nu^2_{y,s}(
\zeta),
\end{eqnarray*}
where
\[
\Upsilon(\xi,\zeta)=\int_{-\infty}^\xi\!\int
_\zeta^\infty \varGamma\bigl(\xi',
\zeta'\bigr)\bigl|\xi'-\zeta'\bigr|
\psi_\delta\bigl(\xi'-\zeta'\bigr)\,
\mathrm{d}\zeta'\,\mathrm{d}\xi'.
\]
Therefore, we get
\[
|\mathrm{I}|\leq\mathbb{E}\int_0^t\!\int
_{(\mathbb{T}^N)^2}\!\int_{\mathbb{R}
^2}\Upsilon(\xi,\zeta)\,
\mathrm{d}\nu^1_{x,s}(\xi)\,\mathrm {d}\nu
^2_{y,s}(\zeta) \bigl|\nabla_x\varrho_\varepsilon(x-y)
\bigr|\, \mathrm{d} x\,\mathrm{d}y\,\mathrm{d}s.
\]
The function $\Upsilon$ can be estimated using the substitution
$\xi''=\xi'-\zeta'$
%
%\[
\begin{eqnarray*}
\Upsilon(\xi,\zeta)&=& \int_\zeta^\infty\!\int
_{|\xi''|<\delta
, \xi''<\xi-\zeta'}\varGamma\bigl(\xi''+
\zeta',\zeta'\bigr)\bigl|\xi''\bigr|
\psi _\delta\bigl(\xi''\bigr)\,\mathrm{d}
\xi''\,\mathrm{d}\zeta'
\\
&\leq &  C \delta\int_\zeta^{\xi+\delta}\max
_{|\xi''|<\delta, \xi
''<\xi-\zeta'}\varGamma\bigl(\xi''+
\zeta',\zeta'\bigr)\,\mathrm{d}\zeta'
\\
&\leq &  C \delta\int_\zeta^{\xi+\delta} \bigl(1+|
\xi|^{p-1}+\bigl|\zeta '\bigr|^{p-1} \bigr)\,\mathrm{d}
\zeta'
\\
&\leq &  C\delta \bigl(1+|\xi|^{p}+|\zeta|^{p} \bigr)
\end{eqnarray*}
%
%\]
%
so
\[
|\mathrm{I}|\leq C t\delta\varepsilon^{-1}.
\]

%%%% the term J

In order to estimate the term $\mathrm{J}$, we observe that
%
%\[
\begin{eqnarray*}
\mathrm{J}&=& \mathbb{E}\int_0^t\!\int
_{(\mathbb{T}^N)^2}\!\int_{\mathbb{R}^2}f_1\bar{f}
_2 \bigl(\sigma(\xi)-\sigma(\zeta) \bigr)^2\dvtx
\mathrm{D}^2_x\varrho _\varepsilon(x-y)
\psi_\delta(\xi-\zeta)\, \mathrm{d}\xi\,\mathrm{d}\zeta\,\mathrm{d}x\,
\mathrm{d}y\, \mathrm{d}s
\\
&&{}+2 \mathbb{E}\int_0^t\!\int
_{(\mathbb{T}^N)^2}\!\int_{\mathbb
{R}^2}f_1\bar{f}
_2 \sigma(\xi)\sigma(\zeta)\dvtx \mathrm{D}^2_x
\varrho_\varepsilon (x-y)\psi_\delta(\xi-\zeta)\, \mathrm{d}\xi\,
\mathrm{d}\zeta\,\mathrm{d}x\, \mathrm{d}y\, \mathrm{d}s
\\
&&{}-\mathbb{E}\int_0^t\!\int
_{(\mathbb{T}^N)^2}\!\int_{\mathbb
{R}^2}\varrho
_\varepsilon(x-y)\psi_\delta(\xi-\zeta)\, \mathrm{d}\nu^ { 1 }
_ { x , s } (\xi) \,\mathrm{d}x\,\mathrm{d}n_{2,1}(y,s,\zeta)
\\
&&{}-\mathbb{E}\int_0^t\!\int
_{(\mathbb{T}^N)^2}\!\int_{\mathbb
{R}^2}\varrho
_\varepsilon(x-y)\psi_\delta(\xi-\zeta)\, \mathrm{d}
\nu^{2}_{y,s} (\zeta) \,\mathrm{d}y\,\mathrm{d}n_{1,1}(x,s,
\xi)
\\
&=&\mathrm{J}_1+\mathrm{J}_2+\mathrm{J}_3+
\mathrm{J}_4.
\end{eqnarray*}
%
%\]
%
Since $\sigma$ is locally $\gamma$-H\"older continuous due to \eqref
{sigma}, it holds
\[
%\begin{eqnarray}
|\mathrm{J}_1|\leq Ct\delta^{2\gamma}\varepsilon^{-2}.
%\end{eqnarray}
%
\]
Next, we will show that $\mathrm{J}_2+\mathrm{J}_3+\mathrm{J}_4\leq0$.
%Recall that in the case of existence of a weak solution to \eqref{eq}
%such that $u\in L^2(\Omega\semicol C([0,T]\semicol L^2(\mt^N)))\cap
%L^2(\Omega\semicol L^2(0,T\semicol H^1(\mt^N)))$, the parabolic
%dissipative measure is given by (cf. Subsection
%\ref{subsecformulation})
%$$\dif n_{1}(t,x,\xi)=|\sigma\nabla u|^2\,\dif\delta_{u=\xi}\,\dif x\,
%\dif t.$$
%Hence for any $\varphi\in C_0([0,T]\times\mt^N\times\mr)$ we have
%\[
%\begin{split}
%n_{1}(\varphi)&=\int_0^T\int_{\mt^N}\varphi(t,x,u)|\sigma(u)\nabla u|^2
%\,\dif x\,\dif t\\
%&=\int_0^T\int_{\mt^N}\varphi(t,x,u)\bigg|\diver\int^u_0\sigma(\xi)\,
%\dif\xi\bigg|^2\dif x\,\dif t.
%\end{split}
%\]
From the definition of the parabolic dissipative measure in Definition~\ref{kinsol}, we have
%
%\[
\begin{eqnarray*}
\mathrm{J}_3+\mathrm{J}_4&=& - \mathbb{E}\int
_0^t\!\int_{(\mathbb{T}
^N)^2}
\varrho_\varepsilon(x-y)\psi_\delta(u_1-u_2) \biggl|
\operatorname{div} _y\int^{u_2}_0
\sigma(\zeta)\,\mathrm{d}\zeta \biggr|^2\,\mathrm {d}x\, \mathrm{d}y\,
\mathrm{d}s
\\
&&{}- \mathbb{E}\int_0^t\!\int
_{(\mathbb{T}^N)^2}\varrho _\varepsilon (x-y)\psi_\delta(u_1-u_2)
\biggl|\operatorname{div}_x\int^{u_1}_0
\sigma(\xi)\, \mathrm{d}\xi \biggr|^2\mathrm{d}x\, \mathrm{d}y\,\mathrm{d}s.
\end{eqnarray*}
%
%\]
%
Moreover, due to the chain rule formula \eqref{eqchainrule} we deduce
%
%\[
\begin{eqnarray*}
\operatorname{div}\int_\mathbb{R}f\phi(\xi)\sigma(\xi)\,
\mathrm {d}\xi&=& \operatorname{div}\int_\mathbb{R}
\chi_f\phi(\xi)\sigma(\xi)\,\mathrm{d}\xi=\operatorname {div}\int
_0^u\phi(\xi )\sigma(\xi)\,\mathrm{d}\xi
\\
&=&\phi(u)\operatorname{div}\int_0^{u}\sigma(
\xi)\,\mathrm{d}\xi,
\end{eqnarray*}
%
%\]
%
where $\chi_{f}=\mathbf{1}_{u>\xi}-\mathbf{1}_{0>\xi}$.
With this in hand, we obtain
%
%\[
\begin{eqnarray*}
\mathrm{J}_2&=& 2 \mathbb{E}\int_0^t
\!\int_{(\mathbb{T}^N)^2}\!\int_{\mathbb{R}
^2}(\nabla_x
f_1)^*\sigma(\xi)\sigma(\zeta) (\nabla_y{f}
_2)\varrho_\varepsilon(x-y)\psi_\delta(\xi-\zeta)\,
\mathrm{d}\xi\,\mathrm{d}\zeta\,\mathrm{d}x\, \mathrm{d}y\, \mathrm{d}s
\\
&=& 2 \mathbb{E}\int_0^t\!\int
_{(\mathbb{T}^N)^2}\varrho_\varepsilon (x-y)\\
&&\hspace*{32pt}\qquad{}\times\operatorname{div}
_y\int_{0}^{u_2}\sigma(\zeta)\cdot
\operatorname{div}_x\int_{0}^{u_1}
\sigma (\xi)\psi_\delta(\xi-\zeta)\,\mathrm{d}\xi\,\mathrm{d}\zeta\,
\mathrm{d}x\, \mathrm{d} y\,\mathrm{d}s
\\
&=& 2 \mathbb{E}\int_0^t\!\int
_{(\mathbb{T}^N)^2}\varrho_\varepsilon (x-y)\psi _\delta(u_1-u_2)
\\
&&\hspace*{32pt}\qquad{}\times\operatorname{div}_x\int_{0}^{u_1}
\sigma(\xi)\, \mathrm{d}\xi\cdot \operatorname{div}_y\int
_{0}^{u_2}\sigma(\zeta)\,\mathrm{d}\zeta\,
\mathrm{d}x\,\mathrm{d}y\, \mathrm{d}s.
\end{eqnarray*}
%
%\]
%
And, therefore,
%
%\[
\begin{eqnarray*}
\mathrm{J}_2+\mathrm{J}_3+\mathrm{J}_4 &=& -
\mathbb{E}\int_0^t\!\int_{(\mathbb{T}^N)^2}
\varrho_\varepsilon(x-y)\psi_\delta(u_1-u_2)
\\
&&\hspace*{35pt}\qquad{}\times \biggl|\operatorname{div}_x\int_{0}^{u_1}
\sigma(\xi)\, \mathrm{d}\xi-\operatorname{div} _y\int
_{0}^{u_2}\sigma(\zeta)\,\mathrm{d}\zeta
\biggr|^2\,\mathrm {d}x\,\mathrm{d}y\, \mathrm{d}s\\
&\leq & 0.
\end{eqnarray*}

The last term is, due to \eqref{skorolip}, bounded as follows:
%
%\[
\begin{eqnarray*}
\mathrm{K}&\leq &  C \mathbb{E}\int_0^t\!\int
_{(\mathbb{T}^N)^2} \varrho _\varepsilon(x-y)|x-y|^2\int
_{\mathbb{R}^2}\psi_\delta(\xi -\zeta)\, \mathrm{d}
\nu^1_{x,s}(\xi)\,\mathrm{d}\nu^2_{y,s}(
\zeta)\,\mathrm {d}x\,\mathrm{d}y\, \mathrm{d}s
\\
&&{}+C \mathbb{E}\int_0^t\!\int
_{(\mathbb{T}^N)^2} \varrho _\varepsilon(x-y)\\
&&\hspace*{48pt}\qquad{}\times \int_{\mathbb{R}^2}
\psi_\delta(\xi-\zeta)|\xi-\zeta|h\bigl(|\xi -\zeta|\bigr)\, \mathrm{d}
\nu^1_{x,s}(\xi)\,\mathrm{d}\nu^2_{y,s}(
\zeta)\,\mathrm {d}x\,\mathrm{d}y\, \mathrm{d}s
\\
%&\leq Ct\delta^{-1}\int_{(\mt^N)^2}|x-y|^2\varrho_\varepsilon(x-y)\,
%\dif x\,\dif y+\frac{LTC_\psi h(\delta)}{2}\int_{(\mt^N)^2}\varrho_
%\varepsilon(x-y)\,\dif x\,\dif y\\
&\leq & Ct\delta^{-1}
\varepsilon^2+Ct h(\delta).
\end{eqnarray*}
%
%\]
%
%where $C_\psi=\sup_{\xi\in\mr}|\xi\psi(\xi)|$.

As a consequence, we deduce for all $t\in[0,T]$
%
%\[
\begin{eqnarray*}
&& \mathbb{E}\int_{\mathbb{T}^N}\!\int_\mathbb{R}f_1^{\pm}(x,t,
\xi )\bar{f}_2^{\pm
}(x,t,\xi)\,\mathrm{d}\xi\,\mathrm{d}x
\\
&&\qquad\leq\mathbb{E}\int_{(\mathbb{T}^N)^2}\!\int_{\mathbb{R}^2}
\varrho _\varepsilon (x-y)\psi_\delta(\xi-\zeta)f_{1,0}(x,
\xi)\bar{f}_{2,0}(y,\zeta)\, \mathrm{d}\xi\,\mathrm{d}\zeta\,\mathrm{d}x
\,\mathrm{d}y
\\
&&\quad\qquad{}+Ct\delta\varepsilon^{-1}+Ct\delta^{2\gamma}\varepsilon
^{-2}+Ct\delta^{-1}\varepsilon^2+Ct h(\delta)+
\eta_t(\varepsilon ,\delta).
\end{eqnarray*}
%
%\]
%
Taking $\delta=\varepsilon^{\beta}$ with $\beta\in(1/\gamma,2)$
and letting $\varepsilon\rightarrow0$ yields
\[
\mathbb{E}\int_{\mathbb{T}^N}\!\int_\mathbb{R}f_1^{\pm}(t)
\bar {f}_2^{\pm}(t)\,\mathrm{d} \xi\,\mathrm{d}x\leq
\mathbb{E}\int_{\mathbb{T}^N}\!\int_\mathbb
{R}f_{1,0}\bar{f}_{2,0}\, \mathrm{d}\xi\,\mathrm{d}x.
\]
Let us now consider $f_1=f_2=f$. Since $f_0=\mathbf{1}_{u_0>\xi}$ we have
the identity $f_0\bar{f}_0=0$ and, therefore, $f^\pm(1-f^\pm)=0$
a.e. $(\omega,x,\xi)$ and for all $t$. The fact that $f^\pm$ is a
kinetic function and Fubini's theorem then imply that, for any $t\in
[0,T]$, there exists a set $\Sigma_t\subset\Omega\times\mathbb
{T}^N$ of
full measure such that, for $(\omega,x)\in\Sigma_t$, $f^\pm(\omega
,x,t,\xi)\in\{0,1\}$ for a.e. $\xi\in\mathbb{R}$.
Therefore, there exist $u^\pm\dvtx  \Omega\times\mathbb{T}^N\times
[0,T]\rightarrow\mathbb{R}$ such that $f^\pm=\mathbf{1}_{u^\pm>\xi
}$ for a.e.
$(\omega,x,\xi)$ and all $t$. In particular, $u^\pm=\int_\mathbb
{R}(f^\pm
-\mathbf{1}_{0>\xi})\,\mathrm{d}\xi$ for a.e. $(\omega,x)$ and all
$t$. It
follows now from Proposition~\ref{limits} and the identity
\[
%\begin{eqnarray}
|\alpha-\beta|=\int_\mathbb{R}|\mathbf{1}_{\alpha>\xi}-
\mathbf {1}_{\beta>\xi}|\, \mathrm{d}\xi,\qquad\alpha, \beta\in\mathbb{R},
%\int_{\mt^N}\big|u^+(x,t)-u^-(x,t)|\,\dif x=\int_\mr\int_{\mt^N}\big|
%\ind_{u^+(x,t)>\xi}-\ind_{u^-(x,t)>\xi}\big|\,\dif\xi\,\dif x.
%\end{eqnarray}
%
\]
that $u^+=u^-=u$ for a.e. $t\in[0,T]$.
Since
\[
\int_\mathbb{R}\mathbf{1}_{u^\pm_1>\xi}\overline{\mathbf
{1}_{u^\pm_2>\xi}}\,\mathrm{d} \xi=\bigl(u^\pm_1-u^\pm_2
\bigr)^+
\]
we obtain the comparison principle \eqref{comparison}.
%$$\stred\big\|\big(u_1^{\pm}(t)-u_2^{\pm}(t)\big)^+\big\|_{L^1(\mt^N)}
%\leq\stred\big\|(u_{1,0}-u_{2,0})^+\big\|_{L^1(\mt^N)}.$$
%
%which completes the proof since
%$$\int_\mr\ind_{u_1>\xi}\overline{\ind_{u_2>\xi}}\,\dif
%\xi=(u_1-u_2)^+.$$
%%we have the contraction property
%%$$\stred\big\|\big(u_1^{\pm}(t)-u_2^{\pm}(t)\big)^+\big\|_{L^1(\mt^N)}
%\leq\stred\big\|(u_{1,0}-u_{2,0})^+\big\|_{L^1(\mt^N)}.$$
\end{pf}

As a consequence, we obtain the continuity of trajectories in
$L^p(\mathbb{T}
^N)$ whose proof is given in \cite{hof}, Corollary~3.4.

%co3.4 #&#
\begin{cor}[(Continuity in time)]\label{cont}
Let $u$ be a kinetic solution to \eqref{eq}. Then there exists a
representative of $u$ which has almost surely continuous trajectories
in $L^p(\mathbb{T}^N)$, for all $p\in[1,\infty)$.
\end{cor}

\section{Existence for nondegenerate case---$B$ Lipschitz continuous}
\label{secnondeg}

As the first step toward the existence part of Theorem~\ref{thmmain},
we prove existence of a weak solution to \eqref{eq}
%\begin{equation}\label{eqnondeg}
%\begin{split}
%\dif u+\diver\big(B(u)\big)\,\dif t&=\diver\big(A(u)\nabla u\big)\,
%\dif t+\Phi(u)\,\dif W,\\
%u(0)&=u_0,
%\end{split}
%\end{equation}
under three additional hypotheses. Recall that once this claim is
verified, Theorem~\ref{thmmain} follows immediately as any weak
solution to \eqref{eq} is also a kinetic solution to \eqref{eq}, due
to Section~\ref{subsecformulation}. Throughout this section, we
suppose that:\vspace*{-2pt}
\begin{longlist}[(H3)]
\item[(H1)] $u_0\in L^p(\Omega;C^5(\mathbb{T}^N))$, for all $p\in
[1,\infty)$,
\item[(H2)] $A$ is positive definite, that is, $A\geq\tau\mathrm{I}$,
\item[(H3)] $B$ is Lipschitz continuous hence it has linear growth
$|B(\xi)|\leq L(1+|\xi|)$.
\end{longlist}
In the following sections, we will show how we may relax all these
assumptions one after the other.

Let\vspace*{-2pt} us approximate \eqref{eq} by
%
%\begin{equation}
%e4.1 #&#
\begin{eqnarray}
\mathrm{d}u+\operatorname{div} \bigl(B^{\eta}(u) \bigr)\,
\mathrm {d}t&=& \operatorname{div} \bigl(A^\eta (u)\nabla u \bigr)\,
\mathrm{d}t-\eta\Delta^2u\,\mathrm{d}t+\Phi ^\eta(u)\,
\mathrm{d}W,
\nonumber
\\[-10pt]
\label{eqapprox}
\\[-10pt]
\nonumber
u(0)&= & u_0,
\end{eqnarray}
%
%\end{equation}
%
where $B^\eta,  A^\eta, \Phi^\eta$ are smooth approximations
of $B, A$ and $\Phi$, respectively, with bounded derivatives. Then
the following existence result holds true.

%th4.1 #&#
\begin{thm}
For any $\eta\in(0,1)$, there exists a unique strong solution to~\eqref{eqapprox} that belongs to
\[
L^p\bigl(\Omega;C\bigl([0,T];C^{4,\lambda}\bigl(
\mathbb{T}^N\bigr)\bigr)\bigr)\qquad\forall \lambda\in (0,1), \forall
p\in[1,\infty).
\]
\end{thm}

\begin{pf}
The second-order term in \eqref{eqapprox} can be rewritten in the
following\vspace*{-2pt} way:
\[
\operatorname{div} \bigl(A^\eta(u)\nabla u \bigr)=\sum
_{i,j=1}^N\partial ^2_{{x_i}{x_j}}
\bar{A}^\eta_{ij}(u),\qquad\bar{A}^\eta(\xi)=\int
_0^{\xi}A^\eta(\zeta)\,\mathrm{d}\zeta,
\]
hence \cite{hof2}, Corollary~2.2, applies.
\end{pf}

%re4.2 #&#
\begin{rem}
Due to the fourth-order term $-\eta\Delta^2 u$ there are no a priori
estimates of the $L^p(\mathbb{T}^N)$-norm for solutions of the approximations
\eqref{eqapprox} and that is the reason why we cannot deal directly
with \eqref{eq} if the coefficients have polynomial growth. To
overcome this difficulty, we proceed in two steps and avoid the
additional assumption upon $B$ in the next section. Note that the
linear growth hypothesis is satisfied for the remaining coefficients,
that is, for $\bar{A}(\xi)=\int_0^\xi A(\zeta)\,\mathrm{d}\zeta$ since
$A\in C_b(\mathbb{R})$ and for $\Phi$ due to \eqref{linrust}.
\end{rem}

\begin{prop}\label{propenergy1}
For any $p\in[2,\infty)$, the solution to \eqref{eqapprox}
satisfies the following energy estimate:
%
%\begin{equation}
%e4.2 #&#
\begin{eqnarray}
 && \mathbb{E}\sup_{0\leq t\leq T}\bigl\|u^\eta(t)
\bigr\|_{L^2(\mathbb
{T}^N)}^p+p\tau \mathbb{E}\int_0^T
\bigl\|u^{\eta}\bigr\|_{L^2(\mathbb{T}^N)}^{p-2}\bigl\|\nabla u^{\eta}\bigr\|
^2_{L^2(\mathbb{T}^N)}\,\mathrm{d}s
\nonumber
\\[-9pt]
\label{eqenergy}
\\[-9pt]
\nonumber
&&\qquad\leq C \bigl(1+\mathbb{E}\|u_0\|_{L^2(\mathbb{T}^N)}^p
\bigr),
\end{eqnarray}
%
%\end{equation}
%
where the constant $C$ does not depend on $\eta, \tau$ and $L$.
\end{prop}

\begin{pf}
Let us apply the It\^o formula to the function $f(v)=\|v\|_{L^2(\mathbb{T}
^N)}^p$. We obtain
%
%\[
\begin{eqnarray*}
%\label{if1}
\bigl\|u^\eta(t)\bigr\|_{L^2(\mathbb{T}^N)}^p&=&
\|u_0\|_{L^2(\mathbb
{T}^N)}^p-p\int_0^t
\bigl\| u^\eta\bigr\|_{L^2(\mathbb{T}^N)}^{p-2} \bigl\langle
u^\eta,\operatorname {div} \bigl(B^\eta \bigl(u^\eta
\bigr) \bigr) \bigr\rangle\,\mathrm{d}s
\\
&&{}+p\int_0^t\bigl\|u^\eta
\bigr\|_{L^2(\mathbb{T}^N)}^{p-2} \bigl\langle u^\eta ,
\operatorname{div} \bigl(A^\eta\bigl(u^\eta\bigr)\nabla
u^\eta \bigr) \bigr\rangle\,\mathrm{d} s
\\
&&{}-p\eta\int_0^t\bigl\|u^\eta
\bigr\|_{L^2(\mathbb{T}^N)}^{p-2} \bigl\langle u^\eta,
\Delta^2 u^\eta \bigr\rangle\,\mathrm{d}s
\\
&&{}+p\sum_{k\geq1}\int_0^t
\bigl\|u^\eta\bigr\|_{L^2(\mathbb
{T}^N)}^{p-2} \bigl\langle u^\eta,
g_k^\eta\bigl(u^\eta\bigr) \bigr\rangle\,
\mathrm{d}\beta _k(s)
\\
&&{}+\frac{p}{2}\int_0^t
\bigl\|u^\eta\bigr\|_{L^2(\mathbb{T}^N)}^{p-2}\bigl\| G_\eta
\bigl(u^\eta\bigr)\bigr\|_{L^2(\mathbb{T}^N)}^2\,\mathrm{d}s
\\
&&{}+\frac{p(p-2)}{2}\sum_{k\geq1}\int
_0^t\bigl\|u^\eta\bigr\| _{L^2(\mathbb{T}
^N)}^{p-4}
\bigl\langle u^\eta, g_k^\eta
\bigl(u^\eta\bigr) \bigr\rangle ^2\,\mathrm{d} s
\\
&=&\mathrm{J}_1+\cdots+\mathrm{J}_7.
\end{eqnarray*}
%
%\]
%
Setting $H(\xi)=\int_0^\xi B^\eta(\zeta)\,\mathrm{d}\zeta$, we conclude
that the second term on the right-hand side vanishes, the third one as
well as the fourth one is nonpositive
%
%\[
\begin{eqnarray*}
\mathrm{J}_3+\mathrm{J}_4
&\leq & -p\tau\int
_0^t\bigl\|u^\eta\bigr\| _{L^2(\mathbb{T}
^N)}^{p-2}
\bigl\|\nabla u^\eta\bigr\|_{L^2(\mathbb{T}^N)}^2\,\mathrm{d}s\\
&&{}-p\eta \int
_0^t\bigl\| u^\eta\bigr\|_{L^2(\mathbb{T}^N)}^{p-2}
\bigl\|\Delta u^\eta\bigr\|_{L^2(\mathbb
{T}^N)}^2\,\mathrm{d}s,
\end{eqnarray*}
%
%\]
%
the sixth and seventh term are estimated as follows:
\[
%\begin{eqnarray}
\mathrm{J}_6+\mathrm{J}_7\leq C \biggl(1+\int
_0^t\bigl\|u^\eta\bigr\|_{L^2(\mathbb{T}^N)}^p
\,\mathrm{d}s \biggr),
%\end{eqnarray}
%
\]
and since expectation of $\mathrm{J}_5$ is zero, we get
%
%\[
\begin{eqnarray*}
&& \mathbb{E}\bigl\|u^\eta(t)\bigr\|_{L^2(\mathbb{T}^N)}^p+p\tau \mathbb
{E}\int_0^t\bigl\| u^{\eta}
\bigr\|_{L^2(\mathbb{T}^N)}^{p-2}\bigl\|\nabla u^\eta\bigr\|^2_{L^2(\mathbb
{T}^N)}
\,\mathrm{d} s
\\
%&+p\eta\,\stred\int_0^t\|u^{\eta}\|_{L^2(\mt^N)}^{p-2}\|\Delta u^\eta
%\|^2_{L^2(\mt^N)}\,\dif s\\
&&\qquad\leq\mathbb{E}\|u_0
\|_{L^2(\mathbb{T}^N)}^p +C \biggl(1+\int_0^t
\mathbb{E}\bigl\|u^\eta(s)\bigr\|_{L^2(\mathbb{T}^N)}^p\,\mathrm{d}s
\biggr).
\end{eqnarray*}
%
%\]
%
Application of the Gronwall lemma now yields
%
%\[
%\label{klm1}
\begin{eqnarray*}
&& \mathbb{E}\bigl\|u^\eta(t)\bigr\|_{L^2(\mathbb{T}^N)}^p+p\tau \mathbb
{E}\int_0^t\bigl\|u^{\eta}
\bigr\|_{L^2(\mathbb{T}^N)}^{p-2}\bigl\|\nabla u^\eta\bigr\|^2_{L^2(\mathbb
{T}^N)}
\,\mathrm{d}s\\
&&\qquad \leq C \bigl(1+\mathbb{E}\|u_0\|_{L^2(\mathbb{T}^N)}^p
\bigr).
\end{eqnarray*}
%
%\]

In order to obtain an estimate of $\mathbb{E}\sup_{0\leq t\leq T}\|
u^\eta
(t)\|_{L^2(\mathbb{T}^N)}^p$, we proceed similarly as above to get
%
%\[
\begin{eqnarray*}
\mathbb{E} \sup_{0\leq t\leq T}\bigl\|u^\eta(t)
\bigr\|_{L^2(\mathbb
{T}^N)}^p &\leq & \mathbb{E} \|u_0
\|_{L^2(\mathbb{T}^N)}^p+C \biggl(1+\int_0^T
\mathbb{E}\bigl\|u^\eta\bigr\| _{L^2(\mathbb{T}
^N)}^p\,\mathrm{d}s \biggr)
\\
&&{}+p \mathbb{E}\sup_{0\leq t\leq T} \biggl|\sum_{k\geq
1}
\int_0^t\bigl\|u^\eta\bigr\|_{L^2(\mathbb{T}^N)}^{p-2}
\bigl\langle u^\eta, g_k^\eta
\bigl(u^\eta\bigr) \bigr\rangle\,\mathrm{d}\beta_k(s) \biggr|
\end{eqnarray*}
%
%\]
%
and for the stochastic integral we employ the Burkholder--Davis--Gundy and the Schwartz inequality, the assumption \eqref
{linrust} and the weighted Young inequality
%
%\[
\begin{eqnarray*}
&& \mathbb{E}\sup_{0\leq t\leq T} \biggl|\sum_{k\geq1}
\int_0^t\bigl\| u^\eta\bigr\|
_{L^2(\mathbb{T}^N)}^{p-2} \bigl\langle u^\eta,
g_k^\eta\bigl(u^\eta\bigr) \bigr\rangle\,
\mathrm{d}\beta_k(s) \biggr|
\\
&&\qquad \leq C \mathbb{E} \biggl(\int_0^T
\bigl\|u^\eta\bigr\|_{L^2(\mathbb
{T}^N)}^{2p-4}\sum
_{k\geq1} \bigl\langle u^\eta, g_k^\eta
\bigl(u^\eta\bigr) \bigr\rangle ^2\,\mathrm{d} s
\biggr)^{{1}/{2}}
\\
&&\qquad \leq C \mathbb{E} \biggl(\int_0^T\bigl\|
u^\eta\bigr\|_{L^2(\mathbb
{T}^N)}^{2p-2}\sum
_{k\geq1}\bigl\| g_k^\eta\bigl(u^\eta
\bigr)\bigr\|_{L^2(\mathbb{T}^N)}^2\,\mathrm {d}s \biggr)^{{1}/{2}}
\\
&&\qquad\leq C \mathbb{E} \Bigl(\sup_{0\leq t\leq T}\bigl\|u^\eta(t)\bigr\|
^p_{L^2(\mathbb{T}
^N)} \Bigr)^{{1}/{2}} \biggl(1+\int
_0^T\bigl\|u^\eta\bigr\|_{L^2(\mathbb{T}
^N)}^p
\,\mathrm{d}s \biggr)^{{1}/{2}}
\\
&&\qquad \leq\frac{1}{2} \mathbb{E}\sup_{0\leq t\leq T}
\bigl\|u^\eta(t)\bigr\| ^p_{L^2(\mathbb{T}^N)}+C \biggl(1+\int
_0^T\mathbb{E}\bigl\|u^\eta\bigr\|
_{L^2(\mathbb{T}^N)}^p\, \mathrm{d}s \biggr).
\end{eqnarray*}
%
%\]
%
%Finally, we obtain \eqref{eqenergy} by Gronwall's lemma.
This gives \eqref{eqenergy}.
\end{pf}
%

%We shall make use of a variation of the Burkholder-Davis-Gundy
%inequality, established in \cite{fland}, which applies to fractional
%derivatives of stochastic integral. For $p\in[2,\infty)$ and $\alpha
%\in[0,1/2)$ we have
%\begin{equation}\label{bdg}
%\stred\bigg\|\int_0^t \Psi\,\dif W\bigg\|^p_{W^{\alpha,p}(0,T;X)}\leq C
%\,\stred\int_0^T\|\Psi\|_{L_2(\mathfrak{U},X)}^p\dif t
%\end{equation}
%holding for any predictable process $\Psi\in L^p(\Omega;L^p(0,T;L_2(
%\mathfrak{U},X)))$.

%pr4.4 #&#
\begin{prop}\label{propholder}
For all $\lambda\in(0,1/2)$, there exists a constant $C>0$ such that
for all $\eta\in(0,1)$
\[
\mathbb{E}\bigl\|u^\eta\bigr\|_{C^\lambda([0,T];H^{-3}(\mathbb{T}^N))}\leq C.
\]
\end{prop}

\begin{pf}
Recall that due to Proposition~\ref{propenergy1}, the set $\{u^\eta
; \eta\in(0,1)\}$ is bounded in $L^2(\Omega;L^2(0,T;H^1(\mathbb{T}^N)))$.
Since the coefficients $B^\eta, \bar{A}^\eta$ have linear growth
uniformly in $\eta$ we conclude, in particular, that
\[
\bigl\{\operatorname{div}\bigl(B^\eta\bigl(u^\eta\bigr)
\bigr)\bigr\},\qquad \bigl\{\operatorname {div}\bigl(A^\eta
\bigl(u^\eta\bigr)\nabla u^\eta\bigr)\bigr\}, \qquad\bigl\{\eta
\Delta^2 u^\eta\bigr\}
\]
are bounded in $L^2(\Omega;L^2(0,T;H^{-3}(\mathbb{T}^N)))$, and consequently
\[
\mathbb{E} \biggl\|u^\eta-\int_0^\cdot
\Phi^\eta\bigl(u^\eta\bigr)\, \mathrm{d}W
\biggr\|_{C^{1/2}([0,T];H^{-3}(\mathbb{T}^N))}\leq C.
\]

Moreover, for all $\lambda\in(0,1/2)$, paths of the above stochastic
integral are $\lambda$-H\"{o}lder continuous $L^2(\mathbb{T}^N)$-valued
functions and
\[
\mathbb{E} \biggl\|\int_0^\cdot\Phi^\eta
\bigl(u^\eta\bigr)\,\mathrm {d}W \biggr\| _{C^\lambda([0,T];L^2(\mathbb{T}^N))}\leq C.
\]
Indeed, it is a consequence of the Kolmogorov continuity theorem (see
\cite{daprato}, Theorem~3.3) since the following uniform estimate
holds true. Let $a>2, s,t\in[0,T]$, then
%
%\[
\begin{eqnarray*}
\mathbb{E} \biggl\|\int_s^t\Phi^\eta
\bigl(u^\eta\bigr)\,\mathrm{d}W \biggr\|^a&\leq &  C \mathbb{E}
\biggl(\int_s^t \bigl\|\Phi^\eta
\bigl(u^\eta\bigr)\bigr\| _{L_2(\mathfrak
{U};L^2(\mathbb{T}^N))}^2\,\mathrm{d}r
\biggr)^{{a}/{2}}
\\
&\leq &  C |t-s|^{{a}/{2}-1}\mathbb{E}\int_s^t
\biggl(\sum_{k\geq
1}\bigl\| g_k^\eta
\bigl(u^\eta\bigr)\bigr\|_{L^2(\mathbb{T}^N)}^2 \biggr)^{{a}/{2}}\,\mathrm {d}r
\\
&\leq &  C |t-s|^{{a}/{2}} \Bigl(1+\mathbb{E}\sup_{0\leq t\leq T}\bigl\|
u^\eta(t)\bigr\|^a_{L^2(\mathbb{T}^N)} \Bigr)
\\
&\leq &  C |t-s|^{{a}/{2}} \bigl(1+\mathbb{E}\|u_0
\|_{L^2(\mathbb
{T}^N)}^a \bigr),
\end{eqnarray*}
%
%\]
%
where we made use of the Burkholder--Davis--Gundy inequality, \eqref
{linrust} and Proposition~\ref{propenergy1}.
\end{pf}
%

%s4.1 #&#
\subsection{Compactness argument}
\label{subseccompact}

Let us define the path space $\mathcal{X}=\mathcal{X}_u\times
\mathcal{X}_W$, where
\[
\mathcal{X}_u=L^2 \bigl(0,T;L^2\bigl(
\mathbb{T}^N\bigr) \bigr)\cap C \bigl([0,T];H^{-4}\bigl(
\mathbb{T}^N\bigr) \bigr),\qquad  \mathcal{X}_W=C \bigl([0,T];
\mathfrak{U}_0 \bigr).
\]
Let us denote by $\mu_{u^\eta}$ the law of $u^\eta$ on $\mathcal
{X}_u$, $\eta\in(0,1)$, and by $\mu_W$ the law of $W$ on~$\mathcal
{X}_W$. Their joint law on $\mathcal{X}$ is then denoted by $\mu^\eta$.

%Now, we have all in hand to show tightness of the collection $\{\mu^
%\eta; \eta\in(0,1)\}$ in $\mathcal{X}.$

%pr4.5 #&#
\begin{prop}\label{tight}
The set $\{\mu^\eta; \eta\in(0,1)\}$ is tight and, therefore,
relatively weakly compact in $\mathcal{X}$.
\end{prop}

\begin{pf}
First, we prove tightness of $\{\mu_{u^\eta}; \eta\in(0,1)\}$
which follows directly from Propositions \ref{propenergy1} and \ref
{propholder} by making use of the embeddings% (see \cite{fland} for
%the general case)
%
%\[
\begin{eqnarray*}
C^\lambda\bigl([0,T];H^{-3}\bigl(\mathbb{T}^N\bigr)
\bigr)& \hookrightarrow &  H^\alpha \bigl(0,T;H^{-3}\bigl(
\mathbb{T}^N\bigr)\bigr),\qquad \alpha<\lambda,
\\
C^\lambda\bigl([0,T];H^{-3}\bigl(\mathbb{T}^N\bigr)
\bigr)& \stackrel{c} {\hookrightarrow} &  C\bigl([0,T];H^{-4}\bigl(
\mathbb{T}^N\bigr)\bigr),
\\
L^2\bigl(0,T;H^1\bigl(\mathbb{T}^N\bigr)
\bigr)\cap H^{\alpha}\bigl(0,T;H^{-3}\bigl(\mathbb
{T}^N\bigr)\bigr)& \stackrel {c} {\hookrightarrow} &  L^2
\bigl(0,T;L^2\bigl(\mathbb{T}^N\bigr)\bigr).
\end{eqnarray*}
%
%\]
%
Indeed, for $R>0$ we define the set
\begin{eqnarray*}
B_{R} &=&  \bigl\{u\in L^2\bigl(0,T;H^{1}\bigl(
\mathbb{T}^N\bigr)\bigr)\cap C^{\lambda
}\bigl([0,T];H^{-3}
\bigl(\mathbb{T}^N\bigr)\bigr);
\\
&& \hspace*{1pt}\quad \|u\|_{L^2(0,T;H^{1}(\mathbb{T}^N))}+\|u\|_{C^{\lambda
}([0,T];H^{-3}(\mathbb{T}
^N))}\leq R \bigr\}
\end{eqnarray*}
which is thus relatively compact in $\mathcal{X}_u$. Moreover, by
Propositions \ref{propenergy1} and \ref{propholder}
%
%\[
\begin{eqnarray*}
\mu_{u^\eta} \bigl(B_{R}^C \bigr)&\leq & \mathbb{P}
\biggl(\bigl\|u^\eta\bigr\| _{L^2(0,T;H^{1}(\mathbb{T}^N))}>\frac{R}{2} \biggr)+\mathbb{P}
\biggl(\bigl\|u^\eta\bigr\| _{C^{\lambda}([0,T];H^{-3}(\mathbb{T}^N))}>\frac{R}{2} \biggr)
\\
&\leq & \frac{2}{R} \bigl(\mathbb{E}\bigl\|u^\eta\bigr\|_{L^2(0,T;H^{1}(\mathbb{T}
^N))}+
\mathbb{E}\bigl\|u^\eta\bigr\|_{C^{\lambda}([0,T];H^{-3}(\mathbb
{T}^N))} \bigr)\leq \frac{C}{R}
\end{eqnarray*}
%
%\]
%
hence given $\vartheta>0$ there exists $R>0$ such that
\[
\mu_{u^\eta}(B_{R})\geq1-\vartheta.
\]
Besides, since the law $\mu_W$ is tight as being a Radon measure on
the Polish space~$\mathcal{X}_W$, we conclude that also the set of
their joint laws $\{\mu^\eta; \eta\in(0,1)\}$ is tight and
Prokhorov's theorem therefore implies that it is relatively weakly compact.
\end{pf}

Passing to a weakly convergent subsequence $\mu^n=\mu^{\eta_n}$ (and
denoting by $\mu$ the limit law), we now apply the Skorokhod embedding
theorem to infer the following result.

%pr4.6 #&#
\begin{prop}
There exists a probability space $(\tilde{\Omega},\tilde{\mathscr{F}
},\tilde{\mathbb{P}})$ with a sequence of $\mathcal{X}$-valued random
variables $(\tilde{u}^n,\tilde{W}^n), n\in\mathbb{N}$, and
$(\tilde
{u},\tilde{W})$ such that:
\begin{longlist}[(ii)]
\item[(i)] the laws of $(\tilde{u}^n,\tilde{W}^n)$ and $(\tilde{u},\tilde
{W})$ under $ \tilde{\mathbb{P}}$ coincide with $\mu^n$ and $\mu$,
respectively,
\item[(ii)] $(\tilde{u}^n,\tilde{W}^n)$ converges $ \tilde{\mathbb{P}
}$-almost surely to $(\tilde{u},\tilde{W})$ in the topology of
$\mathcal{X}$.
\end{longlist}
\end{prop}

Finally, let $(\tilde{\mathscr{F}}_t)$ be the $\tilde{\mathbb{P}}$-augmented
canonical filtration of the process $(\tilde{u},\tilde{W})$, that is
\[
\tilde{\mathscr{F}}_t=\sigma \bigl(\sigma (\varrho_t
\tilde {u},\varrho _t\tilde{W} )\cup \bigl\{N\in\tilde{\mathscr{F}};
\tilde {\mathbb{P} }(N)=0 \bigr\} \bigr),\qquad t\in[0,T],
\]
where $\varrho_t$ is the operator of restriction to the interval
$[0,t]$, that is, if $E$ is a Banach space and $t\in[0,T]$, we define
%
%\[
\begin{eqnarray*}
%\label{restr}
\varrho_t\dvtx  C\bigl([0,T];E\bigr)&\longrightarrow &  C
\bigl([0,t];E\bigr),
\\
k&\longmapsto &  k|_{[0,t]}.
\end{eqnarray*}
%
%\]
%
Clearly, $\varrho_t$ is a continuous mapping.

%s4.2 #&#
\subsection{Identification of the limit}
\label{subsecidentif}

The aim of this subsection is to prove the following.

%pr4.7 #&#
\begin{prop}\label{propmartsol}
$ ((\tilde{\Omega},\tilde{\mathscr{F}},(\tilde{\mathscr
{F}}_t),\tilde{\mathbb{P}
}),\tilde{W},\tilde{u} )$
is a weak martingale solution to~\eqref{eq} provided $\mathrm{(H1),
(H2)}$ and $\mathrm{(H3)}$ are fulfilled.
%Moreover,$$\tilde{u}\in L^2(\Omega;C([0,T];L^2(\mt^N)))\cap L^2(
%\Omega;L^2(0,T;H^1(\mt^N))).$$
\end{prop}

The proof is based on a new general method of constructing martingale
solutions of SPDEs that does not rely on any kind of martingale
representation theorem and, therefore, holds independent interest
especially in situations where these representation theorems are no
longer available. For other applications of this method, we refer the
reader to \cite{on1,hof,hofse,on2}.

Let us define for all $t\in[0,T]$ and a test function $\varphi\in
C^\infty(\mathbb{T}^N)$
%
%\[
%e4.3 #&#
%e4.4 #&#
%e4.5 #&#
%e4.6 #&#
%e4.7 #&#
\begin{eqnarray*}
M^n(t)&=&  \bigl\langle u^n(t),\varphi \bigr\rangle-
\langle u_0,\varphi \rangle+\int_0^t
\bigl\langle\operatorname{div} \bigl(B^n\bigl(u^n\bigr)
\bigr),\varphi \bigr\rangle\,\mathrm{d}s
\\
&&{}-\int_0^t \bigl\langle
\operatorname{div} \bigl(A^n\bigl(u^n\bigr)\nabla
u^n \bigr),\varphi \bigr\rangle\,\mathrm{d}s+\eta_n\int
_0^t \bigl\langle \Delta ^2u^n,
\varphi \bigr\rangle\,\mathrm{d}s,\qquad n\in\mathbb{N},
\\
\tilde{M}^n(t)&=&  \bigl\langle\tilde{u}^n(t),\varphi
\bigr\rangle - \langle u_0,\varphi \rangle+\int_0^t
\bigl\langle \operatorname{div} \bigl(B^n\bigl(\tilde{u}^n
\bigr) \bigr),\varphi \bigr\rangle\,\mathrm{d}s
\\
&&{}-\int_0^t \bigl\langle
\operatorname{div} \bigl(A^n\bigl(\tilde {u}^n\bigr)
\nabla \tilde{u}^n \bigr),\varphi \bigr\rangle\,\mathrm{d}s+
\eta_n\int_0^t \bigl\langle
\Delta^2\tilde{u}^n,\varphi \bigr\rangle\,\mathrm{d}s,
\qquad n\in \mathbb{N},
\\
\tilde{M}(t)&=& \bigl\langle\tilde{u}(t),\varphi \bigr\rangle- \langle
u_0,\varphi \rangle+\int_0^t
\bigl\langle\operatorname {div} \bigl(B(\tilde{u}) \bigr),\varphi \bigr\rangle\,
\mathrm{d}s-\int_0^t \bigl\langle
\operatorname{div} \bigl(A(\tilde{u})\nabla\tilde{u} \bigr),\varphi \bigr\rangle\,
\mathrm{d}s.
\end{eqnarray*}
%
%\]
%
Hereafter, times $s,t\in[0,T], s\leq t$, and a continuous function
\[
\gamma\dvtx  C \bigl([0,s];H^{-4}\bigl(\mathbb{T}^N\bigr) \bigr)
\times C \bigl([0,s];\mathfrak {U}_0 \bigr)\longrightarrow[0,1]
\]
will be fixed but otherwise arbitrary. The proof is an immediate
consequence of the following two lemmas.

%le4.8 #&#
\begin{lemma}
The process $\tilde{W}$ is a $(\tilde{\mathscr{F}}_t)$-cylindrical Wiener
process, that is, there exists a collection of mutually independent
real-valued $(\tilde{\mathscr{F}}_t)$-Wiener processes $\{\tilde
{\beta}_k\}
_{k\geq1}$ such that $\tilde{W}=\sum_{k\geq1}\tilde{\beta}_k e_k$.
\end{lemma}

\begin{pf}
Obviously, $\tilde{W}$ is a $\mathfrak{U_0}$-valued cylindrical
Wiener process and is $(\tilde{\mathscr{F}}_t)$-adapted. According to
the L\'
evy martingale characterization theorem, it remains to show that it is
also a $(\tilde{\mathscr{F}}_t)$-martingale.
It holds true
\[
\tilde{\mathbb{E}} \gamma \bigl(\varrho_s \tilde{u}^n,
\varrho _s\tilde {W}^n \bigr) \bigl[
\tilde{W}^n(t)-\tilde{W}^n(s) \bigr]=\mathbb{E} \gamma
\bigl(\varrho_s u^n,\varrho_s W \bigr)
\bigl[W(t)-W(s) \bigr]=0
\]
since $W$ is a martingale and the laws of $(\tilde{u}^n,\tilde{W}^n)$
and $(u^n,W)$ coincide.
Next, the uniform estimate
\[
\sup_{n\in\mathbb{N}}\tilde{\mathbb{E}}\bigl\|\tilde{W}^n(t)\bigr\|
_{\mathfrak
{U}_0}^2=\sup_{n\in\mathbb{N}}\mathbb{E}\bigl\|W(t)
\bigr\|^2_{\mathfrak
{U}_0}<\infty
\]
and the Vitali convergence theorem yields
\[
\tilde{\mathbb{E}} \gamma (\varrho_s\tilde{u},\varrho
_s\tilde {W} ) \bigl[\tilde{W}(t)-\tilde{W}(s) \bigr]=0
\]
which completes the proof.
\end{pf}

%le4.9 #&#
\begin{lemma}
The processes
\[
\tilde{M},\qquad \tilde{M}^2-\sum_{k\geq1}\int
_0^\cdot \bigl\langle g_k(\tilde{u}),
\varphi \bigr\rangle^2\,\mathrm{d}r,\qquad \tilde {M}\tilde {
\beta}_k-\int_0^\cdot \bigl\langle
g_k(\tilde{u}),\varphi \bigr\rangle\,\mathrm{d}r
\]
are $(\tilde{\mathscr{F}}_t)$-martingales.
\end{lemma}

\begin{pf}
Here, we use the same approach as in the previous lemma. Let us denote
by $\tilde{\beta}^n_k, k\geq1$ the real-valued Wiener processes
corresponding to $\tilde{W}^n$, that is $\tilde{W}^n=\sum_{k\geq
1}\tilde{\beta}_k^n e_k$. For all $n\in\mathbb{N}$, the process
\[
M^n=\int_0^\cdot \bigl\langle
\Phi^n\bigl(u^n\bigr)\,\mathrm{d}W(r),\varphi \bigr
\rangle=\sum_{k\geq1}\int_0^\cdot
\bigl\langle g_k^n\bigl(u^n\bigr),\varphi
\bigr\rangle\,\mathrm{d}\beta_k(r)
\]
is a square integrable $(\mathscr{F}_t)$-martingale by \eqref
{linrust} and
\eqref{eqenergy} and, therefore,
\[
\bigl(M^n\bigr)^2-\sum_{k\geq1}
\int_0^\cdot \bigl\langle g_k^n
\bigl(u^n\bigr),\varphi \bigr\rangle^2\,\mathrm{d}r,
\qquad M^n\beta_k-\int_0^\cdot
\bigl\langle g_k^n\bigl(u^n\bigr),\varphi
\bigr\rangle\,\mathrm{d}r
\]
are $(\mathscr{F}_t)$-martingales. Besides, it follows from the
equality of
laws that
%
%\begin{equation}
%e4.8 #&#
%e4.9 #&#
%e4.10 #&#
\begin{eqnarray}
&&\quad\tilde{\mathbb{E}} \gamma \bigl(\varrho_s
\tilde{u}^n,\varrho _s\tilde{W}^n \bigr)
\bigl[\tilde{M}^n(t)-\tilde{M}^n(s) \bigr]
\nonumber
\\[-8pt]
\label{exp1}
\\[-8pt]
\nonumber
&&\quad\qquad =\mathbb{E} \gamma \bigl(\varrho_s u^n,
\varrho_s W \bigr) \bigl[M^n(t)-M^n(s)
\bigr]=0,
\\
&&\quad\tilde{\mathbb{E}} \gamma \bigl(\varrho_s
\tilde{u}^n,\varrho _s\tilde{W}^n \bigr)
\biggl[\bigl(\tilde{M}^n\bigr)^2(t)-\bigl(
\tilde{M}^n\bigr)^2(s)-\sum_{k\geq1}
\int_s^t \bigl\langle g_k^n
\bigl(\tilde{u}^n\bigr),\varphi \bigr\rangle^2\,
\mathrm{d}r \biggr]
\nonumber
\\
\label{exp2}
&&\quad\qquad=\mathbb{E} \gamma \bigl(\varrho_s u^n,
\varrho_s W \bigr) \biggl[\bigl(M^n\bigr)^2(t)-
\bigl(M^n\bigr)^2(s)-\sum_{k\geq1}
\int_s^t \bigl\langle g_k^n
\bigl(u^n\bigr),\varphi \bigr\rangle^2\,\mathrm{d}r
\biggr]\\
\nonumber
&&\quad\qquad=0,
\\
&&\quad\tilde{\mathbb{E}} \gamma \bigl(\varrho_s
\tilde{u}^n,\varrho _s\tilde{W}^n \bigr)
\biggl[\tilde{M}^n(t)\tilde{\beta}_k^n(t)-
\tilde {M}^n(s)\tilde{\beta}_k^n(s)-\int
_s^t \bigl\langle g_k^n
\bigl(\tilde {u}^n\bigr),\varphi \bigr\rangle\,\mathrm{d}r \biggr]
\nonumber
\\
\label{exp3}
&&\quad\qquad=\mathbb{E} \gamma \bigl(\varrho_s u^n,
\varrho_s W \bigr) \biggl[M^n(t)\beta_k(t)-M^n(s)
\beta_k(s)-\int_s^t \bigl\langle
g_k^n\bigl(u^n\bigr),\varphi \bigr\rangle
\,\mathrm{d}r \biggr]\\
\nonumber
&&\quad\qquad=0.
\end{eqnarray}
%
%\end{equation}
%
Moreover,\vspace*{1pt} since the coefficients $B, \bar{A}, \sum_{k\geq1}g_k$
have linear growth, we can pass to the limit in \eqref{exp1}--\eqref
{exp3} due to \eqref{eqenergy} and the Vitali convergence theorem. We obtain
%
%\[
\begin{eqnarray*}
\tilde{\mathbb{E}} \gamma (\varrho_s\tilde{u},\varrho
_s\tilde {W} ) \bigl[\tilde{M}(t)-\tilde{M}(s) \bigr] &=& 0,
\\
\tilde{\mathbb{E}} \gamma (\varrho_s\tilde{u},\varrho
_s\tilde {W} ) \biggl[\tilde{M}^2(t)-
\tilde{M}^2(s)-\sum_{k\geq1}\int
_s^t \bigl\langle g_k(\tilde{u}),
\varphi \bigr\rangle^2\,\mathrm {d}r \biggr] &=& 0,
\\
\tilde{\mathbb{E}} \gamma (\varrho_s\tilde{u},\varrho
_s\tilde {W} ) \biggl[\tilde{M}(t)\tilde{\beta}_k(t)-
\tilde{M}(s)\tilde {\beta}_k(s)-\int_s^t
\bigl\langle g_k(\tilde{u}),\varphi \bigr\rangle\,\mathrm{d}r
\biggr] &=& 0,
\end{eqnarray*}
%
%\]
%
which gives the $(\tilde{\mathscr{F}}_t)$-martingale property.
\end{pf}

\begin{pf*}{Proof of Proposition \protect\ref{propmartsol}}
Once the above lemmas established, we infer that
\[
\biggl\langle \!\biggl\langle\tilde{M}-\int_0^\cdot
\bigl\langle \Phi(\tilde{u})\,\mathrm{d}\tilde{W},\varphi \bigr\rangle \biggr\rangle\! \biggr\rangle=0,
\]
where $\langle \!\langle \cdot \rangle \!\rangle$ denotes the
quadratic variation process. Accordingly,
%
%\[
\begin{eqnarray*}
%\label{reseni}
\bigl\langle\tilde{u}(t),\varphi \bigr\rangle &=&  \langle
u_0,\varphi \rangle-\int_0^t
\bigl\langle\operatorname{div} \bigl(B(\tilde {u}) \bigr),\varphi \bigr\rangle\,
\mathrm{d}s+\int_0^t \bigl\langle
\operatorname{div} \bigl(A(\tilde{u})\nabla\tilde{u} \bigr),\varphi \bigr\rangle\,
\mathrm{d} s
\\
&&{} +\int_0^t \bigl\langle\Phi(\tilde{u})\,
\mathrm{d}\tilde {W},\varphi \bigr\rangle,\qquad t\in[0,T], \tilde{\mathbb {P}}\mbox{-a.s.},
\end{eqnarray*}
%%
%\]
%%
and the proof is complete.
\end{pf*}

%s4.3 #&#
\subsection{Pathwise solutions}
\label{subsecpathwise}

As a consequence of pathwise uniqueness established in Section~\ref{seccomparison} and existence of a martingale solution that follows
from the previous subsection, we conclude from the Gy\"{o}ngy--Krylov
characterization of convergence in probability that the original
sequence $u^n$ defined on the initial probability space $(\Omega
,\mathscr{F}
,\mathbb{P})$ converges in probability in the topology of $\mathcal{X}_u$
to a random variable $u$ which is a weak solution to \eqref{eq}
provided (H1), (H2) and (H3) are fulfilled. For further details on this
method, we refer the reader to \cite{hof}, Section~4.5.

Moreover, it follows from Proposition~\ref{propenergy1} that
\[
u\in L^2\bigl(\Omega;L^\infty\bigl(0,T;L^2\bigl(
\mathbb{T}^N\bigr)\bigr)\bigr)\cap L^2\bigl(\Omega
;L^2\bigl(0,T;H^1\bigl(\mathbb{T}^N\bigr)
\bigr)\bigr)
\]
and one can also establish continuity of its trajectories in
$L^2(\mathbb{T}
^N)$. Toward this end, we observe that the solution to
%
%\[
\begin{eqnarray*}
\mathrm{d}z &=& \Delta z\, \mathrm{d}t+\Phi(u)\,\mathrm{d}W,
\\
z(0)&=& u_0,
\end{eqnarray*}
%
%\]
%
belongs to $L^2(\Omega;C([0,T];L^2(\mathbb{T}^N)))$. Setting
$r=u-z$, we obtain
%
%\[
\begin{eqnarray*}
\partial_t r &=& \Delta r-\operatorname{div} \bigl(B(u) \bigr)+
\operatorname{div} \bigl(\bigl(A(u)-\mathrm{I}\bigr)\nabla u \bigr),
\\
r(0)&=& 0,
\end{eqnarray*}
%
%\]
%
hence it follows by semigroup arguments that $r\in C([0,T];
L^2(\mathbb{T}^N))$ a.s. and, therefore,
\[
u\in L^2\bigl(\Omega;C\bigl([0,T];L^2\bigl(
\mathbb{T}^N\bigr)\bigr)\bigr)\cap L^2\bigl(\Omega
;L^2\bigl(0,T;H^1\bigl(\mathbb{T}^N\bigr)
\bigr)\bigr).
\]

%s5 #&#
\section{Existence for nondegenerate case---polynomial growth of $B$}
\label{secnondeg1}

In this section, we relax the additional hypothesis upon $B$ and prove
existence of a weak solution to \eqref{eq}
%\begin{equation}\label{eqnondeg2}
%\begin{eqnarray}
%\dif u+\diver\big(B(u)\big)\,\dif t&=\diver\big(A(u)\nabla u\big)\,
%\dif t+\Phi(u)\,\dif W,\\
%u(0)&=u_0,
%\end{eqnarray}
%\end{equation}
under the remaining two additional hypotheses of Section~\ref{secnondeg}, that is, (H1) and (H2).
%\begin{enumerate}
%\item$u_0\in L^p(\Omega;C^5(\mt^N))$, for all $p\in[1,\infty)$,
%\item the diffusion matrix $A$ is positive definite, i.e. $A\geq\tau
%\mathrm{I}$.
%\end{enumerate}

First, we approximate \eqref{eq} by
%
%\begin{equation}
%e5.1 #&#
\begin{eqnarray}
\mathrm{d}u+\operatorname{div} \bigl(B^{R}(u) \bigr)\,
\mathrm {d}t&=& \operatorname{div} \bigl(A(u)\nabla u \bigr)\,\mathrm{d}t+\Phi(u)
\,\mathrm{d}W,
\nonumber
\\[-8pt]
\label{eqapprox2}
\\[-8pt]
\nonumber
u(0)&=& u_0,
\end{eqnarray}
%
%\end{equation}
%
where $B^R$ is a truncation of $B$. According to the previous section,
for all $R\in\mathbb{N}$ there exists a unique weak solution to
\eqref{eqapprox2} such that, for all $p\in[2,\infty)$,
\[
%\label{eqenergy}
%\begin{eqnarray}
\mathbb{E}\sup_{0\leq t\leq T}\bigl\|u^R(t)
\bigr\|_{L^2(\mathbb{T}^N)}^p+2\tau \mathbb{E} \int_0^T
\bigl\|\nabla u^{R}\bigr\|^2_{L^2(\mathbb{T}^N)}\,\mathrm{d}s\leq C
\bigl(1+\mathbb{E} \|u_0\|_{L^2(\mathbb{T}^N)}^p \bigr),
%\end{eqnarray}
%
\]
where the constant $C$ is independent of $R$ and $\tau$. Furthermore,
we can also obtain a uniform estimate of the $L^p(\mathbb{T}^N)$-norm
that is
necessary in order to deal with coefficients having polynomial growth.

%pr5.1 #&#
\begin{prop}\label{propenergy2}
For all $p\in[2,\infty)$, the solution to \eqref{eqapprox2}
satisfies the following estimate:
%
%\begin{equation}
%e5.2 #&#
\begin{equation}
\label{eqenergy222}
\mathbb{E}\sup_{0\leq t\leq T}\bigl\|u^R(t)
\bigr\|_{L^p(\mathbb{T}^N)}^p\leq C \bigl(1+\mathbb{E}\|u_0
\|_{L^p(\mathbb{T}^N)}^p \bigr),
\end{equation}
%
%\end{equation}
%
where the constant $C$ does not depend on $R$ and $\tau$.
\end{prop}

%The set $\{u^R;\,R>0\}$ is bounded in $L^p(\Omega;C([0,T];L^p(
%\mt^N)))$ for all $p\in[2,\infty)$. {\color{red} uniformly in $\tau$}

\begin{pf}
As the generalized It\^o formula \eqref{ito} cannot be applied
directly to $\varphi(\xi)=|\xi|^p$, $p\in[2,\infty)$, and $\psi
(x)=1$, we follow the approach of \cite{denis1} and introduce
functions $\varphi_n\in C^2(\mathbb{R})$ that approximate $\varphi$ and
have quadratic growth at infinity as required by Proposition~\ref{propito}. Namely, let
\[
\,\varphi_n(\xi)=
\cases{ \displaystyle |
\xi|^p,& \hspace*{9pt}$|\xi|\leq n$,$\!\!$\vspace*{3pt}
\cr
\displaystyle
n^{p-2} \biggl[\frac{p(p-1)}{2}\xi^2-p(p-2)n|\xi|+
\frac
{(p-1)(p-2)}{2}n^2 \biggr], & \hspace*{9pt}$|\xi|>n$.$\!\!$}
\]
It is now easy to see that
%
%\begin{equation}
%e5.3 #&#
\begin{eqnarray}
 \bigl|\xi\varphi'_n(\xi)\bigr|&\leq &  p
\varphi_n(\xi),\nonumber
\\
\bigl|\varphi'_n(\xi)\bigr|&\leq &  p \bigl(1+
\varphi_n(\xi)\bigr),\nonumber
\\
\label{odhad}
\bigl|\varphi'_n(\xi)\bigr|&\leq & |\xi|\varphi''_n(
\xi),
\\
\nonumber
\xi^2\varphi''_n(\xi)&\leq &
p(p-1)\varphi_n(\xi),
\\
\nonumber
\varphi''(\xi)&\leq &  p(p-1) \bigl(1+
\varphi_n(\xi)\bigr)
\end{eqnarray}
%
%\end{equation}
%
hold true for all $\xi\in\mathbb{R}, n\in\mathbb{N}, p\in
[2,\infty)$.
Then by Proposition~\ref{propito}
%
%\[
%e5.4 #&#
%e5.5 #&#
%e5.6 #&#
%e5.7 #&#
\begin{eqnarray*}
\int_{\mathbb{T}^N}\varphi_n\bigl(u^R(t)
\bigr)\,\mathrm{d}x&=& \int_{\mathbb
{T}^N}\varphi _n(u_0)
\,\mathrm{d}x-\int_0^t\!\int_{\mathbb{T}^N}
\varphi '_n\bigl(u^R\bigr)
\operatorname{div} \bigl(B^R\bigl(u^R\bigr) \bigr)\,
\mathrm{d}x\,\mathrm{d}s
\\
&&{}+\int_0^t\!\int_{\mathbb{T}^N}
\varphi'_n\bigl(u^R\bigr)\operatorname
{div} \bigl(A\bigl(u^R\bigr)\nabla u^R \bigr)\,
\mathrm{d}x \,\mathrm{d}s
\\
%&\quad+\tau\int_0^t\int_{\mt^N}\varphi'_n(u^R)\Delta u^R\,\dif x \,
%\dif s\\
&&{}+\sum_{k\geq1} \int
_0^t\!\int_{\mathbb{T}^N}\varphi
'_n\bigl(u^R\bigr)g_k
\bigl(u^R\bigr)\,\mathrm{d}x\,\mathrm{d}\beta_k(s)
\\
&&{}+\frac{1}{2}\int_0^t\!\int
_{\mathbb{T}^N}\varphi ''_n
\bigl(u^R\bigr)G^2\bigl(u^R\bigr)\,
\mathrm{d}x\,\mathrm{d}s.
\end{eqnarray*}
%
%\]
%
Setting $H(\xi)=\int_0^\xi\varphi''_n(\zeta)B^R(\zeta)\,\mathrm{d}
\zeta\,$ it can be seen that the second term on the right-hand side
vanishes due to the boundary conditions.
The third term is nonpositive as the matrix $A$ is positive definite
\[
%\begin{eqnarray}
\int_0^t\!\int_{\mathbb{T}^N}
\varphi'_n\bigl(u^R\bigr)
\operatorname{div} \bigl(A\bigl(u^R\bigr)\nabla u^R
\bigr)\,\mathrm{d}x \,\mathrm{d}s%+\tau\int_0^t\int_{\mt^N}
%\varphi'_n(u^R)
%\Delta u^R\,\dif x \,\dif s\\
=-\int
_0^t\!\int_{\mathbb{T}^N}
\varphi''_n\bigl(u^R\bigr)\bigl|
\sigma\bigl(u^R\bigr)\nabla u^R\bigr|^2\,
\mathrm{d}x\,\mathrm{d}s. %\leq-\tau\int_0^t\int_{\mt^N}\varphi''_n(u^R)|\nabla u^R|^2\,\dif x\,
%\dif s.
%\end{eqnarray}
%
\]
The last term is estimated by \eqref{odhad}
%
%\[
%e5.8 #&#
%e5.9 #&#
\begin{eqnarray*}
\frac{1}{2}\int_0^t\!\int
_{\mathbb{T}^N}\varphi''_n
\bigl(u^R\bigr)G^2\bigl(u^R\bigr)\,
\mathrm{d}x\, \mathrm{d}s&\leq & \frac{C}{2}\int_0^t
\!\int_{\mathbb{T}^N}\varphi ''_n
\bigl(u^R\bigr) \bigl(1+\bigl|u^R\bigr|^2\bigr)\,
\mathrm{d}x\,\mathrm{d}s
\\
&\leq & \frac{Cp(p-1)}{2}\int_0^t\!\int
_{\mathbb{T}^N}\bigl(1+\varphi _n\bigl(u^R
\bigr)\bigr)\, \mathrm{d}x\,\mathrm{d}s,
\end{eqnarray*}
%
%\]
%
and, therefore, by Gronwall's lemma we obtain
%
%\begin{equation}
%e5.10 #&#
\begin{equation}
\label{energy111}
\mathbb{E}\int_{\mathbb{T}^N}\varphi_n
\bigl(u^R(t)\bigr)\,\mathrm{d}x\leq C \biggl(1+\mathbb{E} \int
_{\mathbb{T}^N}\varphi(u_0)\,\mathrm{d}x \biggr).
\end{equation}
%
%\end{equation}

As a consequence, a uniform estimate of $\mathbb{E}\sup_{0\leq t\leq
T}\|
u^R(t)\|_{L^p(\mathbb{T}^N)}^p$ follows. Indeed, we proceed similarly as
before only for the stochastic term we apply the Burkholder--Davis--Gundy and the Schwartz inequality, \eqref{odhad} and
the weighted Young inequality
%
%\[
\begin{eqnarray*}
&& \mathbb{E} \sup_{0\leq t\leq T} \biggl|\sum_{k\geq1}
\int_0^t\!\int_{\mathbb{T}
^N}
\varphi'_n\bigl(u^R\bigr) g_k
\bigl(u^R\bigr)\,\mathrm{d}x\,\mathrm{d}\beta_k(s) \biggr|
\\
&&\qquad \leq C \mathbb{E} \biggl(\int_0^T\sum
_{k\geq1} \biggl(\int_{\mathbb{T}
^N}\bigl|
\varphi'_n\bigl(u^R\bigr)\bigr|
\bigl|g_k\bigl(u^R\bigr)\bigr| \,\mathrm{d}x
\biggr)^2\,\mathrm {d}s \biggr)^{{1}/{2}}
\\
&&\qquad \leq C \mathbb{E} \biggl(\int_0^T \bigl\| \bigl|
\varphi'_n\bigl(u^R\bigr)\bigr|^{{1}/{2}}\bigl|u^R\bigr|^{{1}/{2}} \bigr\|_{L^2(\mathbb{T}^N)}^2
\\
&&\hspace*{26pt}\qquad\qquad{}\times\sum_{k\geq
1} \bigl\| \bigl|\varphi'_n
\bigl(u^R\bigr)\bigr|^{{1}/{2}}\bigl|u^R\bigr|^{-{1}/{2}}\bigl|g_k
\bigl(u^R\bigr)\bigr| \bigr\| _{L^2(\mathbb{T}^N)}^2\,\mathrm{d}s
\biggr)^{{1}/{2}}
\\
&&\qquad \leq C \mathbb{E} \biggl(\int_0^T\!\!\int
_{\mathbb{T}^N}\varphi _n\bigl(u^R\bigr)\,
\mathrm{d} x \biggl(1+\int_{\mathbb{T}^N}\varphi_n
\bigl(u^R\bigr)\,\mathrm{d}x \biggr)\, \mathrm{d}s
\biggr)^{{1}/{2}}
\\
&&\qquad\leq C \mathbb{E} \biggl(\sup_{0\leq t\leq T}\int_{\mathbb
{T}^N}
\varphi _n\bigl(u^R\bigr)\,\mathrm{d}x
\biggr)^{{1}/{2}} \biggl(1+\int_0^T\!\!\int
_{\mathbb{T}
^N}\varphi_n\bigl(u^R\bigr)\,
\mathrm{d}x\,\mathrm{d}s \biggr)^{{1}/{2}}
\\
&&\qquad \leq\frac{1}{2} \mathbb{E}\sup_{0\leq t\leq T}\int
_{\mathbb
{T}^N}\varphi _n\bigl(u^R\bigr)\,
\mathrm{d}x+C \biggl(1+\int_0^T\mathbb{E}\int
_{\mathbb
{T}^N}\varphi _n\bigl(u^R\bigr)\,
\mathrm{d}x\,\mathrm{d}s \biggr)
\end{eqnarray*}
%
%\]
%
which together with \eqref{energy111} and Fatou's lemma yields \eqref
{eqenergy222}.
\end{pf}

Having Proposition~\ref{propenergy2} in hand, the proof of
Propositions \ref{propholder} as well as all the proofs in
Sections~\ref{subseccompact}, \ref{subsecidentif}, \ref
{subsecpathwise} can be repeated with only minor modifications and,
consequently, the following result deduced.

%th5.2 #&#
\begin{thm}
Under the additional hypotheses \textup{(H1)}, \textup{(H2)},
there exists a
unique weak solution to \eqref{eq} such that, for all $p\in[2,\infty)$,
%
%e5.11 #&#
\begin{eqnarray}
&& \mathbb{E}\sup_{0\leq t\leq T}\bigl\|u(t)\bigr\|_{L^p(\mathbb
{T}^N)}^p+p(p-1)
\mathbb{E}\int_0^T\!\!\int_{\mathbb{T}^N}|u|^{p-2}\bigl|
\sigma(u)\nabla u\bigr|^2
\nonumber
\\[-8pt]
\label{eqenergyfinal}
\\[-8pt]
\nonumber
&&\qquad\leq C \bigl(1+\mathbb{E}\| u_0
\|_{L^p(\mathbb{T}^N)}^p \bigr)
\end{eqnarray}
and the constant $C$ is independent of $\tau$.
\end{thm}

\begin{pf*}{Sketch of the proof}
Following the approach of the previous section, we obtain:
\begin{longlist}[$\!\!$(iii)]
\item[$\!\!$(i)] For all $\lambda\in(0,1/2)$ there exists $C>0$ such that for
all $R\in\mathbb{N}$
\[
\mathbb{E}\bigl\|u^R\bigr\|_{C^\lambda([0,T];H^{-1}(\mathbb{T}^N))}\leq C.
\]
\item[$\!\!$(ii)] The laws of $\{u^R; R\in\mathbb{N}\}$ form a tight sequence on
\[
L^2\bigl(0,T;L^2\bigl(\mathbb{T}^N\bigr)
\bigr)\cap C\bigl([0,T];H^{-2}\bigl(\mathbb{T}^N\bigr)
\bigr).
\]
\item[$\!\!$(iii)] There exists $ ((\tilde{\Omega},\tilde{\mathscr
{F}},(\tilde{\mathscr{F}
}_t),\tilde{\mathbb{P}}),\tilde{W},\tilde{u} )$ that is a weak
martingale solution to~\eqref{eq}.
\item[$\!\!$(iv)] There exists $u\in L^2(\Omega;C([0,T];L^2(\mathbb{T}^N)))\cap
L^p(\Omega;L^\infty(0,T;L^p(\mathbb{T}^N)))\cap  L^2(\Omega
;L^2(0,T;H^1(\mathbb{T}^N)))$ that is a weak solution to \eqref{eq}.
\item[$\!\!$(v)] By the approach of Proposition~\ref{propenergy2}, we obtain
\eqref{eqenergyfinal}.\quad\qed
\end{longlist}
\noqed\end{pf*}
%

%s6 #&#
\section{Existence for degenerate case---smooth initial data}
\label{secdeg}

As the next step in the existence proof of Theorem~\ref{thmmain}, we
can finally proceed to the degenerate case. Throughout this section, we
only assume the additional hypothesis upon the initial condition, that
is, (H1).
%\begin{enumerate}
%\item[(H1)] $u_0\in L^p(\Omega;C^5(\mt^N))$, for all $p\in[1,\infty)$.
%\end{enumerate}

Consider the following nondegenerate approximations of \eqref{eq}:
%
%e6.1 #&#
\begin{eqnarray}
\mathrm{d}u +\operatorname{div} \bigl(B(u) \bigr)\,\mathrm {d}t&=&
\operatorname{div} \bigl(A(u)\nabla u \bigr)\,\mathrm{d}t+\tau\Delta u\,
\mathrm{d}t+\Phi(u)\,\mathrm{d}W,
\nonumber
\\[-8pt]
\label{eqnondeg1}
\\[-8pt]
\nonumber
u(0)&=& u_0.
\end{eqnarray}
According\vspace*{1pt} to the results of Section~\ref{secnondeg1}, we have for any
fixed $\tau>0$ the existence of $u^\tau\in L^2(\Omega
;C([0,T];L^2(\mathbb{T}^N))) \cap L^2(\Omega;L^2(0,T;H^1(\mathbb
{T}^N)))$ which is
a weak solution to \eqref{eqnondeg1} and satisfies [cf. \eqref{eqenergyfinal}]
%
%e6.2 #&#
\begin{eqnarray}
&& \mathbb{E}\sup_{0\leq t\leq T}\bigl\|u^\tau(t)
\bigr\|_{L^p(\mathbb
{T}^N)}^p\nonumber\\
&&\label{energy}\quad{}+p(p-1) \mathbb{E}\int_0^T
\!\!\int_{\mathbb{T}^N}\bigl|u^\tau\bigr|^{p-2} \bigl(\bigl|\sigma
\bigl(u^\tau \bigr)\nabla u^\tau\bigr|^2+\tau\bigl|\nabla
u^\tau\bigr|^2 \bigr)\,\mathrm{d}x\,\mathrm {d}t
\\
\nonumber
&&\qquad\leq C \bigl(1+\mathbb{E}\|u_0\|_{L^p(\mathbb
{T}^N)}^p
\bigr)
\end{eqnarray}
with a constant that does not depend on $\tau$. As the next step, we
employ the technique of Section~\ref{subsecformulation} to derive
the kinetic formulation that is satisfied by $f^\tau=\mathbf
{1}_{u^\tau
>\xi}$ in the sense of $\mathcal{D}'(\mathbb{T}^N\times\mathbb
{R})$. It reads as follows:
%
%e6.3 #&#
\begin{eqnarray}
&&\mathrm{d}f^\tau+b\cdot\nabla f^\tau
\,\mathrm{d}t-A\dvtx  \mathrm{D}^2 f^\tau\,\mathrm{d} t-\tau\Delta
f^\tau\,\mathrm{d}t
\nonumber
\\[-8pt]
\label{eqkinapprox}
\\[-8pt]
\nonumber
&&\qquad =\delta_{u^\tau=\xi}\Phi\,\mathrm{d}W+\partial_\xi
\bigl(n^\tau _1+n^\tau_2-
\tfrac{1}{2}G^2\delta_{u^\tau=\xi} \bigr)\,\mathrm{d}t,
\end{eqnarray}
where
%
%e6.4 #&#
%e6.5 #&#
\begin{eqnarray*}
\mathrm{d}n^\tau_1(t,x,\xi)&=& \bigl|\sigma\nabla
u^\tau\bigr|^2\,\mathrm {d}\delta _{u^\tau=\xi}\,\mathrm{d}x
\,\mathrm{d}t,
\\
\mathrm{d}n^\tau_2(t,x,\xi)&=& \tau\bigl|\nabla
u^\tau\bigr|^2\,\mathrm {d}\delta_{u^\tau
=\xi}\,\mathrm{d}x\,
\mathrm{d}t.
\end{eqnarray*}
\subsection{Uniform estimates}

Next, we prove a uniform $W^{\lambda,1}(\mathbb{T}^N)$-regularity of the
approximate solutions $u^\tau$. Toward this end, we make use of two
seminorms describing the $W^{\lambda,1}$-regularity of a function
$u\in L^1(\mathbb{T}^N)$ (see \cite{debus}, Section~3.4,  for further
details). Let $\lambda\in(0,1)$ and define
\begin{eqnarray*}
p^{\lambda}(u) &=& \int_{\mathbb{T}^N}\!\int_{\mathbb{T}^N}
\frac
{|u(x)-u(y)|}{|x-y|^{N+\lambda}}\,\mathrm{d}x\,\mathrm{d}y,
\\
p^\lambda_\varrho(u) &=& \sup_{0<\varepsilon<2 D_N}
\frac
{1}{\varepsilon^\lambda}\int_{\mathbb{T}^N}\!\int_{\mathbb{T}
^N}\bigl|u(x)-u(y)\bigr|
\varrho_\varepsilon(x-y)\,\mathrm{d}x\,\mathrm{d}y,
\end{eqnarray*}
where $(\varrho_\varepsilon)$ is the approximation to the identity on
$\mathbb{T}^N$ that is radial, that is, $\varrho_\varepsilon
(x)={1}/{\varepsilon^N}\varrho({|x|}/{\varepsilon})$; and by $D_N$
we denote the diameter of $[0,1]^N$. The fractional Sobolev space
$W^{\lambda,1}(\mathbb{T}^N)$ is defined as a subspace of
$L^1(\mathbb{T}^N)$ with
finite norm
\[
\|u\|_{W^{\lambda,1}(\mathbb{T}^N)}=\|u\|_{L^1(\mathbb
{T}^N)}+p^\lambda(u).
\]
According to \cite{debus}, the following relations holds true between
these seminorms. Let $s\in(0,\lambda)$, there exists a constant
$C=C_{\lambda,\varrho,N}$ such that for all $u\in L^1(\mathbb{T}^N)$
\[
p^\lambda_\varrho(u)\leq C p^\lambda(u),\qquad
p^s(u)\leq\frac
{C}{\lambda-s} p^\lambda_\varrho(u).
\]

\begin{prop}[($W^{\varsigma,1}$-regularity)]\label{propspatial}
Set $\varsigma=\min \{\frac{2\gamma-1}{\gamma+1},\frac
{2\alpha}{\alpha+1} \}$, where $\gamma$ was defined in \eqref
{sigma} and $\alpha$ in \eqref{fceh}. Then for all $s\in(0,\varsigma
)$ there exists a constant $C_{s}>0$ such that for all $t\in[0,T]$ and
all $\tau\in(0,1)$
\[
%\label{spatial}
\mathbb{E} p^{s} \bigl(u^\tau(t) \bigr)\leq
C_{s} \bigl(1+\mathbb{E} p^{\varsigma}(u_0) \bigr).
\]
In particular, there exists a constant $C_{s}>0$ such that for all
$t\in[0,T]$
\[
\mathbb{E}\bigl\|u^\tau(t)\bigr\|_{W^{s,1}(\mathbb{T}^N)}\leq C_{s} \bigl(1+
\mathbb{E}\|u_0\| _{W^{\varsigma,1}(\mathbb{T}^N)} \bigr).
\]
\end{prop}

\begin{pf}
Proof of this statement is based on Proposition~\ref{propdoubling}.
We have
%
%\[
\begin{eqnarray*}
&& \mathbb{E}\int_{(\mathbb{T}^N)^2}\!\int_\mathbb{R}
\varrho _\varepsilon(x-y)f^\tau (x,t,\xi)\bar{f}^\tau(y,t,
\xi)\,\mathrm{d}\xi\,\mathrm{d}x\, \mathrm{d}y
\\
&&\qquad \leq\mathbb{E}\int_{(\mathbb{T}^N)^2}\!\int_{\mathbb{R}^2}
\varrho _\varepsilon (x-y)\psi_\delta(\xi-\zeta)f^\tau(x,t,
\xi)\bar{f}^\tau (y,t,\zeta)\,\mathrm{d}\xi\,\mathrm{d}\zeta\,
\mathrm{d}x\,\mathrm {d}y+\delta
\\
&&\qquad\leq\mathbb{E} \int_{(\mathbb{T}^N)^2}\!\int_{\mathbb{R}^2}
\varrho_\varepsilon (x-y)\psi_\delta(\xi-\zeta)f_{0}(x,\xi)
\bar{f}_{0}(y,\zeta )\,\mathrm{d} \xi\,\mathrm{d}\zeta\,\mathrm{d}x
\,\mathrm{d}y\\
&&\qquad\quad{}+\delta+\mathrm {I}+\mathrm{J}+\mathrm {J}^\tau+\mathrm{K}
\\
&&\qquad \leq\mathbb{E}\int_{(\mathbb{T}^N)^2}\!\int_\mathbb{R}
\varrho _\varepsilon (x-y)f_{0}(x,\xi)\bar{f}_{0}(y,
\xi)\,\mathrm{d}\xi\,\mathrm{d}x\, \mathrm{d} y+2\delta+\mathrm{I}+\mathrm{J}+
\mathrm{J}^\tau+\mathrm{K},
\end{eqnarray*}
%
%\]
%
where $\mathrm{I}, \mathrm{J}, \mathrm{K}$ are defined similarly
to Proposition~\ref{propdoubling}, $\mathrm{J}^\tau$ corresponds to
the second- order term $\tau\Delta u^\tau$:
%
%\[
\begin{eqnarray*}
\mathrm{J}^\tau &=& 2\tau \mathbb{E}\int_0^t
\!\int_{(\mathbb
{T}^N)^2}\!\int_{\mathbb{R}
^2}f^\tau
\bar{f} ^\tau\Delta_x\varrho_\varepsilon(x-y)
\psi_\delta(\xi-\zeta)\, \mathrm{d}\xi\,\mathrm{d}\zeta\,\mathrm{d}x\,
\mathrm{d}y\, \mathrm{d}s
\\
&&{}-\mathbb{E}\int_0^t\!\int
_{(\mathbb{T}^N)^2}\!\int_{\mathbb
{R}^2}\varrho
_\varepsilon(x-y)\psi_\delta(\xi-\zeta)\, \mathrm{d}\nu^ {\tau}
_ { x , s } (\xi)\,\mathrm{d}x\,\mathrm{d}n_{2}^\tau(y,s,
\zeta)
\\
&&{}-\mathbb{E}\int_0^t\!\int
_{(\mathbb{T}^N)^2}\!\int_{\mathbb
{R}^2}\varrho
_\varepsilon(x-y)\psi_\delta(\xi-\zeta)\, \mathrm{d}
\nu^{\tau}_{y,s} (\zeta)\,\mathrm{d}y\,\mathrm{d}n_{2}^\tau(x,s,
\xi)
\\
&=& -\tau\mathbb{E}\int_0^t\!\int
_{(\mathbb{T}^N)^2}\varrho _\varepsilon (x-y)\psi_\delta
\bigl(u^\tau(x)-u^\tau(y) \bigr) \bigl|\nabla_x
u^\tau-\nabla_y u^\tau \bigr|^2\,
\mathrm{d}x\,\mathrm{d}y\,\mathrm {d}s\\
&\leq & 0
\end{eqnarray*}
%
%\]
%
and the error term $\delta$ was obtained as follows:
%
%\begin{equation}
%e6.6 #&#
\begin{eqnarray}
 &&  \bigg|\mathbb{E}\int_{(\mathbb{T}^N)^2}\!\int
_\mathbb{R}\varrho _\varepsilon(x-y) f^\tau(x,t,
\xi)\bar{f}^\tau(y,t,\xi)\,\mathrm{d}\xi\,\mathrm {d}x\,\mathrm{d}y
\nonumber\\
&&\quad{}-\mathbb{E}\int_{(\mathbb{T}^N)^2}\!\int_{\mathbb
{R}^2}
\varrho_\varepsilon (x-y)\psi_\delta(\xi-\zeta)f^\tau(x,t,
\xi)\bar{f}^\tau (y,t,\zeta)\,\mathrm{d}\xi\,\mathrm{d}\zeta\,
\mathrm{d}x\,\mathrm{d}y \bigg|
\nonumber\\
&&\qquad= \biggl|\mathbb{E}\int_{(\mathbb{T}^N)^2}\varrho _\varepsilon(x-y)
\! \int_\mathbb{R}\mathbf{1}_{u^\tau(x)>\xi}\int
_{\mathbb{R}}\psi _\delta(\xi-\zeta )\nonumber\\
&&\hspace*{148pt}\qquad\quad{}\times [\mathbf{1}_{u^\tau(y)\leq\xi}-
\mathbf{1}_{u^\tau(y)\leq
\zeta} ]\,\mathrm{d}\zeta\,\mathrm{d}\xi\,\mathrm{d}x\,\mathrm{d}y \biggr|
\nonumber
\\
\nonumber
&&\qquad\leq\mathbb{E}\int_{(\mathbb{T}^N)^2}\!\int_\mathbb
{R}\varrho_\varepsilon (x-y) \mathbf{1}_{u^\tau(x)>\xi}\int
_{\xi-\delta}^\xi\psi _\delta (\xi-\zeta)
\mathbf{1}_{\zeta<u^\tau(y)\leq\xi}\,\mathrm {d}\zeta\,\mathrm{d} \xi\,\mathrm{d}x\,
\mathrm{d}y
\\
\label{error}
&&\quad\qquad{}+\mathbb{E}\int_{(\mathbb{T}^N)^2}\!\int_\mathbb{R}
\varrho _\varepsilon(x-y) \mathbf{1}_{u^\tau(x)>\xi}\int_{\xi}^{\xi+\delta}
\psi_\delta (\xi -\zeta) \mathbf{1}_{\xi<u^\tau(y)\leq\zeta}\,\mathrm{d}\zeta\,
\mathrm{d}\xi\, \mathrm{d}x\,\mathrm{d}y
\\
\nonumber
&&\qquad\leq\frac{1}{2} \mathbb{E}\int_{(\mathbb{T}^N)^2}\varrho
_\varepsilon (x-y)\int_{u^\tau(y)}^{\min\{u^\tau(x),u^\tau(y)+\delta\}}\,\mathrm{d} \xi
\,\mathrm{d}x\,\mathrm{d}y
\\
\nonumber
&&\quad\qquad{}+\frac{1}{2} \mathbb{E}\int_{(\mathbb{T}^N)^2}\varrho
_\varepsilon (x-y)\int_{u^\tau(y)-\delta}^{\min\{u^\tau(x),u^\tau(y)\}}\,\mathrm{d} \xi
\,\mathrm{d}x\,\mathrm{d}y\leq\delta.
\end{eqnarray}
%
%\end{equation}
%
Hence, by the proof of Theorem~\ref{uniqueness}
%
%\[
\begin{eqnarray*}
&&\mathbb{E}\int_{(\mathbb{T}^N)^2}\varrho_\varepsilon(x-y)
\bigl|u^\tau (x,t)-u^\tau(y,t) \bigr| \,\mathrm{d}x\,\mathrm{d}y
\\
&&\qquad \leq\mathbb{E} \int_{(\mathbb{T}^N)^2} \varrho_\varepsilon (x-y)
\bigl|u_0(x)-u_0(y) \bigr|\, \mathrm{d}x\,\mathrm{d}y\\
&&\qquad\quad{}+C_T
\bigl(\delta +\delta \varepsilon^{-1}+\delta^2
\varepsilon^{-2}+\delta^{-1}\varepsilon ^2+
\delta^\alpha \bigr).
\end{eqnarray*}
%
%\]
%
By optimization in $\delta$, that is, setting $\delta=\varepsilon
^\beta$, we obtain
\[
\sup_{0<\tau<2D_N}\frac{C_T (\delta+\delta\varepsilon
^{-1}+\delta^{2\gamma}\varepsilon^{-2}+\delta^{-1}\varepsilon
^2+\delta^\alpha )}{\varepsilon^\varsigma}\leq C_T,
\]
where the maximal choice of the parameter $\varsigma$ is $\min \{
\frac{2\gamma-1}{\gamma+1},\frac{2\alpha}{\alpha+1} \}$.
% which corresponds to $\beta=\max\big\{\frac{1}{\alpha+1},\frac{1}{2}
%\big\}.$
As a consequence,
%
%\begin{equation}
%e6.7 #&#
\begin{equation}
\label{eqspatial}
\mathbb{E}\int_{(\mathbb{T}^N)^2}\varrho_\varepsilon(x-y)
\bigl|u^\tau (x,t)-u^\tau(y,t) \bigr|\, \mathrm{d}x\,\mathrm{d}y\leq
C_T\varepsilon ^\varsigma \bigl(1+\mathbb{E}
p^\varsigma(u_0) \bigr).\hspace*{-8pt}
\end{equation}
%
%\end{equation}
%
Finally, multiplying the above by $\varepsilon^{-1-s}$, $s\in
(0,\varsigma)$, and integrating with respect to $\varepsilon\in
(0,2D_N)$ gives the claim.
\end{pf}

%s6.2 #&#
\subsection{Strong convergence}
\label{subsecstrong}

According to \eqref{energy}, the set $\{u^\tau; \tau\in(0,1)\}$ is
bounded in $L^p(\Omega;L^p(0,T;L^p(\mathbb{T}^N)))$ and, therefore,
possesses a weakly convergent subsequence. The aim of this subsection
is to show that even strong convergence holds true. Toward this end, we
make use of the ideas developed in Section~\ref{seccomparison}.

%th6.2 #&#
\begin{thm}\label{thmconv}
There exists $u\in L^1(\Omega\times[0,T],\mathcal{P},\mathrm
{d}\mathbb{P}
\otimes\mathrm{d}t;L^1(\mathbb{T}^N))$ such that
\[
u^\tau\longrightarrow u\qquad \mbox{in } L^1\bigl(\Omega\times
[0,T],\mathcal{P},\mathrm{d}\mathbb{P}\otimes\mathrm {d}t;L^1\bigl(
\mathbb{T}^N\bigr)\bigr).
\]
\end{thm}

\begin{pf}
By similar techniques as in the proofs of Proposition~\ref
{propdoubling} and Theorem~\ref{uniqueness}, we obtain for any two
approximate solutions $u^\tau, u^\sigma$% and approximations to the
%identity $(\varrho_\varepsilon), (\psi_\delta)$ on $\mt^N$ and $\mr$,
%respectively,
%
%\begin{equation}
%e6.8 #&#
%e6.9 #&#
\begin{eqnarray}
&&\quad \mathbb{E}\int_{\mathbb{T}^N} \bigl(u^\tau(t)-u^\sigma(t)
\bigr)^+\, \mathrm{d} x\nonumber\\
\label{doubling1}
&&\quad\qquad= \mathbb{E}\int_{\mathbb{T}^N}\!\int
_\mathbb{R}f^\tau(x,t,\xi) \bar{f}^\sigma
(x,t,\xi)\,\mathrm{d}\xi\,\mathrm{d}x
\\
\nonumber
&&\quad\qquad = \mathbb{E}\int_{(\mathbb{T}^N)^2}\!\int_{\mathbb{R}^2}\varrho
_\varepsilon (x-y)\psi_\delta(\xi-\zeta)f^{\tau}(x,t,\xi)
\bar{f}^{\sigma
}(y,t,\zeta)\,\mathrm{d}\xi\,\mathrm{d}\zeta\,\mathrm{d}x\,
\mathrm{d}y\\
\nonumber
&&\quad\qquad\quad{}+\eta_t(\tau ,\sigma,\varepsilon,\delta).
%\\
%&\leq\tilde{\eta}_t(\tau,\sigma,\varepsilon,\delta).
%=\stred\int_{(\mt^N)^2}\int_{\mr^2}\varrho_\varepsilon(x-y)\psi_\delta(
%\xi-\zeta)f_{0}(x,\xi)\bar{f}_{0}(y,\zeta)\,\dif\xi\,\dif\zeta\,\dif x
%\,\dif y
\end{eqnarray}
%
%\end{equation}
%
(Here, $\varepsilon$ and $\delta$ are chosen arbitrarily and their
value will be fixed later.) The idea now is to show that the error term
$\eta_t(\tau,\sigma,\varepsilon,\delta)$ is in fact independent of
$\tau, \sigma$. Indeed, we have
%
%\[
\begin{eqnarray*}
&& \eta_t(\tau,\sigma,\varepsilon,\delta)\\
&&\qquad= \mathbb{E}\int
_{\mathbb
{T}^N}\!\int_\mathbb{R}f^\tau(x,t,
\xi)\bar{f}^\sigma(x,t,\xi)\,\mathrm{d}\xi \,\mathrm{d}x
\\
&&\qquad\quad{}-\mathbb{E}\int_{(\mathbb{T}^N)^2}\!\int_{\mathbb
{R}^2}
\varrho_\varepsilon (x-y)\psi_\delta(\xi-\zeta)f^\tau(x,t,
\xi)\bar{f}^\sigma (y,t,\zeta)\,\mathrm{d}\xi\,\mathrm{d}\zeta\,
\mathrm{d}x\,\mathrm {d}y
\\
&&\qquad= \biggl(\mathbb{E}\int_{\mathbb{T}^N}\!\int_\mathbb{R}f^\tau
(x,t,\xi)\bar {f}^\sigma(x,t,\xi)\,\mathrm{d}\xi\,\mathrm{d}x
\\
&&\qquad\hspace*{16pt}{}-\mathbb{E}\int_{(\mathbb{T}^N)^2}\!\int_{\mathbb{R}}
\varrho _\varepsilon (x-y)f^\tau(x,t,\xi)\bar{f}^\sigma(y,t,
\xi)\,\mathrm{d}\xi\, \mathrm{d}x\, \mathrm{d}y \biggr)
\\
&&\qquad\quad{}+ \biggl(\mathbb{E}\int_{(\mathbb{T}^N)^2}\!\int_\mathbb
{R}\varrho_\varepsilon (x-y) f^\tau(x,t,\xi)\bar{f}^\sigma(y,t,
\xi)\,\mathrm{d}\xi\, \mathrm{d}x\, \mathrm{d}y
\\
&&\qquad\quad\hspace*{19pt}{}-\mathbb{E}\int_{(\mathbb{T}^N)^2}\!\int_{\mathbb
{R}^2}
\varrho_\varepsilon (x-y)\psi_\delta(\xi-\zeta)\\
&&\hspace*{54pt}\qquad\qquad\qquad{}\times f^\tau(x,t,
\xi)\bar{f}^\sigma (y,t,\zeta)\,\mathrm{d}\xi\,\mathrm{d}\zeta\,
\mathrm{d}x\,\mathrm {d}y \biggr)
\\
&&\!\!\qquad=\mathrm{H}_1+\mathrm{H}_2,
\end{eqnarray*}
%
%\]
%
where
%
%\[
%e6.10 #&#
%e6.11 #&#
\begin{eqnarray*}
|\mathrm{H}_1|&= & \biggl|\mathbb{E}\int_{(\mathbb{T}^N)^2}\varrho
_\varepsilon (x-y)\int_\mathbb{R}\mathbf{1}_{u^\tau(x)>\xi}
[\mathbf {1}_{u^\sigma(x)\leq\xi
}-\mathbf{1}_{u^\sigma(y)\leq\xi} ]\,\mathrm{d}\xi\,\mathrm
{d}x\,\mathrm{d}y \biggr|
\\
&=& \biggl|\mathbb{E}\int_{(\mathbb{T}^N)^2}\varrho_\varepsilon (x-y)
\bigl(u^\sigma(y)-u^\sigma(x) \bigr)\,\mathrm{d}x\,\mathrm{d}y \biggr|
\\
&\leq & C \varepsilon^\varsigma
\end{eqnarray*}
%
%\]
%
due to \eqref{eqspatial} and $|H_2|\leq\delta$ due to \eqref{error}.
%\[
%\begin{eqnarray}
%|\mathrm{H}_2|&=\bigg|\stred\!\int_{(\mt^N)^2}\!\varrho_
%\varepsilon(x-y)\!\int_\mr\ind_{u^\tau(x)>\xi}\int_{\mr}\psi_\delta(
%\xi-\zeta)\Big[\ind_{u^\sigma(y)\leq\xi}-\ind_{u^\sigma(y)\leq\zeta}
%\Big]\dif\zeta\dif\xi\dif x\dif y\bigg|\\
%&\leq\stred\int_{(\mt^N)^2}\int_\mr\varrho_\varepsilon(x-y)\,\ind_{u^
%\tau(x)>\xi}\int_{\xi-\delta}^\xi\psi_\delta(\xi-\zeta)\,\ind_{\zeta<u^
%\sigma(y)\leq\xi}\,\dif\zeta\,\dif\xi\,\dif x\,\dif y\\
%&\quad+\stred\int_{(\mt^N)^2}\int_\mr\varrho_\varepsilon(x-y)\,\ind_{u^
%\tau(x)>\xi}\int_{\xi}^{\xi+\delta}\psi_\delta(\xi-\zeta)\,\ind_{\xi<u^
%\sigma(y)\leq\zeta}\,\dif\zeta\,\dif\xi\,\dif x\,\dif y\\
%&\leq\frac{1}{2}\,\stred\int_{(\mt^N)^2}\varrho_\varepsilon(x-y)
%\int_{u^\sigma(y)}^{\min\{u^\tau(x),u^\sigma(y)+\delta\}}\dif\xi\,
%\dif x\,\dif y\\
%&\quad+\frac{1}{2}\,\stred\int_{(\mt^N)^2}\varrho_\varepsilon(x-y)
%\int_{u^\sigma(y)-\delta}^{\min\{u^\tau(x),u^\sigma(y)\}}\dif\xi\,
%\dif x\,\dif y\leq\delta
%\end{eqnarray}
%\]
Therefore, the claim follows, that is, $|\eta_t(\tau,\sigma
,\varepsilon,\delta)|\leq C\varepsilon^\varsigma+\delta$.
Heading back to \eqref{doubling1} and using the same calculations as
in Proposition~\ref{propdoubling}, we deduce
%
%\[
\begin{eqnarray*}
&& \mathbb{E}\int_{\mathbb{T}^N} \bigl(u^\tau(t)-u^\sigma(t)
\bigr)^+\, \mathrm{d}x\\
&&\qquad\leq 2C\varepsilon^\varsigma+2\delta+\mathrm{I}+
\mathrm{J}+\mathrm {J}^\#+\mathrm{K}.
\end{eqnarray*}
%
%\]
%
The terms $\mathrm{I}, \mathrm{J}, \mathrm{K}$ are defined and can
be dealt with exactly as in Proposition~\ref{propdoubling} and
Theorem~\ref{uniqueness}. The term $\mathrm{J}^\#$ is defined as
%$$\mathrm{I}=\,\stred\int_0^t\int_{(\mt^N)^2}\int_{\mr^2}f^\tau\bar{f}^
%\sigma\big(b(\xi)-b(\zeta)\big)\cdotp\nabla_x\varrho_\varepsilon(x-y)
%\psi_\delta(\xi-\zeta)\,\dif\xi\,
%\dif\zeta\,\dif x\, \dif y\,\dif s,$$
%\[
%\begin{eqnarray}
%\mathrm{J}&=\stred\int_0^t\int_{(\mt^N)^2}\int_{\mr^2}f^\tau\bar{f}
%^\sigma\big(A(\xi)+A(\zeta)\big):\totdif^2_x\varrho_\varepsilon(x-y)
%\psi_\delta(\xi-\zeta)\,
%\dif\xi\,\dif\zeta\,\dif x\, \dif y\,\dif s\\
%&\qquad-\stred\int_0^t\int_{(\mt^N)^2}\int_{\mr^2}\varrho_
%\varepsilon(x-y)\psi_\delta(\xi-\zeta)\,
%\dif\nu^ {\tau} _ { x , s }
%(\xi)\,\dif x\,\dif n_{1}^\sigma(y,s,\zeta)\\
%&\qquad-\stred\int_0^t\int_{(\mt^N)^2}\int_{\mr^2}\varrho_
%\varepsilon(x-y)\psi_\delta(\xi-\zeta)\,
%\dif\nu^{\sigma}_{y,s}
%(\zeta)\,\dif y\,\dif n_{1}^\tau(x,s,\xi),
%\end{eqnarray}
%\]
%
%\[
\begin{eqnarray*}
\mathrm{J}^\#&=& (\tau+\sigma) \mathbb{E}\int_0^t
\!\int_{(\mathbb
{T}^N)^2}\!\int_{\mathbb{R}^2}f^\tau
\bar{f} ^\sigma\Delta_x\varrho_\varepsilon(x-y)
\psi_\delta(\xi-\zeta)\, \mathrm{d}\xi\,\mathrm{d}\zeta\,\mathrm{d}x\,
\mathrm{d}y\, \mathrm{d}s
\\
&&{}-\mathbb{E}\int_0^t\!\int
_{(\mathbb{T}^N)^2}\!\int_{\mathbb
{R}^2}\varrho
_\varepsilon(x-y)\psi_\delta(\xi-\zeta)\, \mathrm{d}\nu^ {\tau}
_ { x , s } (\xi)\,\mathrm{d}x\,\mathrm{d}n_{2}^\sigma(y,s,
\zeta)
\\
&&{}-\mathbb{E}\int_0^t\!\int
_{(\mathbb{T}^N)^2}\!\int_{\mathbb
{R}^2}\varrho
_\varepsilon(x-y)\psi_\delta(\xi-\zeta)\, \mathrm{d}
\nu^{\sigma}_{y,s} (\zeta)\,\mathrm{d}y\,\mathrm{d}n_{2}^\tau(x,s,
\xi)
\end{eqnarray*}
%
%\]
%
%$$\mathrm{K}=\frac{1}{2}\stred\int_0^t\!\!\int_{(\mt^N)^2}\!\!\int_{
%\mr^2}
%\!\!\varrho_\varepsilon(x-y)\psi_\delta(\xi-\zeta)\!\sum_{k\geq1}\!
%\big|g_k(x,\xi)-g_k(y,
%\zeta)\big|^2\dif\nu^\tau_{x,s}(\xi)\dif\nu^\sigma_{y,s}(\zeta)\dif x\,
%\dif y\,\dif s.$$
%The terms $\mathrm{I},\,\mathrm{J},\,\mathrm{K}$ can be dealt with
%exactly as in Theorem~\ref{uniqueness} hence we obtain
%\[%\label{eqest}
%\mathrm{I}+\mathrm{J}+\mathrm{K}\leq Ct\big(\delta\varepsilon^{-1}+
%\delta^2\varepsilon^{-2}+\delta^{-1}\varepsilon^{2}+\delta^\alpha\big),
%\]
so
%
%\[
\begin{eqnarray*}
\mathrm{J}^\#&=& (\tau+\sigma) \mathbb{E}\int_0^t
\!\int_{(\mathbb{T}
^N)^2}\varrho_\varepsilon(x-y)\psi_\delta
\bigl(u^\tau-u^\sigma\bigr)\nabla _x
u^\tau\cdot\nabla_y u^\sigma\,\mathrm{d}x\,
\mathrm{d}y\, \mathrm{d}s
\\
&&{}-\tau\mathbb{E}\int_0^t\!\int
_{(\mathbb{T}^N)^2}\varrho _\varepsilon (x-y)\psi_\delta
\bigl(u^\tau-u^\sigma\bigr)\bigl|\nabla_x
u^\tau\bigr|^2\,\mathrm {d}x\, \mathrm{d}y\,\mathrm{d}s
\\
&&{}-\sigma\mathbb{E}\int_0^t\!\int
_{(\mathbb{T}^N)^2}\varrho _\varepsilon (x-y)\psi_\delta
\bigl(u^\tau-u^\sigma\bigr)\bigl|\nabla_y
u^\sigma\bigr|^2\,\mathrm {d}x\, \mathrm{d}y\,\mathrm{d}s
\\
&=& - \mathbb{E}\int_0^t\!\int
_{(\mathbb{T}^N)^2}\varrho_\varepsilon (x-y)\psi _\delta
\bigl(u^\tau-u^\sigma\bigr) \bigl|\sqrt{\tau} \nabla_x
u^\tau-\sqrt {\sigma} \nabla_y u^\sigma
\bigr|^2\,\mathrm{d}x\,\mathrm{d}y\, \mathrm{d}s
\\
&&{}+(\sqrt{\tau}-\sqrt{\sigma} )^2 \mathbb{E}\int
_0^t\!\int_{(\mathbb{T}^N)^2}
\varrho_\varepsilon(x-y)\psi_\delta\bigl(u^\tau
-u^\sigma \bigr)\nabla_x u^\tau\cdot
\nabla_y u^\sigma\,\mathrm{d}x\,\mathrm {d}y\,\mathrm{d}s
\\
&=&\mathrm{J}^\#_1+\mathrm{J}^\#_2.
\end{eqnarray*}
%
%\]
%
The first term on the right-hand side is nonpositive and can be thus
forgotten; for the second one, we have
%
%\[
\begin{eqnarray*}
\bigl|\mathrm{J}^\#_2 \bigr|&\leq & (\sqrt{\tau}-\sqrt{\sigma} )^2
\mathbb{E}\int_0^t\!\int_{(\mathbb{T}^N)^2}
\!\int_{\mathbb
{R}^2}f^\tau\bar{f}^\sigma
\psi_\delta(\xi-\zeta) \bigl|\Delta_x\varrho_\varepsilon(x-y) \bigr|
\,\mathrm{d}\xi\,\mathrm{d}\zeta\,\mathrm{d}x\,\mathrm{d}y\, \mathrm{d}s
\end{eqnarray*}
%
%\]
%
and proceeding similarly as in the case of $\mathrm{I}$ we get
%
%\[
%e6.12 #&#
%e6.13 #&#
\begin{eqnarray*}
\bigl|\mathrm{J}^\#_2 \bigr|&\leq & (\sqrt{\tau}-\sqrt{\sigma} )^2
\\
&&{}\times\mathbb{E}\int_0^t\!\int_{(\mathbb{T}^N)^2}\!
\int_{\mathbb{R}^2}|\xi -\zeta+\delta|\, \mathrm{d}
\nu^\tau_{x,s}(\xi)\,\mathrm{d}\nu^\sigma_{y,s}(
\zeta ) \bigl|\Delta_x\varrho_\varepsilon(x-y) \bigr|\,\mathrm{d}x\,
\mathrm{d}y\, \mathrm{d}s
\\
&\leq &  C(\sqrt{\tau}-\sqrt{\sigma} )^2\varepsilon^{-2},
\end{eqnarray*}
%
%\]
%
where the last inequality follows from \eqref{energy}.
Consequently, we see that
%
%\[
\begin{eqnarray*}
\mathbb{E}\int_0^T\!\!\int_{\mathbb{T}^N}
\bigl(u^\tau(t)-u^\sigma (t) \bigr)^+\, \mathrm{d}x\,
\mathrm{d}t&\leq &  C \bigl(\varepsilon^\varsigma+\delta +\delta
\varepsilon^{-1}+\delta^{2\gamma}\varepsilon^{-2}+\delta
^{-1}\varepsilon^2+\delta^\alpha \bigr)
\\
&&{}+ C(\tau+\sigma) \varepsilon^{-2}
\end{eqnarray*}
%
%\]
%
and, therefore, given $\vartheta>0$ one can fix $\varepsilon$ and
$\delta$ small enough so that the first term on the right-hand side is
estimated by $\vartheta/2$ and then find $\iota>0$ such that also the
second term is estimated by $\vartheta/2$ for any $\tau,\sigma<\iota
$. Thus, we have shown that the set of approximate solutions $\{u^\tau
\}$ is Cauchy in $L^1(\Omega\times[0,T],\mathcal{P},\mathrm
{d}\mathbb{P}
\otimes\mathrm{d}t;L^1(\mathbb{T}^N))$, as $\tau\rightarrow0$.
\end{pf}
%

%co6.3 #&#
\begin{cor}\label{corconv}
For all $p\in[1,\infty)$,
%$$u\in L^p(\Omega\times[0,T],\mathcal{P},\dif\prst\otimes\dif t;L^p(
%\mt^N))\cap L^p(\Omega\semicol L^\infty(0,T\semicol L^p(\mt^N)))$$
\[
u^\tau\longrightarrow u\qquad\mbox{in } L^p\bigl(
\Omega\times [0,T],\mathcal{P},\mathrm{d}\mathbb{P}\otimes\mathrm
{d}t;L^p\bigl(\mathbb{T}^N\bigr)\bigr)
\]
and the following estimate holds true:
\[
\mathbb{E}\mathop{\operatorname{ess}\operatorname{sup}}_{0\leq t\leq T}\bigl\| u(t)
\bigr\|_{L^p(\mathbb{T}^N)}^p\leq C \bigl(1+\mathbb{E}\|u_0
\|_{L^p(\mathbb{T}^N)}^p \bigr).
\]
\end{cor}

\begin{pf}
The claim follows directly from Theorem~\ref{thmconv} and the
estimate \eqref{energy}.
\end{pf}
%

%th6.4 #&#
\begin{thm}\label{thmfinal}
The process $u$ constructed in Theorem~\ref{thmconv} is the unique
kinetic solution to \eqref{eq} under the additional hypothesis
$\mathrm{(H1)}$.
\end{thm}

\begin{pf}
Let $t\in[0,T]$. According to Corollary~\ref{corconv}, there exists
a set $\Sigma\subset\Omega\times[0,T]\times\mathbb{T}^N$ of full measure
and a subsequence still denoted by $\{u^n; n\in\mathbb{N}\}$ such that
$u^n(\omega,t,x)\rightarrow u(\omega,t,x)$ for all $(\omega,t,x)\in
\Sigma$. We infer that
%
%e6.14 #&#
\begin{equation}
\label{c} \mathbf{1}_{u^n(\omega,t,x)>\xi}\longrightarrow\mathbf
{1}_{u(\omega,t,x)>\xi}
\end{equation}
whenever
\[
(\mathbb{P}\otimes\mathcal{L}_{\mathbb{T}^N}\otimes\mathcal {L}_{[0,T]}
) \bigl\{(\omega,x)\in\Sigma; u(\omega,t,x)=\xi \bigr\}=0,
\]
where by $\mathcal{L}_{\mathbb{T}^N}$ and $\mathcal{L}_{[0,T]}$ we denoted
the Lebesque measure on $\mathbb{T}^N$ and $[0,T]$, respectively.
However, the set
\[
D= \bigl\{\xi\in\mathbb{R}; (\mathbb{P}\otimes\mathcal {L}_{\mathbb{T}
^N}
\otimes\mathcal{L}_{[0,T]} ) (u=\xi )>0 \bigr\}
\]
is at most countable since we deal with finite measures. To obtain a
contradiction, suppose that $D$ is uncountable and denote
\[
D_k= \biggl\{\xi\in\mathbb{R}; (\mathbb{P}\otimes\mathcal
{L}_{\mathbb{T}
^N}\otimes\mathcal{L}_{[0,T]} ) (u=\xi )>
\frac
{1}{k} \biggr\},\qquad k\in\mathbb{N}.
\]
Then $D=\bigcup_{k\in\mathbb{N}}D_k$ is a countable union so there exists
$k_0\in\mathbb{N}$ such that $D_{k_0}$ is uncountable. Hence,
%
%\[
\begin{eqnarray*}
(\mathbb{P}\otimes\mathcal{L}_{\mathbb{T}^N}\otimes\mathcal {L}_{[0,T]}
) (u\in D ) &\geq &  (\mathbb{P}\otimes \mathcal {L}_{\mathbb{T}^N}\otimes
\mathcal{L}_{[0,T]} ) (u\in D_{k_0} )
\\
&=& \sum_{\xi\in D_{k_0}} (\mathbb{P}\otimes
\mathcal{L}_{\mathbb{T}
^N}\otimes\mathcal{L}_{[0,T]} ) (u=\xi )>\sum
_{\xi\in
D_{k_0}}\frac{1}{k_0}=\infty
\end{eqnarray*}
%
%\]
%
and the desired contradiction follows.
We conclude that the convergence in \eqref{c} holds true for a.e.
$(\omega,t,x,\xi)$ and obtain by the dominated convergence theorem
\[
\label{c2} f^n\stackrel{w^*} {\longrightarrow}f\qquad \mbox{in }
L^\infty \bigl(\Omega\times[0,T]\times\mathbb{T}^N\times
\mathbb{R}\bigr).
\]
As a consequence, we can pass to the limit in all the terms on the
left-hand side of the weak form of \eqref{eqkinapprox} and obtain the
left-hand side of \eqref{eqkinformul}. Convergence of the stochastic
integral as well as the last term in the weak form \eqref
{eqkinapprox} to the corresponding terms in \eqref{eqkinformul} can
be verified easily using Corollary~\ref{corconv} and the energy
estimate \eqref{energy}.

In order to obtain the convergence of the remaining term $\partial_\xi
m^\tau=\partial_\xi n^\tau_1+\partial_\xi n^\tau_2$ to a kinetic
measure, we observe that due to the computations used in the proof of~\eqref{energy}, it holds
%
%\[
\begin{eqnarray*}
&& \int_0^T\!\!\int_{\mathbb{T}^N}\bigl|
\sigma\bigl(u^\tau\bigr)\nabla u^\tau \bigr|^2
\,\mathrm{d}x\,\mathrm{d} t+\tau\bigl|\nabla u^\tau\bigr|^2\,\mathrm{d}x\,
\mathrm{d}t
\\
&&{}\qquad \leq C\|u_0\| _{L^2(\mathbb{T}^N)}^2+C\sum_{k\geq1}\int_0^T
\!\!\int_{\mathbb{T}^N}u^\tau g_k
\bigl(u^\tau\bigr)\, \mathrm{d}x\, \mathrm{d}\beta_k(t)\\
&&\qquad\quad{}+C
\int_0^T\!\!\int_{\mathbb{T}^N}G^2
\bigl(u^\tau\bigr)\, \mathrm{d}x\,\mathrm{d}s.
\end{eqnarray*}
%
%\]
%
Taking square and expectation and finally by the It\^o isometry, we deduce
%
%\[
\begin{eqnarray*}
 && \mathbb{E} \bigl|m^\tau\bigl([0,T]\times\mathbb{T}^N\times
\mathbb {R}\bigr) \bigr|^2 \\
&&\qquad= \mathbb{E} \biggl|\int_0^T\!\!
\int_{\mathbb{T}^N}\bigl|\sigma\bigl(u^\tau\bigr)\nabla
u^\tau \bigr|^2\,\mathrm{d}x\, \mathrm{d}t+\tau\bigl|\nabla
u^\tau\bigr|^2\,\mathrm{d}x\,\mathrm{d}t \biggr|^2
\leq   C.
\end{eqnarray*}
%
%\]
%
Hence, the set $\{m^\tau; \tau\in(0,1)\}$ is bounded in
$L^2_w(\Omega;\mathcal{M}_b([0,T]\times\mathbb{T}^N\times\mathbb
{R}))$ and,
according to the Banach--Alaoglu theorem, it possesses a weak$^*$
convergent subsequence, denoted by $\{m^n; n\in\mathbb{N}\}$.
% that convergesthere exists $m\in L^2_w(\Omega;\mathcal{M}_b(\mt^N
%\times[0,T]\times\mr))$ such that (up to subsequences)
%$m^\tau\overset{w^*}{ \longrightarrow} m\quad\mbox{in}\quad L^2_w(
%\Omega;\mathcal{M}_b(\mt^N\times[0,T]\times\mr)).$$
Now, it only remains to show that its weak$^*$ limit $m$ is actually a
kinetic measure.
The first point of Definition~\ref{mees} is straightforward as it
corresponds to the weak$^*$-measurability of $m$. The second one giving
the behavior for large $\xi$ is a consequence of the uniform estimate
\eqref{energy}. Indeed, let $(\chi_\delta)$ be a truncation on
$\mathbb{R}$, then it holds, for $p\in[2,\infty)$,
%
%\[
\begin{eqnarray*}
&& \mathbb{E}\int_{[0,T]\times\mathbb{T}^N\times\mathbb{R}}|\xi |^{p-2}\,\mathrm{d}
m(t,x,\xi) \\
&&\qquad\leq  \mathop{\lim\inf}_{\delta\rightarrow0}\mathbb{E}\int_{[0,T]\times\mathbb{T}^N\times\mathbb{R}}|
\xi|^{p-2}\chi_\delta (\xi)\,\mathrm{d} m(t,x,\xi)
\\
&&\qquad=\mathop{\lim\inf}_{\delta\rightarrow0}\lim_{n\rightarrow\infty}\mathbb{E} \int
_{[0,T]\times\mathbb{T}^N\times\mathbb{R}}|\xi|^{p-2}\chi _\delta(\xi)\,
\mathrm{d}m^n(t,x,\xi)\leq C,
\end{eqnarray*}
%
%\]
%
where the last inequality follows from \eqref{energy}.
Accordingly, $m$ vanishes for large $\xi$.
In order to verify the remaining requirement of Definition~\ref{mees},
let us define
\[
x^n(t)=\int_{[0,t]\times\mathbb{T}^N\times\mathbb{R}}\psi(x,\xi )\,\mathrm{d}
m^n(s,x,\xi)
\]
and take the limit as $n\rightarrow\infty$.
These processes are predictable due to the definition of measures
$m^n$. Let $\alpha\in L^2(\Omega), \gamma\in L^2(0,T)$, then by the
Fubini theorem,
\[
\mathbb{E} \biggl(\alpha\int_0^T\gamma(t)
x^n(t)\,\mathrm{d}t \biggr)=\mathbb{E} \biggl(\alpha\int
_{[0,T]\times\mathbb{T}^N\times\mathbb{R}}\psi (x,\xi)\Gamma (s)\,\mathrm{d}m^n(s,x,
\xi) \biggr),
\]
where $\Gamma(s)=\int_s^T\gamma(t)\,\mathrm{d}t$. Hence, since
$\Gamma$
is continuous, we obtain by the weak convergence of $m^n$ to $m$
\[
\mathbb{E} \biggl(\alpha\int_0^T\gamma(t)
x^n(t)\,\mathrm{d}t \biggr)\longrightarrow\mathbb{E} \biggl(\alpha
\int_0^T\gamma(t) x(t)\, \mathrm{d} t \biggr),
\]
where
\[
x(t)=\int_{[0,t]\times\mathbb{T}^N\times\mathbb{R}}\psi(x,\xi)\, \mathrm{d}m(s,x,\xi).
\]
Consequently, $x^n$ converges to $x$ weakly in $ L^2(\Omega\times
[0,T])$ and, in particular, since the space of predictable
$L^2$-integrable functions is weakly closed, the claim follows.

Finally, by the same approach as above, we deduce that there exist
kinetic measures $o_1, o_2 $ such that
\[
n_1^n\stackrel{w^*} {\longrightarrow}o_1,\qquad
n_2^n\stackrel {w^*} {\longrightarrow} o_2
\qquad\mbox{in } L_w^2\bigl(\Omega ;\mathcal{M}_b
\bigl([0,T]\times\mathbb{T}^N\times\mathbb{R}\bigr)\bigr).
\]
Then from \eqref{energy} we obtain
\[
\mathbb{E}\int_0^T\!\!\int_{\mathbb{T}^N}
\biggl|\operatorname{div}\int_0^{u^n}\sigma(\zeta )\,
\mathrm{d}\zeta \biggr|^2\,\mathrm{d}x\,\mathrm{d}t\leq C
\]
hence application of the Banach--Alaoglu theorem yields that, up to
subsequence, $\operatorname{div}\int_0^{u^n}\sigma(\zeta)\,\mathrm
{d}\zeta$
converges weakly in $L^2(\Omega\times[0,T]\times\mathbb{T}^N)$. On the
other hand, from the strong convergence given by Corollary~\ref
{corconv} and the fact that $\sigma\in C_b(\mathbb{R})$, we conclude using
integration by parts, for all $\psi\in C^1([0,T]\times\mathbb{T}^N)$,
$\mathbb{P}\mbox{-a.s.}$,
\begin{eqnarray*}
&& \int_0^T\!\!\int_{\mathbb{T}^N}
\biggl(\operatorname{div}\int_0^{u^n}\sigma(\zeta)
\, \mathrm{d}\zeta \biggr)\psi(t,x)\,\mathrm{d}x\,\mathrm {d}t\\
&&\qquad\longrightarrow\int
_0^T\!\! \int_{\mathbb{T}^N} \biggl(
\operatorname{div}\int_0^u\sigma(\zeta )\,
\mathrm{d}\zeta \biggr)\psi(t,x)\,\mathrm{d}x\,\mathrm{d}t,
\end{eqnarray*}
and, therefore,
%
%e6.15 #&#
\begin{equation}
\label{eqconvergence}
\operatorname{div}\int_0^{u^n}
\sigma(\zeta)\,\mathrm{d}\zeta \stackrel {w} {\longrightarrow}\operatorname{div}
\int_0^u\sigma(\zeta)\, \mathrm{d}\zeta\qquad
\mbox{in } L^2\bigl([0,T]\times\mathbb{T}^N\bigr),
\mathbb{P}\mbox{-a.s.}\hspace*{-6pt}
\end{equation}
Since any norm is weakly sequentially lower semicontinuous, it follows
for all $\varphi\in C_0([0,T]\times\mathbb{T}^N\times\mathbb{R})$
and fixed $\xi
\in\mathbb{R}$, $\mathbb{P}$-a.s.,
%
%\[
\begin{eqnarray*}
&& \int_0^T\!\!\int_{\mathbb{T}^N} \biggl|
\operatorname{div}\int_0^{u}\sigma(\zeta)\,
\mathrm{d} \zeta \biggr|^2\varphi^2(t,x,\xi)\,\mathrm{d}x\,
\mathrm{d}t
\\
&&\qquad\leq\mathop{\lim\inf}_{n\rightarrow\infty}\int_0^T
\!\!\int_{\mathbb
{T}^N} \biggl|\operatorname{div}\int_0^{u^n}
\sigma(\zeta)\,\mathrm{d}\zeta \biggr|^2\varphi ^2(t,x,\xi)\,
\mathrm{d}x\,\mathrm{d}t
\end{eqnarray*}
%
%\]
%
and by the Fatou lemma
%
%\[
\begin{eqnarray*}
&& \int_0^T\!\!\int_{\mathbb{T}^N}\!\int
_{\mathbb{R}} \biggl|\operatorname {div}\int_0^{u}
\sigma (\zeta)\,\mathrm{d}\zeta \biggr|^2\varphi^2(t,x,\xi)\,
\mathrm {d}\delta_{u=\xi
}\,\mathrm{d}x\,\mathrm{d}t
\\
&&\qquad \leq\mathop{\lim\inf}_{n\rightarrow\infty}\int_0^T
\!\!\int_{\mathbb{T}^N}\!\int_{\mathbb{R}
} \biggl|\operatorname{div}
\int_0^{u^n}\sigma(\zeta)\,\mathrm {d}\zeta
\biggr|^2\varphi^2(t,x,\xi)\,\mathrm{d}\delta_{u^n=\xi}\,
\mathrm{d}x\, \mathrm{d}t,\qquad \mathbb{P}\mbox{-a.s.}
\end{eqnarray*}
%
%\]
%
In other words, this yields that $n_1\leq o_1$, $\mathbb{P}$-a.s., hence
$n_2=o_2+(o_1-n_1)$ is a.s. a nonnegative measure.

Concerning the chain rule formula \eqref{eqchainrule}, we observe
that it holds true for all $u^n$ due to their regularity, that is, for
any $\phi\in C_b(\mathbb{R})$
%
%e6.16 #&#
\begin{equation}
\label{chain}
\quad\operatorname{div}\int_0^{u^n}
\phi(\zeta)\sigma(\zeta)\,\mathrm {d}\zeta=\phi \bigl(u^n\bigr)
\operatorname{div}\int_0^{u^n}\sigma(\zeta)\,
\mathrm{d}\zeta \qquad\mbox{in }\mathcal{D'}\bigl(
\mathbb{T}^N\bigr),\mbox{ a.e. }(\omega,t).\hspace*{-20pt}
\end{equation}
Furthermore, as we can easily obtain \eqref{eqconvergence} with the
integrant $\sigma$ replaced by $\phi\sigma$, we can pass to the
limit on the left-hand side and, making use of the strong-weak
convergence, also on the right-hand side of \eqref{chain}. The proof
is complete.
\end{pf}
%

%s7 #&#
\section{Existence for degenerate case---general initial data}
\label{secgeneral}

In this final section, we complete the proof of Theorem~\ref
{thmmain}. In particular, we show existence of a kinetic solution to
\eqref{eq} for a general initial data $u_0\in L^p(\Omega;L^p(\mathbb{T}
^N))$, $\forall p\in[1,\infty)$. It is a straightforward consequence
of the previous section.
Indeed, let us approximate the initial condition by a sequence $\{
u^\varepsilon_0\}\subset L^p(\Omega;\break C^\infty(\mathbb{T}^N))$,
$\forall
p\in[1,\infty)$, such that $u_0^\varepsilon\rightarrow u_0$ in
$L^1(\Omega;L^1(\mathbb{T}^N))$ and
%
%e7.1 #&#
\begin{equation}
\label{iio}
\bigl\|u^\varepsilon_0\bigr\|_{L^p(\Omega;L^p(\mathbb{T}^N))}\leq
\|u_0\| _{L^p(\Omega
;L^p(\mathbb{T}^N))},\qquad\varepsilon\in(0,1), p\in[1,\infty).
\end{equation}
According to Theorem~\ref{thmfinal}, for each $\varepsilon\in
(0,1)$, there exists a unique kinetic solution $u^\varepsilon$ to
\eqref{eq} with initial condition $u_0^\varepsilon$.
Besides, by the comparison principle~\eqref{comparison},
\[
\mathbb{E}\int_0^T\bigl\|u^{\varepsilon_1}(t)-u^{\varepsilon_2}(t)
\bigl\| _{L^1(\mathbb{T}^N)}\,\mathrm{d}t\leq T \mathbb{E}\bigl\|u^{\varepsilon
_1}_0-u^{\varepsilon_2}_0
\bigr\|_{L^1(\mathbb{T}^N)},\qquad\varepsilon _1,\varepsilon_2
\in(0,1),
\]
hence $\{u^\varepsilon; \varepsilon\in(0,1)\}$ is a Cauchy sequence
in $L^1(\Omega\times[0,T],\mathcal{P},\mathrm{d}\mathbb{P}\otimes
\mathrm{d}
t;  L^1(\mathbb{T}^N))$. Consequently, there exists $u\in L^1(\Omega
\times
[0,T],\mathcal{P},\mathrm{d}\mathbb{P}\otimes\mathrm
{d}t;L^1(\mathbb{T}^N))$ such that
\[
u^\varepsilon\longrightarrow u\qquad\mbox{in } L^1\bigl(
\Omega \times[0,T],\mathcal{P},\mathrm{d}\mathbb{P}\otimes\mathrm
{d}t;L^1\bigl(\mathbb{T}^N\bigr)\bigr).
\]
By \eqref{iio}, we have the uniform energy estimate, $p\in[1,\infty)$,
\[
%\label{en}
%\begin{eqnarray}
\mathbb{E}\mathop{\operatorname{ess}\operatorname{sup}}_{0\leq t\leq T}\bigl\|
u^\varepsilon(t)\bigr\|^p_{L^p(\mathbb{T}
^N)}\leq C,
%\sup_{0\leq t\leq T}\stred\|u^\varepsilon(t)\|^p_{L^p(\mt^N)}&\leq
%C_{T,u_0},
%\end{eqnarray}
%
\]
as well as
\[
\mathbb{E} \bigl|m^\varepsilon\bigl([0,T]\times\mathbb{T}^N\times
\mathbb {R}\bigr) \bigr|^2\leq C.
\]
Thus, using this observations as in Theorem~\ref{thmfinal}, one finds
that there exists a subsequence $\{u^n; n\in\mathbb{N}\}$ such that:
\begin{longlist}[(ii)]
\item[(i)] $f^n\stackrel{w^*}{\longrightarrow}f  \mbox{ in }
L^\infty(\Omega\times[0,T]\times\mathbb{T}^N\times\mathbb{R})$,
\item[(ii)] there exists a kinetic measure $m$ such that
\[
m^n\stackrel{w^*} {\longrightarrow}m\qquad \mbox{in } L^2_w
\bigl(\Omega;\mathcal{M}_b\bigl([0,T]\times\mathbb{T}^N
\times\mathbb{R}\bigr)\bigr)
\]
and $m=n_1+n_2$, where
\[
\mathrm{d}n_1(t,x,\xi)= \biggl|\operatorname{div}\int_0^u
\sigma (\zeta) \,\mathrm{d}\zeta \biggr|^2\,\mathrm{d}\delta_{u(t,x)}(
\xi)\,\mathrm{d}x\,\mathrm{d}t
\]
and $n_2$ is a.s. a nonnegative measure over $[0,T]\times\mathbb
{T}^N\times
\mathbb{R}$.
\end{longlist}
With these facts in hand, we are ready to pass to the limit in \eqref
{eqkinformul} and conclude that $u$ is the unique kinetic solution to
\eqref{eq}.
%satisfies the kinetic formulation in the sense of distributions.
%Note, that \eqref{en} remains valid also for $u$ so \eqref{integrov}
%follows and, according to the embedding $L^p(\mt^N)\hookrightarrow L^1(
%\mt^N)$, for all $p\in[1,\infty)$, we deduce
%$$u\in L^p(\Omega\times[0,T],\mathcal{P},\dif\prst\otimes\dif t;L^p(
%\mt^N)).$$
The proof of Theorem~\ref{thmmain} is thus complete.

%sA #&#
\begin{appendix}
\section*{Appendix:  Generalized It\^o's formula}
\label{secito}
In this section, we establish a generalized It\^o formula for weak
solutions of a very general class of SPDEs of the form
%
%\begin{equation}
%eA.1 #&#
\setcounter{equation}{0}
\begin{eqnarray}
\label{eq1}
\mathrm{d}u &=& F(t)\,\mathrm{d}t+\operatorname{div}G(t)\,
\mathrm{d}t+ H(t)\,\mathrm{d}W,
\nonumber
\\[-8pt]
\\[-8pt]
\nonumber
u(0)&=& u_0,
\end{eqnarray}
%
%\end{equation}
%
where $W$ is the cylindrical Wiener process defined in Section~\ref{secnotation}. Similar ideas were already used in \cite{fellah}. In
the present context, the result is applied in the derivation of the
kinetic formulation in Section~\ref{subsecformulation} as well as
in the proof of a priori $L^p(\mathbb{T}^N)$-estimates in Proposition~\ref{propenergy2}. The result reads as follows.

%prA.1 #&#
\begin{prop}\label{propito}
Let $\psi\in C^1(\mathbb{T}^N)$ and $\varphi\in C^2(\mathbb{R})$
with bounded
second-order derivative. Assume that the coefficients $F$, $G_i,
i=1,\dots,N$, belong to $L^2(\Omega;L^2(0,T;L^2(\mathbb{T}^N)))$ and
$H\in
L^2(\Omega;L^2(0,T;L_2(\mathfrak{U};L^2(\mathbb{T}^N))))$, we denote
$H_k=He_k$, $k\in\mathbb{N}$. Let the equation \eqref{eq1} be
satisfied in
$H^{-1}(\mathbb{T}^N)$ for some
\[
u\in L^2\bigl(\Omega;C\bigl([0,T];L^2\bigl(
\mathbb{T}^N\bigr)\bigr)\bigr) \cap L^2\bigl(\Omega
;L^2\bigl(0,T;H^1\bigl(\mathbb{T}^N\bigr)
\bigr)\bigr).
\]
Then almost surely, for all $t\in[0,T]$,
%
%\begin{equation}
%eA.2 #&#
\begin{eqnarray}
\nonumber
\bigl\langle\varphi\bigl(u(t)\bigr),\psi \bigr\rangle &=& \bigl
\langle\varphi (u_0),\psi \bigr\rangle+\int_0^t
\bigl\langle\varphi'\bigl(u(s)\bigr)F(s),\psi \bigr\rangle\,
\mathrm{d}s
\\
\nonumber
&&{}-\int_0^t \bigl\langle
\varphi''\bigl(u(s)\bigr)\nabla u\cdot G(s),\psi \bigr
\rangle\,\mathrm{d}s
\\
\label{ito}
&&{}+\int_0^t \bigl\langle
\operatorname{div} \bigl(\varphi'\bigl(u(s)\bigr) G(s) \bigr),\psi
\bigr\rangle\,\mathrm{d}s
\\
\nonumber
&&{}+\int_0^t \bigl\langle
\varphi'\bigl(u(s)\bigr)H(s)\,\mathrm{d}W(s),\psi \bigr\rangle
\\
\nonumber
&&{}+\frac{1}{2}\sum_{k\geq1}\int
_0^t \bigl\langle\varphi ''
\bigl(u(s)\bigr)H_k^2(s),\psi \bigr\rangle\,\mathrm{d}s.
\end{eqnarray}
%
%\end{equation}
\end{prop}

\begin{pf}
In order to prove the claim, we use regularization by convolutions. Let
$(\varrho_\delta)$ be an approximation to the identity on $\mathbb{T}^N$.
For a function $f$ on $\mathbb{T}^N$, we denote by $f^\delta$ the
convolution $f*\varrho_\delta$. Recall, that if $f\in L^2(\mathbb
{T}^N)$ then
\[
\bigl\|f^\delta\bigr\|_{L^2(\mathbb{T}^N)}\leq\|f\|_{L^2(\mathbb
{T}^N)},\qquad
\bigl\|f^\delta -f\bigr\|_{L^2(\mathbb{T}^N)}\longrightarrow0.
\]
Using $\varrho_\delta(x-\cdot)$ as a test function in \eqref{eq1},
we obtain that
\[
u^\delta(t)=u_0^\delta+\int_0^t
F^\delta(s)\,\mathrm{d}s+ \int_0^t
\operatorname{div}G^\delta(s)\,\mathrm{d}s+\sum
_{k\geq1}\int_0^t
H_k^\delta (s)\,\mathrm{d}\beta_k(s)
\]
holds true for every $x\in\mathbb{T}^N$. Hence, we can apply the classical
1-dimensional It\^o formula to the function $u(x)\mapsto\varphi
(u(x))\psi(x)$ and integrate with respect to $x$
%
%
%eA.3 #&#
\begin{eqnarray}
\nonumber
\bigl\langle\varphi\bigl(u^\delta(t)\bigr),\psi \bigr
\rangle &= &\bigl\langle \varphi\bigl(u^\delta_0\bigr),\psi
\bigr\rangle+\int_0^t \bigl\langle\varphi
'\bigl(u^\delta(s)\bigr)F^\delta(s),\psi \bigr
\rangle\,\mathrm{d}s
\\
\nonumber
&&{}-\int_0^t \bigl\langle
\varphi''\bigl(u^\delta(s)\bigr)\nabla
u^\delta (s)\cdot G^\delta(s),\psi \bigr\rangle\,\mathrm{d}s
\\
\nonumber
&&{}+\int_0^t \bigl\langle
\operatorname{div} \bigl(\varphi '\bigl(u^\delta (s)
\bigr)G^\delta(s) \bigr),\psi \bigr\rangle\,\mathrm{d}s
\\[-8pt]
\label{ito1}
\\[-8pt]
\nonumber
&&{}+\sum_{k\geq1}\int_0^t
\bigl\langle\varphi'\bigl(u^\delta (s)\bigr)H_k^\delta(s),
\psi \bigr\rangle\,\mathrm{d}\beta_k(s)
\\
\nonumber
&&{}+\frac{1}{2}\sum_{k\geq1}\int
_0^t \bigl\langle\varphi ''
\bigl(u^\delta(s)\bigr)\bigl[H_k^\delta(s)
\bigr]^2,\psi \bigr\rangle\,\mathrm {d}s\\\
\nonumber
&=& \mathrm {J}_1+
\cdots+\mathrm{J}_6.
\end{eqnarray}
%
%\end{equation}
%
We will now show that each term in \eqref{ito1} converge a.s. to the
corresponding term in \eqref{ito}.
For the stochastic term, we apply the Burkholder--Davis--Gundy inequality
%
%\begin{equation}
%eA.4 #&#
\begin{eqnarray}
&& \mathbb{E}\sup_{0\leq t\leq T} \biggl|\sum
_{k\geq1}\int_0^t \bigl\langle
\varphi'\bigl(u^\delta\bigr)H_k^\delta-
\varphi'(u)H_k,\psi \bigr\rangle\,\mathrm{d}
\beta_k(s) \biggr|
\nonumber
\\
&&\qquad \leq C \mathbb{E} \biggl(\int_0^T\sum
_{k\geq1} \bigl| \bigl\langle \varphi'
\bigl(u^\delta\bigr)H_k^\delta-\varphi'(u)H_k,
\psi \bigr\rangle \bigr|^2\,\mathrm{d}s \biggr)^{{1}/{2}}
\nonumber
\\[-8pt]
\label{stoch}
\\[-8pt]
\nonumber
&&\qquad \leq C \mathbb{E} \biggl(\int_0^T \bigl\|
\varphi'\bigl(u^\delta \bigr)-\varphi '(u)
\bigr\|_{L^2(\mathbb{T}^N)}^2 \bigl\|H^\delta \bigr\| _{L_2(\mathfrak
{U};L^2(\mathbb{T}^N))}^2
\,\mathrm{d}s \biggr)^{{1}/{2}}
\\
\nonumber
&&\qquad\quad{}+ C \mathbb{E} \biggl(\int_0^T \bigl\|
\varphi'(u) \bigr\| _{L^2(\mathbb{T}
^N)}^2 \bigl\|H^\delta-H
\bigr\|_{L_2(\mathfrak{U};L^2(\mathbb
{T}^N))}^2\,\mathrm{d} s \biggr)^{{1}/{2}}.
\end{eqnarray}
%
%\end{equation}
%
Since $\varphi'$ is Lipschitz, we have $\|\varphi'(u^\delta)-\varphi
'(u)\|_{L^2(\mathbb{T}^N)}\rightarrow0$ a.e. in $\omega,t$ and
%
%\[
\begin{eqnarray*}
&& \mathbb{E} \biggl(\int_0^T \bigl\|
\varphi'\bigl(u^\delta\bigr)-\varphi '(u) \bigr\|
_{L^2(\mathbb{T}^N)}^2 \bigl\|H^\delta \bigr\|_{L_2(\mathfrak
{U};L^2(\mathbb{T}
^N))}^2
\,\mathrm{d}s \biggr)^{{1}/{2}}
\\
&&\qquad \leq C \mathbb{E} \biggl(\int_0^T\|u
\|_{L^2(\mathbb{T}^N)}^2\|H\| _{L_2(\mathfrak{U};L^2(\mathbb{T}^N))}^2\,\mathrm{d}s
\biggr)^{{1}/{2}}
\\
&&\qquad\leq C \mathbb{E}\sup_{0\leq t\leq T}\|u\|_{L^2(\mathbb
{T}^N)}^2+C
\mathbb{E} \int_0^T\|H\|_{L_2(\mathfrak{U};L^2(\mathbb{T}^N))}^2
\,\mathrm{d}s
\end{eqnarray*}
%
%\]
%
hence the first term on the right-hand side of \eqref{stoch} converges
to zero by dominated convergence theorem. The second one can be dealt
with similarly as $\|H^\delta-H\|_{L_2(\mathfrak{U};L^2(\mathbb{T}
^N))}\rightarrow0$ a.e. in $\omega,t$. As a consequence, we obtain
(up to subsequences) the almost sure convergence of $\mathrm{J}_5$.

All the other terms can be dealt with similarly using the dominated
convergence theorem. Let us now verify the convergence of $\mathrm
{J}_3$. It holds
%
%\[
\begin{eqnarray*}
&& \biggl|\int_0^t  \bigl\langle\varphi''
\bigl(u^\delta\bigr)\nabla u^\delta \cdot G^\delta-
\varphi''(u)\nabla u\cdot G,\psi \bigr\rangle\,
\mathrm{d} s \biggr|
\\
&&\qquad \leq\int_0^t \bigl| \bigl\langle
\varphi''\bigl(u^\delta\bigr)\nabla
u^\delta \cdot \bigl(G^\delta-G \bigr),\psi \bigr\rangle \bigr|\,
\mathrm{d}s
\\
&&\qquad\quad{}+\int_0^t \bigl| \bigl\langle
\varphi''\bigl(u^\delta\bigr) \bigl(\nabla
u^\delta-\nabla u \bigr)G,\psi \bigr\rangle \bigr|\,\mathrm{d}s
\\
&&\quad\qquad{}+\int_0^t \bigl| \bigl\langle \bigl(
\varphi''\bigl(u^\delta\bigr)-\varphi
''(u) \bigr)\nabla u\cdot G,\psi \bigr\rangle \bigr|\,
\mathrm{d}s.
\end{eqnarray*}
%
%\]
%
Since $\varphi''$ is bounded and $\|G^\delta-G\|_{L^2(\mathbb{T}
^N)}\rightarrow0$, $\|\nabla u^\delta-\nabla u\|_{L^2(\mathbb{T}
^N)}\rightarrow0$ a.e. in $\omega,t$ we deduce by dominated
convergence that the first two terms converge to zero. For the
remaining term, we shall use the fact that $\varphi''(u^\delta
)-\varphi''(u)\rightarrow0$ a.e. in $\omega,t,x$ and dominated
convergence again.

In the case of $\mathrm{J}_4$, we have
%
%\[
\begin{eqnarray*}
&& \biggl|\int_0^t  \bigl\langle\varphi'
\bigl(u^\delta\bigr)G^\delta-\varphi '(u)G,\nabla
\psi \bigr\rangle\,\mathrm{d}s \biggr|
\\
&&\qquad \leq\int_0^t \bigl| \bigl\langle
\varphi'\bigl(u^\delta\bigr) \bigl(G^\delta -G
\bigr),\nabla\psi \bigr\rangle \bigr|\,\mathrm{d}s
\\
&&\qquad\quad{}+ \int_0^t \bigl| \bigl\langle \bigl(
\varphi'\bigl(u^\delta\bigr)-\varphi '(u)
\bigr)G,\nabla\psi \bigr\rangle \bigr|\,\mathrm{d}s
\end{eqnarray*}
%
%\]
%
hence $\|G^\delta-G\|_{L^2(\mathbb{T}^N)}\rightarrow0$, $\|\varphi
'(u^\delta)-\varphi'(u)\|_{L^2(\mathbb{T}^N)}\rightarrow0$ a.e. in
$\omega
,t$ yield the conclusion. Similarly for $\mathrm{J}_2$.

Concerning $\mathrm{J}_6$, it holds
%
%\[
\begin{eqnarray*}
&& \biggl|\sum_{k\geq1}\int_0^t
\bigl\langle\varphi''\bigl(u^\delta \bigr)
\bigl[H_k^\delta\bigr]^2-\varphi''(u)H_k^2,
\psi \bigr\rangle\,\mathrm {d}s \biggr|
\\
&&\qquad\leq\sum_{k\geq1}\int_0^t
\bigl| \bigl\langle\varphi''\bigl(u^\delta \bigr)
\bigl(\bigl[H_k^\delta\bigr]^2-H_k^2
\bigr),\psi \bigr\rangle \bigr|\,\mathrm {d}s
\\
&&\qquad\quad{}+\sum_{k\geq1}\int_0^t
\bigl| \bigl\langle \bigl(\varphi ''\bigl(u^\delta
\bigr)-\varphi''(u) \bigr)H_k^2,
\psi \bigr\rangle \bigr|\, \mathrm{d}s,
\end{eqnarray*}
%
%\]
%
where for the first term we make use of boundedness of $\varphi''$,
the fact that
%
%\[
\begin{eqnarray*}
&& \bigl\|\bigl[H_k^\delta\bigr]^2 -H_k^2
\bigr\|_{L^1(\mathbb{T}^N)}\leq \bigl\| H_k^\delta-H_k
\bigr\|_{L^2(\mathbb{T}^N)} \bigl\|H_k^\delta+H_k \bigr\|
_{L^2(\mathbb{T}^N)}\longrightarrow0
\end{eqnarray*}
%
%\]
%
a.e. in $\omega,t$ and dominated convergence. For the second one, we
employ that $\varphi''(u^\delta)-\varphi(u)\rightarrow0$ a.e. in
$\omega,x,t$ together with boundedness of $\varphi''$.

Since $\varphi'$ has a linear growth, we obtain the convergence of
$\mathrm{J}_1$ as well as the term on the left-hand side of \eqref
{ito1}. Indeed, for all $t\in[0,T]$ we have
%
%\[
\begin{eqnarray*}
&& \bigl| \bigl\langle\varphi\bigl(u^\delta(t)\bigr)-\varphi\bigl(u(t)\bigr),
\psi \bigr\rangle \bigr| %\leq C\big\langle\big(1+|u^\delta(t)|+|u(t)|\big),|u^\delta(t)-u(t)|
%\big\rangle\\
\leq C \bigl(1+\bigl\|u(t)
\bigr\|_{L^2(\mathbb{T}^N)} \bigr)\bigl\|u^\delta(t)-u(t)\bigr\| _{L^2(\mathbb{T}^N)}
\longrightarrow0
\end{eqnarray*}
%
%\]
%
and the proof is complete.
\end{pf}
\end{appendix}

% imsref loaded by daiva.urboniene, 2015-03-25 17:14:04

\printaddresses

\begin{thebibliography}{28}
% pybtex-1.25. Style name=ims, version=2.91, label_style=nolabel, sorting_style=complex, cfg=None, language=None.


%b1 ###
%b1 #&#
\bibitem{bauzet1}
\begin{bmisc}[auto:parserefs-M02]
\bauthor{\bsnm{Bauzet},~\bfnm{C.}\binits{C.}},
\bauthor{\bsnm{Vallet},~\bfnm{G.}\binits{G.}} \AND
\bauthor{\bsnm{Wittbold},~\bfnm{P.}\binits{P.}}
\bhowpublished{A degenerate parabolic--hyperbolic Cauchy problem with a stochastic force.
Preprint, hal-01003069.}
\end{bmisc}
%
\iffalse\OrigBibText
C. Bauzet, G. Vallet and P. Wittbold, A degenerate
parabolic-hyperbolic Cauchy problem
with a stochastic force, preprint, hal-01003069.
\endOrigBibText\fi
\bptok{imsref}%
\endbibitem

%b2 ###
%b2 #&#
\bibitem{bauzet}
\begin{barticle}[mr]
\bauthor{\bsnm{Bauzet},~\bfnm{Caroline}\binits{C.}},
\bauthor{\bsnm{Vallet},~\bfnm{Guy}\binits{G.}} \AND
\bauthor{\bsnm{Wittbold},~\bfnm{Petra}\binits{P.}}
(\byear{2012}).
\btitle{The {C}auchy problem for conservation laws with a multiplicative stochastic perturbation}.
\bjournal{J. Hyperbolic Differ. Equ.}
\bvolume{9}
\bpages{661--709}.
\bid{doi={10.1142/S0219891612500221}, issn={0219-8916}, mr={3021756}}
\end{barticle}
%
\iffalse\OrigBibText
C. Bauzet, G. Vallet, P. Wittbold, The Cauchy problem
for conservation laws with a multiplicative noise, Journal of Hyp.
Diff. Eq. 9 (4) (2012) 661-709.
\endOrigBibText\fi
\bptok{imsref}%
% NOT OUTPUTTED:
%   number = 4
%   doi = http://dx.doi.org/10.1142/S0219891612500221
%   fjournal = Journal of Hyperbolic Differential Equations
\endbibitem

%b3 ###
%b3 #&#
\bibitem{on1}
\begin{barticle}[mr]
\bauthor{\bsnm{Brze{\'z}niak},~\bfnm{Zdzis{\l}aw}\binits{Z.}} \AND
\bauthor{\bsnm{Ondrej{\'a}t},~\bfnm{Martin}\binits{M.}}
(\byear{2007}).
\btitle{Strong solutions to stochastic wave equations with values in {R}iemannian manifolds}.
\bjournal{J. Funct. Anal.}
\bvolume{253}
\bpages{449--481}.
\bid{doi={10.1016/j.jfa.2007.03.034}, issn={0022-1236}, mr={2370085}}
\end{barticle}
%
\iffalse\OrigBibText
Z. Brze\'zniak, M. Ondrej\'at, Strong solutions to
stochastic wave equations with values in Riemannian manifolds, J.
Funct. Anal. 253 (2007) 449-481.
\endOrigBibText\fi
\bptok{imsref}%
% NOT OUTPUTTED:
%   number = 2
%   doi = http://dx.doi.org/10.1016/j.jfa.2007.03.034
%   coden = JFUAAW
%   fjournal = Journal of Functional Analysis
\endbibitem

%b4 ###
%b4 #&#
\bibitem{car}
\begin{barticle}[mr]
\bauthor{\bsnm{Carrillo},~\bfnm{Jos{\'e}}\binits{J.}}
(\byear{1999}).
\btitle{Entropy solutions for nonlinear degenerate problems}.
\bjournal{Arch. Ration. Mech. Anal.}
\bvolume{147}
\bpages{269--361}.
\bid{doi={10.1007/s002050050152}, issn={0003-9527}, mr={1709116}}
\end{barticle}
%
\iffalse\OrigBibText
J. Carrillo, Entropy solutions for nonlinear degenerate
problems, Arch. Rational Mech. Anal. 147 (1999) 269-361.
\endOrigBibText\fi
\bptok{imsref}%
% NOT OUTPUTTED:
%   number = 4
%   doi = http://dx.doi.org/10.1007/s002050050152
%   fjournal = Archive for Rational Mechanics and Analysis
\endbibitem

%b5 ###
%b5 #&#
\bibitem{chen}
\begin{barticle}[mr]
\bauthor{\bsnm{Chen},~\bfnm{Gui-Qiang}\binits{G.-Q.}} \AND
\bauthor{\bsnm{Perthame},~\bfnm{Beno{\^{\i}}t}\binits{B.}}
(\byear{2003}).
\btitle{Well-posedness for non-isotropic degenerate parabolic--hyperbolic equations}.
\bjournal{Ann. Inst. H. Poincar\'e Anal. Non Lin\'eaire}
\bvolume{20}
\bpages{645--668}.
\bid{doi={10.1016/S0294-1449(02)00014-8}, issn={0294-1449}, mr={1981403}}
\end{barticle}
%
\iffalse\OrigBibText
G. Q. Chen, B. Perthame, Well-posedness for
non-isotropic degenerate parabolic-hyperbolic equations, Ann. Inst. H.
Poincar\'e Anal. Non Lin\'eaire 20 (4) (2003) 645-668.
\endOrigBibText\fi
\bptok{imsref}%
% NOT OUTPUTTED:
%   number = 4
%   doi = http://dx.doi.org/10.1016/S0294-1449(02)00014-8
%   fjournal = Annales de l'Institut Henri Poincar\'e. Analyse Non Lin\'eaire
\endbibitem

%b6 ###
%b6 #&#
\bibitem{daprato}
\begin{bbook}[mr]
\bauthor{\bsnm{Da Prato},~\bfnm{Giuseppe}\binits{G.}} \AND
\bauthor{\bsnm{Zabczyk},~\bfnm{Jerzy}\binits{J.}}
(\byear{1992}).
\btitle{Stochastic Equations in Infinite Dimensions}.
\bseries{Encyclopedia of Mathematics and Its Applications}
\bvolume{44}.
\bpublisher{Cambridge Univ. Press},
\blocation{Cambridge}.
\bid{doi={10.1017/CBO9780511666223}, mr={1207136}}
\end{bbook}
%
\iffalse\OrigBibText
G. Da Prato, J. Zabczyk, Stochastic Equations in
Infinite Dimensions, Encyclopedia Math. Appl., vol. 44, Cambridge
University Press, Cambridge, 1992.
\endOrigBibText\fi
\bptok{imsref}%
% NOT OUTPUTTED:
%   doi = http://dx.doi.org/10.1017/CBO9780511666223
%   isbn = 0-521-38529-6
%   fpage = xviii+454
\endbibitem

%b7 ###
%b7 #&#
\bibitem{debus}
\begin{barticle}[mr]
\bauthor{\bsnm{Debussche},~\bfnm{A.}\binits{A.}} \AND
\bauthor{\bsnm{Vovelle},~\bfnm{J.}\binits{J.}}
(\byear{2010}).
\btitle{Scalar conservation laws with stochastic forcing}.
\bjournal{J.~Funct. Anal.}
\bvolume{259}
\bpages{1014--1042}.
\bid{doi={10.1016/j.jfa.2010.02.016}, issn={0022-1236}, mr={2652180}}
\end{barticle}
%
\iffalse\OrigBibText
A. Debussche, J. Vovelle, Scalar conservation laws
with stochastic forcing, J. Funct. Anal. 259 (2010) 1014--1042.
\endOrigBibText\fi
\bptok{imsref}%
% NOT OUTPUTTED:
%   number = 4
%   doi = http://dx.doi.org/10.1016/j.jfa.2010.02.016
%   coden = JFUAAW
%   fjournal = Journal of Functional Analysis
\endbibitem

%b8 ###
%b8 #&#
\bibitem{denis1}
\begin{barticle}[mr]
\bauthor{\bsnm{Denis},~\bfnm{Laurent}\binits{L.}},
\bauthor{\bsnm{Matoussi},~\bfnm{Anis}\binits{A.}} \AND
\bauthor{\bsnm{Stoica},~\bfnm{Lucretiu}\binits{L.}}
(\byear{2005}).
\btitle{{$L\sp p$} estimates for the uniform norm of solutions of quasilinear SPDE's}.
\bjournal{Probab. Theory Related Fields}
\bvolume{133}
\bpages{437--463}.
\bid{doi={10.1007/s00440-005-0436-5}, issn={0178-8051}, mr={2197109}}
\end{barticle}
%
\iffalse\OrigBibText
L. Denis, A. Matoussi, L. Stoica, $L^p$ estimates for
the uniform norm of solutions of quasilinear SPDE's, Probab. Theory
Related Fields 133 (4) (2005) 437--463.
\endOrigBibText\fi
\bptok{imsref}%
% NOT OUTPUTTED:
%   number = 4
%   doi = http://dx.doi.org/10.1007/s00440-005-0436-5
%   coden = PTRFEU
%   fjournal = Probability Theory and Related Fields
\endbibitem

%b9 ###
%b9 #&#
\bibitem{fellah}
\begin{bincollection}[mr]
\bauthor{\bsnm{Fellah},~\bfnm{D.}\binits{D.}} \AND
\bauthor{\bsnm{Pardoux},~\bfnm{{\'E}.}\binits{{\'E}.}}
(\byear{1992}).
\btitle{Une formule d'{I}t\^o dans des espaces de {B}anach, et application}.
In \bbooktitle{Stochastic Analysis and Related Topics ({S}ilivri, 1990)}.
\bseries{Progress in Probability}
\bvolume{31}
\bpages{197--209}.
\bpublisher{Birkh\"auser},
\blocation{Boston, MA}.
\bid{mr={1203376}}
\end{bincollection}
%
\iffalse\OrigBibText
D. Fellah, E. Pardoux, Une formule d'It\^o dans des
espaces de Banach et application, in: Stochastic analysis and related
topics (Silivri, 1990), Vol. 31 of Progr. Probab., Birkh\"auser
Boston, Boston, MA, 1992, 197-209.
\endOrigBibText\fi
\bptok{imsref}%
\endbibitem

%b10 ###
%b10 #&#
\bibitem{feng}
\begin{barticle}[mr]
\bauthor{\bsnm{Feng},~\bfnm{Jin}\binits{J.}} \AND
\bauthor{\bsnm{Nualart},~\bfnm{David}\binits{D.}}
(\byear{2008}).
\btitle{Stochastic scalar conservation laws}.
\bjournal{J. Funct. Anal.}
\bvolume{255}
\bpages{313--373}.
\bid{doi={10.1016/j.jfa.2008.02.004}, issn={0022-1236}, mr={2419964}}
\end{barticle}
%
\iffalse\OrigBibText
J. Feng, D. Nualart, Stochastic scalar conservation
laws, J. Funct. Anal. 255 (2) (2008) 313-373.
\endOrigBibText\fi
\bptok{imsref}%
% NOT OUTPUTTED:
%   number = 2
%   doi = http://dx.doi.org/10.1016/j.jfa.2008.02.004
%   coden = JFUAAW
%   fjournal = Journal of Functional Analysis
\endbibitem

%%b11 ###
%%b11 #&#
%\bibitem{fland}
%\begin{barticle}[mr]
%\bauthor{\bsnm{Flandoli},~\bfnm{Franco}\binits{F.}} \AND
%\bauthor{\bsnm{G{\c{a}}tarek},~\bfnm{Dariusz}\binits{D.}}
%(\byear{1995}).
%\btitle{Martingale and stationary solutions for stochastic {N}avier--{S}tokes equations}.
%\bjournal{Probab. Theory Related Fields}
%\bvolume{102}
%\bpages{367--391}.
%\bid{doi={10.1007/BF01192467}, issn={0178-8051}, mr={1339739}}
%\end{barticle}
%%
%\iffalse\OrigBibText
%F. Flandoli, D. G\c{a}tarek, Martingale and
%stationary solutions for stochastic Navier-Stokes equations, Probab.
%Theory Related Fields 102 (3) (1995) 367-391.
%\endOrigBibText\fi
%\bptok{imsref}%
%% NOT OUTPUTTED:
%%   number = 3
%%   doi = http://dx.doi.org/10.1007/BF01192467
%%   coden = PTRFEU
%%   fjournal = Probability Theory and Related Fields
%\endbibitem

%b12 ###
%b12 #&#
\bibitem{petrol}
\begin{bbook}[mr]
\bauthor{\bsnm{Gagneux},~\bfnm{G{\'e}rard}\binits{G.}} \AND
\bauthor{\bsnm{Madaune-Tort},~\bfnm{Monique}\binits{M.}}
(\byear{1996}).
\btitle{Analyse Math\'ematique de Mod\`eles Non Lin\'eaires de L'ing\'enierie P\'etroli\`ere}.
\bseries{Math\'ematiques \& Applications (Berlin) [Mathematics \& Applications]}
\bvolume{22}.
\bpublisher{Springer},
\blocation{Berlin}.
%\bnote{With a preface by Charles-Michel Marle}.
\bid{mr={1616513}}
\end{bbook}
%
\iffalse\OrigBibText
G. Gagneux, M. Madaune-Tort, Analyse math\'ematique
de mod\`{e}les non lin\'eaires de l'ing\'enierie p\'etroli\`{e}re,
Springer-Verlag, 1996.
\endOrigBibText\fi
\bptok{imsref}%
% NOT OUTPUTTED:
%   isbn = 3-540-60588-6
%   fpage = xviii+187
\endbibitem

%b13 ###
%b13 #&#
\bibitem{hof}
\begin{barticle}[mr]
\bauthor{\bsnm{Hofmanov{\'a}},~\bfnm{Martina}\binits{M.}}
(\byear{2013}).
\btitle{Degenerate parabolic stochastic partial differential equations}.
\bjournal{Stochastic Process. Appl.}
\bvolume{123}
\bpages{4294--4336}.
\bid{doi={10.1016/j.spa.2013.06.015}, issn={0304-4149}, mr={3096355}}
\end{barticle}
%
\iffalse\OrigBibText
M. Hofmanov\'a, Degenerate parabolic stochastic partial
differential equations, Stoch. Pr. Appl. 123 (12) (2013) 4294-4336.
\endOrigBibText\fi
\bptok{imsref}%
% NOT OUTPUTTED:
%   number = 12
%   doi = http://dx.doi.org/10.1016/j.spa.2013.06.015
%   fjournal = Stochastic Processes and their Applications
\endbibitem

%b14 ###
%b14 #&#
\bibitem{hof2}
\begin{barticle}[mr]
\bauthor{\bsnm{Hofmanov{\'a}},~\bfnm{Martina}\binits{M.}}
(\byear{2013}).
\btitle{Strong solutions of semilinear stochastic partial differential equations}.
\bjournal{NoDEA Nonlinear Differential Equations Appl.}
\bvolume{20}
\bpages{757--778}.
\bid{doi={10.1007/s00030-012-0178-x}, issn={1021-9722}, mr={3057153}}
\end{barticle}
%
\iffalse\OrigBibText
M. Hofmanov\'a, Strong solutions of semilinear
stochastic partial differential equations, NoDEA Nonlinear Differential
Equations Appl. 20 (3) (2013) 757-778.
\endOrigBibText\fi
\bptok{imsref}%
% NOT OUTPUTTED:
%   number = 3
%   doi = http://dx.doi.org/10.1007/s00030-012-0178-x
%   fjournal = NoDEA. Nonlinear Differential Equations and Applications
\endbibitem

%b15 ###
%b15 #&#
\bibitem{hofse}
\begin{barticle}[mr]
\bauthor{\bsnm{Hofmanov{\'a}},~\bfnm{Martina}\binits{M.}} \AND
\bauthor{\bsnm{Seidler},~\bfnm{Jan}\binits{J.}}
(\byear{2012}).
\btitle{On weak solutions of stochastic differential equations}.
\bjournal{Stoch. Anal. Appl.}
\bvolume{30}
\bpages{100--121}.
\bid{doi={10.1080/07362994.2012.628916}, issn={0736-2994}, mr={2870529}}
\end{barticle}
%
\iffalse\OrigBibText
M. Hofmanov\'a, J. Seidler, On weak solutions of
stochastic differential equations, Stoch. Anal. Appl. 30 (1) (2012) 100-121.
\endOrigBibText\fi
\bptok{imsref}%
% NOT OUTPUTTED:
%   number = 1
%   doi = http://dx.doi.org/10.1080/07362994.2012.628916
%   fjournal = Stochastic Analysis and Applications
\endbibitem

%b16 ###
%b16 #&#
\bibitem{vov}
\begin{barticle}[mr]
\bauthor{\bsnm{Imbert},~\bfnm{C.}\binits{C.}} \AND
\bauthor{\bsnm{Vovelle},~\bfnm{J.}\binits{J.}}
(\byear{2004}).
\btitle{A kinetic formulation for multidimensional scalar conservation laws with boundary conditions and applications}.
\bjournal{SIAM J. Math. Anal.}
\bvolume{36}
\bpages{214--232 (electronic)}.
\bid{doi={10.1137/S003614100342468X}, issn={0036-1410}, mr={2083859}}
\bptnote{check pages}%
\end{barticle}
%
\iffalse\OrigBibText
C. Imbert, J. Vovelle, A kinetic formulation for
multidimensional scalar conservation laws with boundary conditions and
applications, SIAM J. Math. Anal. 36 (2004), no. 1, 214-232.
\endOrigBibText\fi
\bptok{imsref}%
% NOT OUTPUTTED:
%   number = 1
%   doi = http://dx.doi.org/10.1137/S003614100342468X
%   fjournal = SIAM Journal on Mathematical Analysis
\endbibitem

%b17 ###
%b17 #&#
\bibitem{kim}
\begin{barticle}[mr]
\bauthor{\bsnm{Kim},~\bfnm{Jong~Uhn}\binits{J.~U.}}
(\byear{2003}).
\btitle{On a stochastic scalar conservation law}.
\bjournal{Indiana Univ. Math. J.}
\bvolume{52}
\bpages{227--256}.
\bid{doi={10.1512/iumj.2003.52.2310}, issn={0022-2518}, mr={1970028}}
\end{barticle}
%
\iffalse\OrigBibText
J. U. Kim, On a stochastic scalar conservation law,
Indiana Univ. Math. J. 52 (1) (2003) 227-256.
\endOrigBibText\fi
\bptok{imsref}%
% NOT OUTPUTTED:
%   number = 1
%   doi = http://dx.doi.org/10.1512/iumj.2003.52.2310
%   coden = IUMJAB
%   fjournal = Indiana University Mathematics Journal
\endbibitem

%b18 ###
%b18 #&#
\bibitem{kruzk}
\begin{barticle}[mr]
\bauthor{\bsnm{Kru{\v{z}}kov},~\bfnm{S.~N.}\binits{S.~N.}}
(\byear{1970}).
\btitle{First order quasilinear equations with several independent
variables}.
\bjournal{Mat. Sb. (N.S.)}
\bvolume{81}
\bpages{228--255}.
\bid{mr={0267257}}
\end{barticle}
%
\iffalse\OrigBibText
S. N. Kru\v{z}kov, First order quasilinear equations
with several independent variables, Mat. Sb. (N.S.) 81 (123) (1970) 228-255.
\endOrigBibText\fi
\bptok{imsref}%
\endbibitem

%b19 ###
%b19 #&#
\bibitem{lady}
\begin{bbook}[mr]
\bauthor{\bsnm{Lady{\v{z}}enskaja},~\bfnm{O.~A.}\binits{O.~A.}},
\bauthor{\bsnm{Solonnikov},~\bfnm{V.~A.}\binits{V.~A.}} \AND
\bauthor{\bsnm{Ural'ceva},~\bfnm{N.~N.}\binits{N.~N.}}
(\byear{1968}).
\btitle{Linear and Quasilinear Equations of Parabolic Type}.
\bseries{Translations of Mathematical Monographs}
\bvolume{23}.
\bpublisher{Amer. Math. Soc.},
\blocation{Providence, RI}.
\bid{mr={0241822}}
\end{bbook}
%
\iffalse\OrigBibText
O. A. Ladyzhenskaya, V. A. Solonnikov, N. N.
Ural'ceva, Linear and Quasilinear Equations of Parabolic Type,
Translations of Mathematical Monographs 23, Am. Math. Soc., Providence,
R. I. (1968).
\endOrigBibText\fi
\bptok{imsref}%
% NOT OUTPUTTED:
%   fpage = xi+648
\endbibitem

%b20 ###
%b20 #&#
\bibitem{LionsPerthameSouganidis12}
\begin{bmisc}[auto:parserefs-M02]
\bauthor{\bsnm{Lions},~\bfnm{P.~L.}\binits{P.~L.}},
\bauthor{\bsnm{Perthame},~\bfnm{B.}\binits{B.}} \AND
\bauthor{\bsnm{Souganidis},~\bfnm{P.~E.}\binits{P.~E.}}
(\byear{2011/12}).
\bhowpublished{Stochastic averaging lemmas for kinetic equations.
In \textit{Seminaire Equations Aux Derivee Partielles (Ecole Polytechnique)}.
Available at \arxivurl{arXiv:1204.0317}.}
\end{bmisc}
%
\iffalse\OrigBibText
P. L. Lions, B. Perthame, P. E.
Souganidis, Stochastic averaging lemmas for kinetic equations, arXiv:1204.0317
\endOrigBibText\fi
\bptok{imsref}%
\endbibitem




%b21 ###
%b21 #&#
\bibitem{LionsPerthameSouganidis13}
\begin{barticle}[auto:parserefs-M02]
\bauthor{\bsnm{Lions},~\bfnm{P.~L.}\binits{P.~L.}},
\bauthor{\bsnm{Perthame},~\bfnm{B.}\binits{B.}} \AND
\bauthor{\bsnm{Souganidis},~\bfnm{P.~E.}\binits{P.~E.}}
(\byear{2013}).
\btitle{Scalar conservation laws with rough (stochastic) fluxes}.
\bjournal{Stochastic Partial Differential Equations: Analysis and Computations}
\bvolume{1}
\bpages{664--686}.
\end{barticle}
%
\iffalse\OrigBibText
P. L. Lions, B. Perthame, P. E.
Souganidis, Scalar conservation laws with rough (stochastic) fluxes,
arXiv:1309.1931
\endOrigBibText\fi
\bptok{imsref}%
\endbibitem


%b22 ###
%b22 #&#
\bibitem{lpt1}
\begin{barticle}[mr]
\bauthor{\bsnm{Lions},~\bfnm{Pierre-Louis}\binits{P.-L.}},
\bauthor{\bsnm{Perthame},~\bfnm{Beno{\^{\i}}t}\binits{B.}} \AND
\bauthor{\bsnm{Tadmor},~\bfnm{Eitan}\binits{E.}}
(\byear{1991}).
\btitle{Formulation cin\'etique des lois de conservation scalaires multidimensionnelles}.
\bjournal{C. R. Acad. Sci. Paris S\'er. I Math.}
\bvolume{312}
\bpages{97--102}.
\bid{issn={0764-4442}, mr={1086510}}
\bptnote{check volume}%
\end{barticle}
%
\iffalse\OrigBibText
P. L. Lions, B. Perthame, E. Tadmor, Formulation cin\'etique des lois de conservation scalaires multidimensionnelles,
C.R. Acad. Sci. Paris (1991) 97-102, S\'erie I.
\endOrigBibText\fi
\bptok{imsref}%
% NOT OUTPUTTED:
%   number = 1
%   coden = CASMEI
%   fjournal = Comptes Rendus de l'Acad\'emie des Sciences. S\'erie I. Math\'ematique
\endbibitem

%b23 ###
%b23 #&#
\bibitem{lions}
\begin{barticle}[mr]
\bauthor{\bsnm{Lions},~\bfnm{P.-L.}\binits{P.-L.}},
\bauthor{\bsnm{Perthame},~\bfnm{B.}\binits{B.}} \AND
\bauthor{\bsnm{Tadmor},~\bfnm{E.}\binits{E.}}
(\byear{1994}).
\btitle{A kinetic formulation of multidimensional scalar conservation laws and related equations}.
\bjournal{J. Amer. Math. Soc.}
\bvolume{7}
\bpages{169--191}.
\bid{doi={10.2307/2152725}, issn={0894-0347}, mr={1201239}}
\end{barticle}
%
\iffalse\OrigBibText
P. L. Lions, B. Perthame, E. Tadmor, A kinetic
formulation of multidimensional scalar conservation laws and related
equations, J. Amer. Math. Soc. 7 (1) (1994) 169-191.
\endOrigBibText\fi
\bptok{imsref}%
% NOT OUTPUTTED:
%   number = 1
%   doi = http://dx.doi.org/10.2307/2152725
%   fjournal = Journal of the American Mathematical Society
\endbibitem

%b24 ###
%b24 #&#
\bibitem{malek}
\begin{bbook}[mr]
\bauthor{\bsnm{M{\'a}lek},~\bfnm{J.}\binits{J.}},
\bauthor{\bsnm{Ne{\v{c}}as},~\bfnm{J.}\binits{J.}},
\bauthor{\bsnm{Rokyta},~\bfnm{M.}\binits{M.}} \AND
\bauthor{\bsnm{R{$\stackrel{\raisebox{-2pt}{\mbox{\fontsize{5pt}{7pt}\selectfont{$\circ$}}}}{\mbox{u}}$}{\v{z}}i{\v{c}}ka},~\bfnm{M.}\binits{M.}}
(\byear{1996}).
\btitle{Weak and Measure-Valued Solutions to Evolutionary {PDE}s}.
\bseries{Applied Mathematics and Mathematical Computation}
\bvolume{13}.
\bpublisher{Chapman \& Hall},
\blocation{London}.
\bid{doi={10.1007/978-1-4899-6824-1}, mr={1409366}}
\end{bbook}
%
\iffalse\OrigBibText
J. M\'alek, J. Ne\v{c}as, M. Rokyta, M. R\r{u}\v{z}i\v{c}ka,
 Weak and Measure-valued Solutions to Evolutionary PDEs,
Chapman \& Hall, London, Weinheim, New York, 1996.
\endOrigBibText\fi
\bptok{imsref}%
% NOT OUTPUTTED:
%   doi = http://dx.doi.org/10.1007/978-1-4899-6824-1
%   isbn = 0-412-57750-X
%   fpage = xii+317
\endbibitem

%b25 ###
%b25 #&#
\bibitem{on2}
\begin{barticle}[mr]
\bauthor{\bsnm{Ondrej{\'a}t},~\bfnm{Martin}\binits{M.}}
(\byear{2010}).
\btitle{Stochastic nonlinear wave equations in local {S}obolev spaces}.
\bjournal{Electron. J. Probab.}
\bvolume{15}
\bpages{1041--1091}.
\bid{doi={10.1214/EJP.v15-789}, issn={1083-6489}, mr={2659757}}
\bptnote{check pages}%
\end{barticle}
%
\iffalse\OrigBibText
M. Ondrej\'at, Stochastic nonlinear wave equations in
local Sobolev spaces, Electronic Journal of Probability 15 (33) (2010)
1041--1091.
\endOrigBibText\fi
\bptok{imsref}%
% NOT OUTPUTTED:
%   doi = http://dx.doi.org/10.1214/EJP.v15-789
%   fjournal = Electronic Journal of Probability
\endbibitem

%b26 ###
%b26 #&#
\bibitem{pert1}
\begin{barticle}[mr]
\bauthor{\bsnm{Perthame},~\bfnm{B.}\binits{B.}}
(\byear{1998}).
\btitle{Uniqueness and error estimates in first order quasilinear conservation laws via the kinetic entropy defect measure}.
\bjournal{J. Math. Pures Appl.}
\bvolume{77}
\bpages{1055--1064}.
\bid{doi={10.1016/S0021-7824(99)80003-8}, issn={0021-7824}, mr={1661021}}
\end{barticle}
%
\iffalse\OrigBibText
B. Perthame, Uniqueness and error estimates in first
order quasilinear conservation laws via
the kinetic entropy defect measure, J. Math. Pures et Appl. 77 (1998) 1055-1064.
\endOrigBibText\fi
\bptok{imsref}%
% NOT OUTPUTTED:
%   number = 10
%   doi = http://dx.doi.org/10.1016/S0021-7824(99)80003-8
%   coden = JMPAAM
%   fjournal = Journal de Math\'ematiques Pures et Appliqu\'ees. Neuvi\`eme S\'erie
\endbibitem

%b27 ###
%b27 #&#
\bibitem{perth}
\begin{bbook}[mr]
\bauthor{\bsnm{Perthame},~\bfnm{Beno{\^{\i}}t}\binits{B.}}
(\byear{2002}).
\btitle{Kinetic Formulation of Conservation Laws}.
\bseries{Oxford Lecture Series in Mathematics and Its Applications}
\bvolume{21}.
\bpublisher{Oxford Univ. Press},
\blocation{Oxford}.
\bid{mr={2064166}}
\end{bbook}
%
\iffalse\OrigBibText
B. Perthame, Kinetic Formulation of Conservation Laws,
Oxford Lecture Ser. Math. Appl., vol. 21, Oxford University Press,
Oxford, 2002.
\endOrigBibText\fi
\bptok{imsref}%
% NOT OUTPUTTED:
%   isbn = 0-19-850913-8
%   fpage = xii+198
\endbibitem

%b28 ###
%b28 #&#
\bibitem{wittbold}
\begin{barticle}[mr]
\bauthor{\bsnm{Vallet},~\bfnm{Guy}\binits{G.}} \AND
\bauthor{\bsnm{Wittbold},~\bfnm{Petra}\binits{P.}}
(\byear{2009}).
\btitle{On a stochastic first-order hyperbolic equation in a bounded domain}.
\bjournal{Infin. Dimens. Anal. Quantum Probab. Relat. Top.}
\bvolume{12}
\bpages{613--651}.
\bid{doi={10.1142/S0219025709003872}, issn={0219-0257}, mr={2590159}}
\end{barticle}
%
\iffalse\OrigBibText
G. Vallet, P. Wittbold, On a stochastic first order
hyperbolic equation in a bounded domain, Infin. Dimens. Anal. Quantum
Probab. Relat. Top. 12 (4) (2009) 613-651.
\endOrigBibText\fi
\bptok{imsref}%
% NOT OUTPUTTED:
%   number = 4
%   doi = http://dx.doi.org/10.1142/S0219025709003872
%   fjournal = Infinite Dimensional Analysis, Quantum Probability and Related Topics
\endbibitem
\end{thebibliography}
\end{document}